\documentclass[
11pt,  reqno]{amsart}
\pdfoutput=1

\usepackage[margin=1.1in,marginparwidth=1.5cm, marginparsep=0.5cm]{geometry}

\usepackage{booktabs} 
\usepackage{microtype}
\usepackage{amssymb}
\usepackage{mathrsfs}

\usepackage{color}
\usepackage[implicit=true]{hyperref}

\usepackage{xsavebox}

\usepackage{cases}

\allowdisplaybreaks[2]

\definecolor{gr}{rgb}   {0.,   0.69,   0.23 }
\definecolor{bl}{rgb}   {0.,   0.5,   1. }
\definecolor{mg}{rgb}   {0.85,  0.,    0.85}
\definecolor{yl}{rgb}   {0.8,  0.7,   0.}
\definecolor{or}{rgb}  {0.7,0.2,0.2}

\usepackage{tikz} 
\usepackage{marginnote}
\usepackage{scalerel} 

\usetikzlibrary{shapes.misc}
\usetikzlibrary{shapes.symbols}
\usetikzlibrary{decorations}
\usetikzlibrary{decorations.markings}

\tikzset{
	dot/.style={circle,fill=black,draw=black,inner sep=0pt,minimum size=0.5mm},
	>=stealth,
	}

\tikzset{
	dot2/.style={circle,fill=black,draw=black,inner sep=0pt,minimum size=0.2mm},
	>=stealth,
	}

\tikzset{
	ddot/.style={circle,fill=white,draw=black,inner sep=0pt,minimum size=0.8mm},
	>=stealth,
	}

\tikzset{decision/.style={ 
        draw,
        diamond,
        aspect=1.5
    }}

\tikzset{dia2/.style
={diamond,fill=white,draw=black,inner sep=0pt,minimum size=1mm},
	>=stealth,
	}

\tikzset{dia/.style
={star,fill=black,draw=black,inner sep=0pt,minimum size=1mm},
	>=stealth,
	}

\tikzset{dia/.style
={diamond,fill=black,draw=black,inner sep=0pt,minimum size=1.3mm},
	>=stealth,
	}

\makeatletter
\def\DeclareSymbol#1#2#3{\xsavebox{#1}{\tikz[baseline=#2,scale=0.15]{#3}}}
\def\<#1>{\xusebox{#1}}
\makeatother

\newcommand{\pe}{\mathbin{\scaleobj{0.7}{\tikz \draw (0,0) node[shape=circle,draw,inner sep=0pt,minimum size=8.5pt] {\scriptsize  $=$};}}}

\newcommand{\pl}{\mathbin{\scaleobj{0.7}{\tikz \draw (0,0) node[shape=circle,draw,inner sep=0pt,minimum size=8.5pt] {\scriptsize $<$};}}}
\newcommand{\pg}{\mathbin{\scaleobj{0.7}{\tikz \draw (0,0) node[shape=circle,draw,inner sep=0pt,minimum size=8.5pt] {\scriptsize $>$};}}}

\newcommand{\pge}{\mathbin{\scaleobj{0.7}{\tikz \draw (0,0) node[shape=circle,draw,inner sep=0pt,minimum size=8.5pt] {\scriptsize $\geqslant$};}}}
\newcommand{\pez}{\mathbin{\scaleobj{0.7}{\tikz \draw (0,0) node[shape=circle,draw,
fill=white, 
inner sep=0pt,minimum size=8.5pt]{} ;}}}

\usetikzlibrary{decorations.pathmorphing, patterns,shapes}

\DeclareSymbol{dot}{-2.7}{\draw (0,0) node[dot]{};}

\DeclareSymbol{1}{0}{\draw[white] (-.4,0) -- (.4,0); \draw (0,0)  -- (0,1.2) node[dot] {};}
\DeclareSymbol{2}{0}{\draw (-0.5,1.2) node[dot] {} -- (0,0) -- (0.5,1.2) node[dot] {};}
\DeclareSymbol{3}{0}{\draw (0,0) -- (0,1.2) node[dot] {}; \draw (-.7,1) node[dot] {} -- (0,0) -- (.7,1) node[dot] {};}
\DeclareSymbol{30}{-3}{\draw (0,0) -- (0,-1); \draw (0,0) -- (0,1.2) node[dot] {}; \draw (-.7,1) node[dot] {} -- (0,0) -- (.7,1) node[dot] {};}
\DeclareSymbol{31}{-3}{\draw (0,0) -- (0,-1) -- (1,0) node[dot] {}; \draw (0,0) -- (0,1.2) node[dot] {}; \draw (-.7,1) node[dot] {} -- (0,0) -- (.7,1) node[dot] {};}
\DeclareSymbol{32}{-3}{\draw (0,0) -- (0,-1) -- (1,0) node[dot] {}; \draw (0,0) -- (0,-1) -- (-1,0) node[dot] {}; \draw (0,0) -- (0,1.2) node[dot] {}; \draw (-.7,1) node[dot] {} -- (0,0) -- (.7,1) node[dot] {};}
\DeclareSymbol{20}{-3}{\draw (0,0) -- (0,-1);\draw (-.7,1) node[dot] {} -- (0,0) -- (.7,1) node[dot] {};}
\DeclareSymbol{22}{-3}{\draw (0,0.3) -- (0,-1) -- (1,0) node[dot] {}; \draw (0,0.3) -- (0,-1) -- (-1,0) node[dot] {};\draw (-.7,1) node[dot] {} -- (0,0.3) -- (.7,1) node[dot] {};}
\DeclareSymbol{31p}{-3}{\draw (0,0) -- (0,-1) -- (1,0) node[dot] {}; \draw (0,0) -- (0,1.2) node[dot] {}; \draw (-.7,1) node[dot] {} -- (0,0) -- (.7,1) node[dot] {}; 
\draw (0,-1) node{\scaleobj{0.5}{\pez}}; 
\draw (0,-1) node{\scaleobj{0.5}{\pe}}; }
\DeclareSymbol{32p}{-3}{\draw (0,0) -- (0,-1) -- (1,0) node[dot] {}; \draw (0,0) -- (0,-1) -- (-1,0) node[dot] {}; \draw (0,0) -- (0,1.2) node[dot] {}; \draw (-.7,1) node[dot] {} -- (0,0) -- (.7,1) node[dot] {}; 
\draw (0,-1) node{\scaleobj{0.5}{\pez}};
\draw (0,-1) node{\scaleobj{0.5}{\pe}};}
\DeclareSymbol{22p}{-3}{\draw (0,0.3) -- (0,-1) -- (1,0) node[dot] {}; \draw (0,0.3) -- (0,-1) -- (-1,0) node[dot] {};\draw (-.7,1) node[dot] {} -- (0,0.3) -- (.7,1) node[dot] {}; 
\draw (0,-1) node{\scaleobj{0.5}{\pez}};
\draw (0,-1) node{\scaleobj{0.5}{\pe}};}
\DeclareSymbol{21p}{-3}{\draw (0,0.3) -- (0,-1) -- (1,0) node[dot] {}; 
\draw (-.7,1) node[dot] {} -- (0,0.3) -- (.7,1) node[dot] {}; 
\draw (0,-1) node{\scaleobj{0.5}{\pez}};
\draw (0,-1) node{\scaleobj{0.5}{\pe}};}

\DeclareSymbol{11p}{-3}{\draw (0,0) node[dot2] {} -- (0,-1) -- (1,0) node[dot] {}; 
\draw (0,0) -- (0,1.2) node[dot] {};  
\draw (0,-1) node{\scaleobj{0.5}{\pez}}; 
\draw (0,-1) node{\scaleobj{0.5}{\pe}}; }
\DeclareSymbol{100}{-3}
{\draw[white] (-.4,0) -- (.4,0); 
\draw (0,0) node[dot2] {} -- (0,-1)  ; 
\draw (0,0) -- (0,1.2) node[dot] {};}  
\usetikzlibrary{mindmap,backgrounds,trees,arrows,decorations.pathmorphing,decorations.pathreplacing,positioning,shapes.geometric,shapes.misc,decorations.markings,decorations.fractals,calc,patterns}

\tikzset{>=stealth',
         cvertex/.style={circle,draw=black,inner sep=1pt,outer sep=3pt},
         vertex/.style={circle,fill=black,inner sep=1pt,outer sep=3pt},
         star/.style={circle,fill=yellow,inner sep=0.75pt,outer sep=0.75pt},
         tvertex/.style={inner sep=1pt,font=\scriptsize},
         gap/.style={inner sep=0.5pt,fill=white}}
\tikzstyle{mybox} = [draw=black, fill=blue!10, very thick,
    rectangle, rounded corners, inner sep=10pt, inner ysep=20pt]
\tikzstyle{boxtitle} =[fill=blue!50, text=white,rectangle,rounded corners]

\tikzstyle{decision} = [diamond, draw, fill=blue!20,
    text width=4.5em, text badly centered, node distance=2.5cm, inner sep=0pt]
\tikzstyle{block} = [rectangle, draw, fill=blue!20,
    text width=2.7cm, text centered, rounded corners, minimum height=3em]

\tikzstyle{line} = [draw, very thick, color=black!50, -latex']

\tikzstyle{cloud} = [draw, ellipse,fill=red!40, 
node distance=2.5cm,
    minimum height=1em]

\tikzstyle{cloud2} = [draw, ellipse,fill=red!30, text=white,text width=10em, node distance=2.5cm, text centered, minimum height=4em]

\tikzstyle{cloud3} = [draw, ellipse, fill=cyan!30, 
node distance=2.5cm,
    minimum height=3em, 
    text width=1.7cm]

\tikzstyle{cloud4} = [draw, ellipse,fill=orange!70, node distance=2.5cm,
   text width=2.1cm,  
           minimum height=3em]

\tikzstyle{cloud5} = [draw, ellipse,fill=red!20, node distance=2.5cm,
    text centered, 
    text width=3.0cm, 
    minimum height=3em]

\tikzstyle{cloud6} = [draw, ellipse,fill=red!20, node distance=2.5cm,
    text width=1.5cm, 
    text centered, 
    minimum height=3em]

\newcommand{\arrow}[2][20]
 {
  \hspace{-5pt}
  \begin{tikzpicture}
   \node (A) at (0,0) {};
   \node (B) at (#1pt,0) {};
   \draw [#2] (A) -- (B);
  \end{tikzpicture}
  \hspace{-5pt}
 }

\tikzset{
    position/.style args={#1:#2 from #3}{
        at=(#3.#1), anchor=#1+180, shift=(#1:#2)
    }
}

\newtheorem{theorem}{Theorem} [section]

\newtheorem{lemma}[theorem]{Lemma}
\newtheorem{proposition}[theorem]{Proposition}
\newtheorem{remark}[theorem]{Remark}

\newtheorem{corollary}[theorem]{Corollary}

\DeclareMathOperator*{\supp}{supp}

\newcommand{\1}{\hspace{0.2mm}\text{I}\hspace{0.2mm}}
\newcommand{\II}{\text{I \hspace{-2.8mm} I} }

\newcommand{\noi}{\noindent}
\newcommand{\Z}{\mathbb{Z}}
\newcommand{\R}{\mathbb{R}}
\newcommand{\T}{\mathbb{T}}

\let\Re=\undefined\DeclareMathOperator*{\Re}{Re}
\let\Im=\undefined\DeclareMathOperator*{\Im}{Im}

\DeclareMathOperator{\com}{\textsf{com}}

\let\P= \undefined
\newcommand{\P}{\mathbf{P}}

\newcommand{\E}{\mathbb{E}}

\renewcommand{\L}{\mathcal{L}}

\newcommand{\F}{\mathcal{F}}

\newcommand{\al}{\alpha}
\newcommand{\be}{\beta}
\newcommand{\dl}{\delta}

\newcommand{\nb}{\nabla}

\newcommand{\Dl}{\Delta}
\newcommand{\eps}{\varepsilon}
\newcommand{\kk}{\kappa}
\newcommand{\g}{\gamma}

\newcommand{\ld}{\lambda}
\newcommand{\Ld}{\Lambda}
\newcommand{\s}{\sigma}
\newcommand{\Si}{\Sigma}
\newcommand{\ft}{\widehat}

\newcommand{\wt}{\widetilde}
\newcommand{\cj}{\overline}

\newcommand{\dt}{\partial_t}

\newcommand{\ta}{\theta}

\renewcommand{\l}{\ell}
\renewcommand{\o}{\omega}
\renewcommand{\O}{\Omega}

\newcommand{\les}{\lesssim}
\newcommand{\ges}{\gtrsim}

\newcommand{\jb}[1]
{\langle #1 \rangle}

\newcommand{\ind}{\mathbf 1}

\renewcommand{\S}{\mathcal{S}}

\newcommand{\N}{\mathbb{N}}

\renewcommand{\H}{\mathcal{H}}

\newtheorem*{ackno}{Acknowledgements}

\newcommand{\I}{\mathcal{I}}

\newcommand{\If}{\mathfrak{I}}

\newcommand{\RR}{\mathcal{R}}

\newcommand{\NR}{\textit{NR}}

\newcommand{\A}{\mathcal{A}}

\numberwithin{equation}{section}
\numberwithin{theorem}{section}

\begin{document}
\baselineskip = 14pt

\title[Paracontrolled approach to 3-$d$ stochastic  NLW]
{Paracontrolled approach to  the\\ three-dimensional stochastic nonlinear wave equation\\
with quadratic nonlinearity}

\author[M.~Gubinelli, H.~Koch, and T.~Oh]
{Massimiliano Gubinelli, Herbert Koch, and Tadahiro Oh}

\address{
Massimiliano Gubinelli\\
Hausdorff Center for Mathematics \&  Institut f\"ur Angewandte Mathematik\\
 Universit\"at Bonn\\
Endenicher Allee 60\\
D-53115 Bonn\\
Germany}
\email{gubinelli@iam.uni-bonn.de}

\address{Herbert Koch\\
Mathematisches Institut\\
 Universit\"at Bonn\\
Endenicher Allee 60\\
D-53115 Bonn\\
Germany
}

\email{koch@math.uni-bonn.de}

\address{
Tadahiro Oh, School of Mathematics\\
The University of Edinburgh\\
and The Maxwell Institute for the Mathematical Sciences\\
James Clerk Maxwell Building\\
The King's Buildings\\
Peter Guthrie Tait Road\\
Edinburgh\\ 
EH9 3FD\\
 United Kingdom}

\email{hiro.oh@ed.ac.uk}



\begin{abstract}
Using  ideas from  paracontrolled calculus, 
we prove local well-posedness
of a renormalized version of the
 three-dimensional stochastic nonlinear wave equation 
with  quadratic nonlinearity forced by 
an additive space-time white noise on a periodic domain.
There are two new ingredients as compared to the parabolic setting.
(i) In constructing stochastic objects, 
we have to carefully exploit dispersion at a multilinear level.
(ii) We introduce novel random operators and leverage their regularity to overcome 
the lack of smoothing of usual paradifferential commutators.

\end{abstract}

%
\maketitle
\tableofcontents

\newpage

\section{Introduction}
\label{SEC:1}

\subsection{Singular stochastic nonlinear wave equations}

We continue the study of singular stochastic  nonlinear wave equations (SNLW) driven by additive 
space-time white noise initiated in~\cite{GKO}. 
There we studied  the case of the SNLW equation with a polynomial nonlinearity on 
the two-dimensional torus $\T^2=(\R/2\pi \Z)^2$.
By  introducing a suitable renormalization of the nonlinearity, 
we proved a local-in-time existence and uniqueness theory. 
Global solutions on $\T^2$ have been obtained in~\cite{GKOT} for the defocusing cubic nonlinearity. 
See also \cite{Tolomeo}
for an analogous global well-posedness result on the Euclidean space $\R^2$.
Here,  we consider SNLW on the three-dimensional torus $\T^3=(\R/2\pi \Z)^3$ starting with the case of quadratic 
nonlinearity. Our aim is to provide a local well-posedness theory for the equation which formally reads
\begin{align}
\begin{cases}
\dt^2 u  + (1 -  \Dl)  u =  -u^2 +\infty + \xi\\
(u, \dt u) |_{t = 0} = (u_0, u_1) \in \H^s(\T^3), 
\end{cases}
\quad (x, t) \in \T^3\times \R_+,
\label{SNLW1}
\end{align}

\noi
where  $\H^s(\T^3) = H^s(\T^3)\times H^{s-1}(\T^3)$
 and $\xi(x, t)$ denotes a (Gaussian) space-time white noise on $\T^3\times \R_+$
with the space-time covariance given by
\[ \E\big[ \xi(x_1, t_1) \xi(x_2, t_2) \big]
= \dl(x_1 - x_2) \dl (t_1 - t_2).\]

\noi
The expression $-u^2+\infty$ denotes the renormalization of the product $u^2$. 
As we will see below, indeed, solutions to this equation are expected to be distributions 
of (spatial) regularity below~$-\frac12$.

We state our main result.

\begin{theorem}\label{THM:2-intro}
Given $\frac 14 < s < \frac 12$, 
let $(u_0, u_1) \in \H^{s}(\T^3)$.
Given $N \in \N$, let $ \xi_N =  \pi_N \xi$, 
where $\pi_N$ is the frequency projector onto the spatial frequencies $\{|n|\leq N\}$
defined in~\eqref{pi} below.
Then,  there exists a sequence of time-dependent constants $\{\s_N(t)\}_{N\in \N}$
 tending to $\infty$ \textup{(}see \eqref{sigma1} below\textup{)} such that, given small $\eps = \eps(s) > 0$, 
the solution $u_N$ to the following renormalized SNLW:
\begin{align}
\begin{cases}
\dt^2 u_N + (1-  \Dl)  u_N  =  - u_N^2  + \s_N+ \xi_N\\
(u_N, \dt u_N)|_{t = 0} = (u_0, u_1)
\end{cases}
\label{SNLW10}
\end{align}
converges to a stochastic process $u \in C([0, T]; H^{-\frac 12 -\eps} (\T^3))$ almost surely,
where $T = T(\o)$ is an almost surely positive stopping time.
\end{theorem}

Furthermore, 
 we will provide a description of the limiting distribution $u$ in terms of the notion of 
 \emph{paracontrolled} distributions introduced in~\cite{GIP}. 

Let us comment on the need of the renormalized formulation~\eqref{SNLW10}. 
In the context of parabolic stochastic partial differential equations (SPDEs),  the need and meaning of renormalization of SPDEs have been intensely studied and much progress has been achieved
in recent years, starting with Da Prato and Debussche's strong solutions approach~\cite{DPD2} to the dynamical $\Phi^4_2$ model, continuing with Hairer's solution of the KPZ equation~\cite{Hairer0}, the subsequent invention of regularity structures~\cite{Hairer},
 and the discovery of alternative approaches such as paracontrolled distributions~\cite{GIP}, Kupiainen's 
 renormalization group  approach~\cite{Kupi, kupiainen_renormalization_2017},
and  the approach of Otto, Weber, and coauthors~\cite{otto_quasilinear_2016, bailleul_quasilinear_2016, otto_parabolic_2018}. 

On the one hand, the theory of regularity structures~\cite{Hairer, friz_course_2014, MR3728488} has since grown into a complete framework~\cite{bruned_algebraic_2016, chandra_analytic_2016, bruned_renormalising_2017, MR3855740} which can deal with a large class of parabolic equations (in the so-called \emph{subcritical} regime) 
such as  the dynamical sine-Gordon model~\cite{hairer_dynamical_2016, chandra_dynamical_2018}, the generalized KPZ equation used to describe a natural random evolution on the space of paths over a manifold~\cite{hairer_motion_2016}, and other interesting models like those related to Abelian gauge theories~\cite{shen_stochastic_2018} or some equations in the full space~\cite{MR3779690}.
On the other hand, the theory of paracontrolled distributions 
has revealed itself as an effective method for a restricted class of singular SPDEs~\cite{CC, GP1, allez_continuous_2015, zhu_three_dimensional_2015, 
hoshino_2016,
bailleul_higher_2016,  bailleul_quasilinear_2016, MW1, funaki_2017, hoshino_2017, hoshino_2017b, GH18, perkowski_kpz_2018, MR3785598}. Let us also mention that 
 certain quasilinear parabolic equations can 
  be considered using natural extension of these theories~\cite{otto_quasilinear_2016, furlan_paracontrolled_2016, gerencser_solution_2017}.

Singular SPDEs have been shown to describe large scale behavior of many random dynamical models, including both  continuous~\cite{HQ, MR3737909, MR3772400, shen_weak_2017, furlan_weak_2018, hairer_large_scale_2018} and discrete ones~\cite{MW2, cannizzaro_space_time_2016, shen_glauber_2016, martin_paracontrolled_2017, hairer_discretisations_2018, matetski_martingale_driven_2018}. 
This phenomenon has been named~\emph{weak universality}.

Renormalization can be, in first instance,  justified in order to obtain non-trivial (i.e.~nonlinear) limiting problems. 
At a deeper level, 
 singular PDEs and the need of their renormalization are tightly linked with the phenomenon of weak universality. 
 These equations are meant to describe the large scale fluctuations of well-behaved smooth random systems
  and, in this perspective, both the distributional nature of the solution and the renormalization have clear physical meanings; the irregularity of the solutions is the manifestation of the microscopic random fluctuations, while the renormalization is linked to the fine tuning of the parameters needed to allow for nonlinear fluctuations at the macroscopic level. While this discussion is quite informal and general,  this picture can be understood rigorously in many specific cases, at least in the parabolic setting. 

As far as wave equations are concerned, it has been observed in~\cite{russo2, russo4, russo1, russo3} that SNLW with regularized additive space-time noise converges to a linear equation as the regularization is removed, essentially independently of the kind of (Lipschitz) nonlinearity considered. This hints to the fact that 
wave equations also need a certain fine tuning of the parameters in order to exhibit singular nonlinear fluctuations. 

All the theories we mentioned above are, however,  designed to handle parabolic equations and 
it is not a priori clear how to adapt them to handle dispersive or hyperbolic phenomena. 

Schr\"odinger and wave equations in two and three dimensions with multiplicative spatial white noise have been considered with spectral methods in~\cite{debussche_schrodinger_2018, debussche_solution_2017, gubinelli_semilinear_2018}. The spatial nature of the noise allowed the authors
 to use techniques similar to the parabolic setting~\cite{allez_continuous_2015}.
In our  paper~\cite{GKO}, we gave the first example of (non-trivial) weak universality in wave equations by showing that the renormalized SNLW on $\mathbb{T}^2$ describes a particular large scale limit of a random nonlinear wave equation with smooth noise. 
There, it was shown how, despite the hyperbolic setting, renormalization proceeds in a way quite similar to the parabolic one. 

In the present paper, we will also show that, despite the fact that notions such as  \emph{homogeneity} 
(fundamental in the theory of regularity structures) or Besov-H\"older regularity 
(similarly fundamental in the theory of  paracontrolled distributions in the parabolic setting) are less compelling in the hyperbolic setting, we can set up a paracontrolled analysis of
the SNLW equation~\eqref{SNLW1} 
which takes into account multilinear dispersive regularization and renormalization of resonant  stochastic terms via the introduction of certain random operators,
 replacing the  commutators standard in the parabolic paracontrolled approach of~\cite{GIP}. Let us note that the control of certain random operators already appeared in the analysis of discrete approximations to 
 the KPZ equation in~\cite{GP1}.

As an application of our results, we can identify the solutions to the SNLW equation~\eqref{SNLW1} 
constructed in Theorem \ref{THM:2-intro} as the universal limit of a certain class of random wave equations. 
Consider the following stochastic nonlinear wave equation
on $(\kk^{-1} \T)^3 \times \R_+$:
\begin{align}
\begin{cases}
\partial_t^2 w_{\kk}  + (1 - \Dl) w_{\kk} = f (w_{\kk})
   + a_{\kk}^{(0)} + a_{\kk}^{(1)} w_{\kk} +
   {\kk}^2 \eta_{\kk} \\
(w_{\kk}, \dt w_{\kk})|_{t = 0} = (0, 0), 
\end{cases}
\label{weq}
\end{align}

\noi
 where $\kk > 0$, where
$f:\R\to\R$ is an arbitrary bounded $C^3$-function with bounded derivatives and $\eta_{\kk}$ is a Gaussian noise which
is white in time (for simplicity) and smooth in space with finite range translation-invariant
correlations (see \eqref{eta-kappa} below for a precise definition). 
Here, $a_{\kk}^{(0)}$ and $a_{\kk}^{(1)}$
are parameters to be chosen later.
It is not difficult to show that this equation has global smooth solutions.
We think of this equation to be a \textit{microscopic} model of a given
space-time random field $w_{\kk}$ living on a large spatial domain $(\kk^{-1} \T)^3$ and subject to a very small random driving force of order $\kk^2 \ll 1$. For technical reasons we prefer to work in a bounded domain but the reader should think that the equation is set up in the full space and that the parameter $\kk$ sets the size of the random perturbation.  In order to focus on the large scale / long time behavior of the solutions to this equation,
 we perform an hyperbolic rescaling of the independent variables $(x, t)$ and introduce a new random field $u_\kk$ given by 
\begin{align}
u_{\kk} ( x, t) = \kk^{-2} w_{\kk} (\kk^{-1} x, \kk^{-1} t),
\qquad (x, t) \in \T^3 \times \R_+.
\label{scale1}
\end{align}

\noi
The following theorem gives a precise description of the limiting behavior of $u_\kk$ as $\kk\to 0$ and as the parameters $a_{\kk}^{(0)},a_{\kk}^{(1)}$ are tuned in order to have
\[
 f (w_{\kk})
   + a_{\kk}^{(0)} + a_{\kk}^{(1)} w_{\kk} \simeq w_\kk^2
\]

\noi
implying that $w_\kk =0$ is a solution of the unperturbed dynamics. 
Accordingly, the initial data in~\eqref{weq} are set to zero in order not to interfere with the analysis of the long time effect of the random perturbation.

\begin{theorem}\label{THM:weak}
There exists a \textup{(}time-dependent\textup{)} choice of the coefficients $a_{\kk}^{(0)},a_{\kk}^{(1)}=O(1)$ and 
an almost surely positive random time $T$ 
such that 
 the random field $u_\kk$ defined in \eqref{scale1} converges in probability to a well defined limit $u$
 in $C([0,T];H^{-\frac12 -\eps}(\T^3))$  as $\kk\to 0$. The limiting random field $u$ is 
 \textup{(}modulo a possible rescaling\textup{)} 
 a local-in-time solution to the  renormalized quadratic SNLW~\eqref{SNLW1} with the zero initial data.
\end{theorem}

In fact, we will choose the coefficient $a_{\kk}^{(1)}$ depending only on $f$
and $\kk > 0$,
namely it is deterministic and independent of time.
See \eqref{WU5} below.

\begin{remark}\rm

The equation \eqref{SNLW1}
indeed corresponds
to the stochastic nonlinear Klein-Gordon equation.
The same results with  inessential modifications
also hold for
the stochastic nonlinear wave equation,
where we replace the left-hand side in \eqref{SNLW1}
by $\dt^2 u- \Dl u$.
In the following, we simply refer to \eqref{SNLW1} as 
the stochastic nonlinear wave equation.

\end{remark}

\subsection{The Da~Prato--Debussche trick}

Let us now describe the strategy which we used in~\cite{GKO,GKOT} to tackle the renormalization 
of the two-dimensional  SNLW equation:
 \begin{align*}
\dt^2 u  + (1 -  \Dl)  u =  -u^k  + \xi,  
\qquad (x, t) \in \T^2\times \R_+,
\end{align*}

\noi
for a generic monomial nonlinearity $u^k$. The first step is to introduce a new variable
\begin{align}
\label{decomp}
u = \Psi+v,
\end{align}

\noi
where $\Psi$ is the stochastic convolution given by 
\begin{align*}
\begin{split}
\Psi(t)
 := 
\I \xi (t)
& =  \int_{0}^t \frac{\sin ((t-t')\jb{\nb})}{\jb{\nb}}dW(t').
\end{split}
\end{align*}

\noi
Here, 
$W$ is a cylindrical Wiener process on  $L^2(\T^2)$, 
 and $\I = (\dt^2 +1 - \Dl)^{-1}$ is the Duhamel integral operator, 
corresponding to 
the forward fundamental solution to the linear wave equation,
 and  
$\jb{\nb}$ is the Fourier multiplier operator corresponding to the multiplier 
 $\jb{n} = (1 + |n|^2)^\frac{1}{2}$.
By a standard argument, 
it is easy to see that the stochastic convolution $\Psi$
almost surely has the regularity $ C(\R_+; W^{-\eps, \infty}(\T^2))$, $\eps > 0$.
Moreover,  it can be shown that for each $t> 0$, $\Psi (t)\notin L^2(\T^2)$ almost surely, thus creating  an issue in making sense of powers
$\Psi^k$ and a fortiori of the full nonlinearity $u^k$. The appropriate renormalization corresponds to replace the powers $\Psi(t)^k$ by the Wick powers $:\! \Psi(t)^k\!:$ of the stochastic convolution.
It then follows that the equation for the residual term $v = u - \Psi $ takes the form:
\begin{align}
(\dt^2 + 1 -  \Dl)  v =  -\sum_{\l=0}^k {k\choose \l} :\! \Psi^\l   \!: v^{k-\ell}. 
\label{SNLW2b}
\end{align}

\noi
By viewing 
$(u_0, u_1, \Psi, :\! \Psi^2  \!:, \dots, :\! \Psi^k  \!:\,)$
as a given {\it enhanced} data set, 
we studied the fixed point problem \eqref{SNLW2b} 
for $v$  via the Strichartz estimates\footnote{In fact, one may prove local well-posedness
of \eqref{SNLW2b} on $\T^2$ by Sobolev's inequality, i.e.~without  the Strichartz estimates.
See \cite{GKOT}.}
(see Lemma~\ref{LEM:Str} below) and  we proved that the renormalized SNLW on $\T^2$ is locally well-posed
for any integer  $k \geq 2$ and is globally well-posed when $k = 3$.
See also \cite{OT2, OOTz} for a related problem
on the deterministic (renormalized) NLW with random initial data.

\begin{remark}\label{REM:behavior1}\rm 
(i) 
In the field of stochastic parabolic  PDEs, 
the decomposition~\eqref{decomp} is usually referred to as the Da~Prato--Debussche trick \cite{DPD, DPD2}.
Note that such an idea also appears in 
McKean~\cite{McKean} and 
Bourgain \cite{BO96} in the context of (deterministic) dispersive PDEs with random initial data, preceding \cite{DPD}. See also Burq-Tzvetkov \cite{BT1}.

\smallskip

\noi
(ii) While $\Psi$ is not a function, it turns out that the residual part $v$ is a function 
of positive regularity.
Namely, the decomposition \eqref{decomp} shows that 
the solution $u$  ``behaves like'' the stochastic convolution
in the high-frequency regime (or equivalently on small scales).
\end{remark}

For our problem on the three-dimensional torus $\T^3$, 
 the Da~Prato--Debussche trick does not suffice.
Indeed,  the stochastic convolution $\Psi$ 
is less regular in three dimensions: $\Psi \in C(\R_+; W^{-\frac{1}{2}-\eps, \infty}(\T^3))$ 
almost surely for any $\eps > 0$.
See Lemma \ref{LEM:stoconv} below. 
This worse behavior also causes  the higher Wick powers $:\!\Psi^k\!:$ of $\Psi$ 
to become less and less regular. 
Correspondingly, the Cauchy problem with higher powers of the nonlinear term become more and more difficult to study. 
This is the reason that, in this paper, we limit ourselves to the first non-trivial situation,
namely the case $k=2$ which is already not amenable to be harnessed by the Da~Prato--Debussche trick. 
The main difficulty here is the lack of sufficient regularity for 
the residual term $v$ in order for the product $\Psi \cdot v$ to be well defined. 
In the next subsection, we will describe in detail the difficulty and the strategy  to overcome this difficulty.
In particular we will use ideas from the paracontrolled calculus introduced
by the first author (with Imkeller and Perkowski) \cite{GIP, CC}
and rewrite the equation \eqref{SNLW1} in an appropriate form,
 where the residual term $v$ is further decomposed and analyzed to expose 
 other multilinear stochastic objects of the stochastic convolution $\Psi$ which will be subsequently estimated via probabilistic methods (and via detailed analysis exploiting  their multilinear dispersive structures).\footnote{We also mention a recent preprint \cite{DNY1} by Deng, Nahmod, and Yue
 which appeared about one year after our current paper.  This remarkable work \cite{DNY1}
 elaborates the ideas 
 introduced by Gubinelli, Imkeller, and Perkowski~\cite{GIP} and
by 
  Bringmann \cite{Bringmann}, 
 where the basic objects are (frequency-localized) random {\it non-homogenous} linear solutions
 with random potentials, 
 thus incorporating the bad part of a nonlinearity, 
  and introduces the so-called random averaging operators, 
 which are well adapted to the dispersive setting.
See also a very recent preprint \cite{DNY2} by the same authors
which extends the method of random averaging operators in \cite{DNY1} 
and introduces the theory of random tensors.}

For further reference,
 let us now describe  the construction of $\Psi$ in the three-dimensional setting. 
 Let  $W$ denote a cylindrical Wiener process on $L^2(\T^3)$.\footnote{By convention, 
 we endow $\T^3$ with the normalized Lebesgue measure $(2\pi)^{-3} dx$.}
More precisely, by letting 
\begin{align}
e_n(x) = e^{ i n \cdot x}, \quad 
\Ld = \bigcup_{j=0}^{2} \Z^j\times \Z_{+}\times \{0\}^{2-j}, 
\quad \text{and}\quad \Ld_0 = \Ld \cup\{(0, 0, 0)\},
\label{index}
\end{align}

\noi
we have\footnote{Note that $\{e_0, \sqrt 2\cos( n\cdot x), 
\sqrt 2 \sin ( n \cdot x):
n \in \Ld \}$ forms an orthonormal basis of $L^2(\T^3)$ 
(endowed with the normalized Lebesgue measure) in the real-valued setting.}
\begin{align}
\begin{split}
 W(t) & =  \sum_{n\in \Z^3}
\be_n (t) e_n \\
& = \be_0(t) e_0 +  \sum_{n\in \Ld}
\Big[\sqrt 2 \Re(\be_n (t)) \cdot  \sqrt 2 \cos ( n \cdot x)\\
& \hphantom{XXXXX}
- \sqrt 2\Im(\be_n (t) )\cdot \sqrt 2 \sin ( n \cdot x)\Big],
\end{split}
\label{Wiener1}
\end{align}

\noi
where $\{ \be_n\}_{n \in \Ld_0}$ is a family of  mutually independent 
 complex-valued Brownian
motions\footnote{Here, we take $\be_0$ to be real-valued.
Moreover, we normalized $\be_n$ such that 
$\text{Var}(\be_n(t)) = t$. In particular, we have 
$\text{Var}(\Re \be_n(t)) = \text{Var}(\Im \be_n(t)) = \frac{t}{2}$
for $n \ne 0$.}
on a fixed probability space $(\O, \F, P)$
and $\be_{-n} := \cj{\be_n}$ for $n \in \Ld_0$.
It is easy to see that $W$ almost surely lies in 
$C^{\al}(\R_+;W^{-\frac{3}{2}-\eps, \infty}(\T^3))$
for any $\al < \frac 12$ and $\eps > 0$.
We then define the  stochastic convolution $\Psi$  
in the three-dimensional setting by 
\begin{align}
\begin{split}
\Psi(t) := 
\I \xi (t) = \sum_{n \in \mathbb{Z}^3} e_n 
   \int_0^t \frac{\sin ((t - t') \jb{ n })}{\jb{ n }} d \beta_n (t'). 
\end{split}
\label{sconv1}
\end{align}

\noi
See Lemma \ref{LEM:stoconv} below
for the regularity property of $\Psi$.

\subsection{The paracontrolled approach}
In the field of stochastic parabolic PDEs, 
there has been a significant progress
over the last five years.
In \cite{Hairer}, Hairer introduced the theory of regularity structures
and gave a precise meaning to 
certain (subcritical) singular stochastic parabolic PDEs,
which are classically ill-posed.
In particular, he showed that the stochastic quantization equation (SQE) on $\T^3$:
\begin{align}
\dt u -  \Dl u  = - u^3 + \infty \cdot u  + \xi
\label{SQE1}
\end{align}
\noi
is locally well-posed in an appropriate sense.

In \cite{CC}, Catellier and Chouk
proved an analogous local well-posedness result of SQE \eqref{SQE1} 
via the paracontrolled calculus approach of Imkeller, Perkowski,  and the first author~\cite{GIP}.
This result was extended to global well-posedness on the torus
(with a uniform-in-time bound) 
in a recent work \cite{MW1}
by Mourrat and Weber. 
More recently,  Hofmanov\'a and the first author~\cite{GH18} 
proved global  existence of unique solutions to \eqref{SQE2} on the Euclidean space $\R^3$.

In \cite{Hairer, CC}, 
the ``solution'' $u$ to \eqref{SQE1}
is constructed as a unique limit
of the following smoothed equation:
\begin{align}
\dt u_\dl -  \Dl u_\dl = - u_\dl^3 + C_\dl u_\dl  + \xi_\dl, 
\label{SQE2}
\end{align}

\noi
where $ \xi_\dl =  \rho_\dl * \xi$
denotes the smoothed noise
by a mollifier $\rho_\dl$.\footnote{In \cite{Hairer}, 
the mollifier $\rho_\dl$ is on both spatial and temporal variables, 
while it is only on spatial variables in~\cite{CC}.
In \cite{Kupi}, the author employs a different kind of regularization.}
Here, the uniqueness refers to the following;
while  the diverging constant $C_\dl$ depends
on the choice of the  mollifier $\rho_\dl$, 
the limit $u$ is independent of the choice of the mollifier.
As it is written, 
one may wonder if $u$ actually solves any equation in the end.
In fact, one can introduce a decomposition of $u$
analogous to~\eqref{decomp}
such that the residual terms satisfy a system
of PDEs in the pathwise sense.

In the following, we briefly describe this decomposition of $u$
in the paracontrolled setting.
For this purpose, let us define
the stochastic convolution 
$\<1>$  
by 
\begin{align*}
 \<1> = (\dt - \Dl)^{-1} \xi.
\end{align*}

\noi
Here,  we adopted Hairer's convention to denote the stochastic terms by trees;
the vertex ``\,$\<dot>$\,'' in $\<1>$ corresponds to the space-time white noise $\xi$,
while the edge denotes the Duhamel integral operator $(\dt - \Dl)^{-1}$.
On $\T^3$, $\<1>$ has spatial regularity\footnote{Hereafter, we use $a-$ 
(and $a+$) to denote $a- \eps$ (and $a+ \eps$, respectively)
for arbitrarily small $\eps > 0$.
If this notation appears in an estimate, 
then  an implicit constant 
is allowed to depend on $\eps> 0$ (and it usually diverges as $\eps \to 0$).
}
$-\frac 12 -$
and hence its powers do not make sense.
Denoting the renormalized cubic power ``$\<1>^3$''
by $\<3>$,\footnote{In the three-dimensional case, it is known that the ``renormalized'' cubic 
power $\<3>$ does not quite make sense as a distribution-valued function
of time due to a logarithmic divergence.
Note, however, that   $\<30>$ defined in \eqref{so2} is a well-defined function.}
we define
\begin{align}
  \<30> = (\dt - \Dl)^{-1} \<3>.
\label{so2}
\end{align}

\noi
Thanks to the parabolic smoothing of degree 2, 
it can be seen that $\<30>$ has the regularity $\frac 12- = 3(-\frac 12-) + 2$.
See for example \cite{MWX}.
We now write  $u$ as 
\begin{align}
u = \<1> - \<30> + v, 
\label{decomp2}
\end{align}

\noi
where $v$ is expected to be smoother than $\<30>$.
As mentioned in Remark \ref{REM:behavior1},
 the decomposition $u = \<1> + v$ in the Da~Prato--Debussche trick 
 postulates $u$ behaves like $\<1>$ on small scales.
This new  decomposition~\eqref{decomp2}
postulates that the second order fluctuation of $u$
is given by $-\<30>$.
By further splitting $v$
as $v = X + Y$
and introducing more stochastic objects (corresponding to Step~(i) in~\eqref{FIG:1} below),   
one arrives at a system of PDEs
for $X$ and $Y$.\footnote{Here, we are oversimplifying the argument.
In fact, this decomposition $v = X + Y$ is based on a paracontrolled ansatz,
postulating that $(\dt - \Dl)v$ is {\it paracontrolled} by $\<2>$.
See \cite{CC, MW1} for further details.
We will describe details of this step 
in studying SNLW.
See \eqref{SNLW5} and \eqref{SNLW6}
}
Note that 
the  stochastic objects thus introduced 
satisfy certain regularity properties in an almost sure manner.
Hence, by simply viewing them  as given deterministic data,
 we solve the resulting system for $X$ and $Y$
in a purely {\it deterministic} manner (corresponding to Step (ii) in \eqref{FIG:1}).

The following diagram summarizes the discussion above:
\begin{align} (u_0, \xi) \stackrel{\tiny \text{(i)}}{\longmapsto} \underbrace{(u_0, \<1>, \<2>, \<30>, \<31p>, \<32p>, \<22p>)}_{\text{enhanced data set}} 
\stackrel{\tiny \text{(ii)}}{\longmapsto}   (X, Y) \longmapsto  u = \<1>- \<30> + X+Y, 
\label{FIG:1}
\end{align}

\noi
where  the ill-posed solution map:
$(u_0, \xi)\mapsto u$
is factorized into two steps (i) a canonical lift, generating an enhanced data set, 
and (ii) a deterministic continuous solution map called the It\^o-Lyons map.
Note that stochastic analysis is needed only in Step (i).
%
%
%
%
%
The decomposition~\eqref{FIG:1} together with the equations
satisfied by $X$ and $Y$ 
provides a precise meaning to 
the limiting equation \eqref{SQE1}.\footnote{The term $\infty\cdot u$
in \eqref{SQE1}  is introduced so that all the terms
appearing in the system for $X$ and $Y$ are finite.
Here, $\infty$ is interpreted as a limit of the diverging constant $C_\dl$
in \eqref{SQE2}, which depends on the choice of a mollifier $\rho_\dl$
(but the limiting distribution $u$ is independent of the choice of the mollifier).
See also Remark \ref{REM:uniq} below.
}
See a nice exposition in the introduction of~\cite{MW1}.
In~\cite{Hairer}, a similar decomposition of $u$ holds at the level of regularity structures
adapted to \eqref{SQE1}.

\medskip

In the following, we describe a procedure based on a paracontrolled ansatz.
This  transforms~\eqref{SNLW1}
into a system of PDEs, 
which we can solve by standard deterministic tools.

\begin{remark}\rm
The theory of regularity structures introduced by Hairer \cite{Hairer}
provides a more complete framework
to study singular parabolic equations than the paracontrolled calculus
introduced in \cite{GIP}.
However, 
the theory of regularity structures
 is more rigid and we do not know how to handle
 stochastic wave equations in high dimensions at this point.
 In particular, 
 we don't know how to lift the Duhamel integral operator $\I$.

Moreover, in the parabolic setting, it is easy to predict a regularity of a product.
In the theory of regularity structures, 
this  provides an intuition of a resulting homogeneity 
of a product of two elements in a regularity structure.\footnote{More precisely, 
a product of elements in a model space $T$ of a given regularity structure $(A, T, G)$.}
In the current dispersive setting, 
we need to exploit a multilinear smoothing property
to calculate a regularity of a product
of two functions (under the Duhamel integral operator)
in a much more careful manner.
Hence, any implementation of regularity structures
to study dispersive PDEs also needs to incorporate
this extra smoothing via an explicit product structure, 
which seems to be highly non-trivial.
\end{remark}

We keep our discussion 
at a formal level
and discuss spatial regularities (= differentiability) of various objects
without worrying about precise spatial Sobolev spaces that they belong to.
We also use the following ``rules'':\footnote{In the remaining part of the paper, 
we will justify these rules.}

\begin{itemize}

\item A product of functions of regularities $s_1$ and $s_2$
is defined if $s_1 + s_2 > 0$.
When $s_1 > 0$ and $s_1 \geq s_2$, the resulting product has regularity $s_2$.

\item A product of stochastic objects (not depending on the unknown)
is always well defined, possibly with a renormalization.
The product of stochastic objects of regularities $s_1$ and $s_2$
has regularity $\min( s_1, s_2, s_1 + s_2)$.

\end{itemize}

As in the case of  SQE \eqref{SQE1}, 
 we use~$\<1>$
to denote the stochastic convolution $\Psi$ for the wave equation defined in \eqref{sconv1}:
\begin{align}
 \<1> : = 
 \Psi = \I  (\xi)
  =  \int_{0}^t \frac{\sin ((t-t')\jb{\nb})}{\jb{\nb}}dW(t').
\label{so3}
\end{align}

\noi
In this context, 
the vertex ``\,$\<dot>$\,'' in $\<1>$ corresponds to the space-time white noise $\xi$,
while the edge denotes the Duhamel integral operator $\I$.
Recalling that the spatial regularity $-\frac 32 -$ of the space-time white noise $\xi$, 
the smoothing under $\I$ 
shows that 
 $\<1>$ has (spatial) regularity $- \frac 12-$.  See Lemma \ref{LEM:stoconv}.

Next, we define the second order stochastic term $\<20>$ by 
\begin{align}
  \<20> := 
   \I (\<2>)
  =  \int_{0}^t \frac{\sin ((t-t')\jb{\nb})}{\jb{\nb}}\<2>(t') dt', 
\label{so4}
\end{align}

\noi
where $\<2>$ is the renormalized version of $\<1>^2$;
see Proposition 2.1 in \cite{GKO}
and Lemma~\ref{LEM:stoconv} below.
This corresponds to the second term in the Picard iteration scheme for \eqref{SNLW1}
(with the zero initial data).
Note that the Wick power $\<2>$ has regularity $-1- = 2(-\frac{1}{2}-)$.
If one proceeds with a ``parabolic thinking'',\footnote{Namely, 
if we only count the regularity of each of $\<1>$ in $\<2>$
and put them together with one degree of smoothing
from the Duhamel integral operator $\I$ {\it without} taking into account the product structure
and the oscillatory nature of the linear wave propagator.}
then one might expect that $\<20>$ has
regularity 
\begin{align}
0- = 2\big(-\tfrac 12 -\big) + 1,
\label{naive1}
\end{align}

\noi
where we gain one derivative from the Duhamel integral operator $\I$, 
in particular from $\jb{\nb}^{-1}$ in \eqref{so4}.
In fact, we exhibit  an extra $\frac 12$-smoothing
for $ \<20>$  
by exploiting the explicit product structure
and multilinear dispersion 
 in \eqref{so4}.

Before proceeding further, let us introduce some notations.
Given $N \in \N$, we define the (spatial) frequency projector $\pi_N$  by 
\begin{align}
\pi_N u := 
\sum_{ |n| \leq N}  \ft u (n)  \, e_n.
\label{pi}
\end{align}

\noi
We then define the truncated stochastic terms  $\<1>_N$ 
and   $\<20>_N$ by 
\begin{align}
\<1>_N := \pi_N \<1>
\qquad \text{and}
\qquad 
  \<20>_N := \I (\<2>_N), 
\label{so4a}
\end{align}

\noi
where $\<2>_N$ is the Wick power
defined by 
\begin{align}
\<2>_N   := (\<1>_N)^2 - \s_N
\label{so4b}
\end{align}

\noi
 with\footnote{In our spatially homogeneous setting,
the variance $\s_N(t)$ is independent of $x \in \T^3$.} 
\begin{align}
\begin{split}
\s_N(t) & = \E\big[ (\<1>_N(x, t))^2\big]
=  \sum_{|n|\leq N}
\int_0^t \bigg[\frac{\sin((t - t')\jb{n})}{\jb{n}} \bigg]^2 dt'\\
& = \sum_{|n|\leq N} \bigg\{\frac{t}{2\jb{n}^2} - \frac{\sin(2t \jb{n})}{4\jb{n}^3}\bigg\}
\sim t N.
\end{split}
\label{sigma1}
\end{align}

\noi
Note that we have 
$\<2> = \lim_{N \to \infty} \<2>_N$ in $C([0,T];W^{-1 - ,\infty}(\T^3))$ almost surely.
See Lemma~\ref{LEM:stoconv} below.

\begin{proposition}\label{PROP:sto1}
Let   $T >0$.
Then, 
$ \<20>_N $ converges to 
$\<20>$
in 
$C([0,T];W^{\frac 12 -\eps,\infty}(\T^3)) \cap C^1([0,T];W^{-1 -\eps,\infty}(\T^3))$
almost surely
for any  $\eps > 0$.
 In particular, we have
  \[\<20> \in C([0,T];W^{\frac 12 -\eps,\infty}(\T^3))
  \cap C^1([0,T];W^{-1 -\eps,\infty}(\T^3))\]
  
  \noi
  almost surely for any  $\eps > 0$.
\end{proposition}

This proposition shows an extra $\frac 12$-smoothing
for $\<20>$ as compared to \eqref{naive1}.
This extra smoothing results from a multilinear interaction of waves
and is a manifestation of {\it dispersion} (at a multilinear level), 
which is  a key difference between  dispersive and parabolic equations.
In proving Proposition \ref{PROP:sto1}, 
we combine stochastic tools with multilinear dispersive analysis, 
in particular, carefully estimating
the  (nearly) time resonant and time non-resonant contributions.
See Remark \ref{REM:res}.
In the following, we will exploit
the dispersive nature of our problem in a crucial manner. 

We now write $u$ as 
\begin{align}
u = \<1> - \<20> + v.
\label{decomp3}
\end{align}

\noi
Then, it follows from \eqref{SNLW1} and \eqref{decomp3} that $v$ satisfies
\begin{align*}
(\dt^2  +1 -  \Dl)  v 
&  =  - (v + \<1> - \<20>)^2 + \<2>\notag \\
& = -  (v  - \<20>)^2 -2 v \<1> + 2 \<1>\<20>.
\end{align*}

\noi
At the second equality,  we performed the Wick renormalization: $\<1>^2 \rightsquigarrow \<2>$.
The last term  $\<1>\<20>$ 
has regularity $-\frac 12 -$, inheriting the worse regularity of $\<1>$.
Hence, we expect $v$ to have regularity at most $\frac 12- = (-\frac 12-) + 1$.
In particular, the product 
$v \<1>$ is not well defined since $(\frac 12-) + (-\frac 12-) < 0$.

In order to overcome this problem, 
we proceed with the paracontrolled calculus.
The main ingredients for the paracontrolled approach in the parabolic setting
are (i) a paracontrolled ansatz 
and (ii) commutator estimates.
For the wave equation, however, 
there seems to be no smoothing for a certain relevant commutator
(Remark \ref{REM:SQE})
and we need to introduce an alternative argument. 

Let us first recall
 the definition and basic properties of  paraproducts
 introduced by Bony~\cite{Bony}.
 See Section~\ref{SEC:2} and \cite{BCD, GIP} for further details.
Given $j \in \N\cup\{0\}$,  let $\P_{j}$
 be the (non-homogeneous) Littlewood-Paley projector
 onto the (spatial) frequencies $\{n \in \Z^3: |n|\sim 2^j\}$
 such that 
 \[ f = \sum_{j = 0}^\infty \P_jf.\]
 
 \noi
Given two functions $f$ and $g$ on $\T^3$
of regularities $s_1$ and $s_2$, 
we write the product $fg$ as
\begin{align}
fg & \hspace*{0.3mm}
= f\pl g + f \pe g + f \pg g\notag \\
& := \sum_{j < k-2} \P_{j} f \, \P_k g
+ \sum_{|j - k|  \leq 2} \P_{j} f\,  \P_k g
+ \sum_{k < j-2} \P_{j} f\,  \P_k g.
\label{para1}
\end{align}

\noi
The first term 
$f\pl g$ (and the third term $f\pg g$) is called the paraproduct of $g$ by $f$
(the paraproduct of $f$ by $g$, respectively)
and it is always well defined as a distribution
of regularity $\min(s_2, s_1+ s_2)$.
On the other hand, 
the resonant product $f \pe g$ is well defined in general 
only if $s_1 + s_2 > 0$.
In the following, we also use the notation $f\pge g := f\pg g + f\pe g$.

As in the study of SQE on $\T^3$, we now introduce our paracontrolled ansatz.
Namely, we suppose that $v = u - \<1> + \<20>$
can be decomposed as 
\begin{align}
v = X+ Y,
\label{decomp3a}
\end{align}

\noi
where $X$ and $Y$ satisfy
\begin{align}
 (\dt^2  +1 - \Dl) X  & = - 2 (X+Y-\<20>)\pl \<1>, \label{SNLW5}\\
 (\dt^2 +1  - \Dl) Y 
&  = - (X+Y-\<20>)^2 -  2(X+Y-\<20>)\pge \<1>.
\label{SNLW6}
\end{align}

\noi
Furthermore, 
we  postulate that both $X$ and $Y$ have positive regularities $s_1$ and $s_2$, respectively,
with $0 < s_1 < s_2$.

\begin{remark}\rm
We say that a distribution $f$ is paracontrolled (by a given distribution $g$)
if there exists $f'$ such that
\[ f= f' \pl g + h\]

\noi
where $h$ is a ``smoother'' remainder.
See Definition 3.6 in \cite{GIP}
for a precise definition.
Note, however, that the definition in \cite{GIP} 
is given in terms of the Besov-H\"older spaces $\mathcal{C}^s = B^s_{\infty, \infty}$
and is not necessarily useful for our dispersive problem.
Formally speaking, 
via the decomposition \eqref{decomp3a}
with \eqref{SNLW5} and the regularity assumption
$0 < s_1 < s_2$, 
we are postulating 
$(\dt^2  +1 - \Dl)v$ is paracontrolled by $\<1>$.

\end{remark}

From  the first equation  \eqref{SNLW5}, 
we see that
 $X$ has regularity $\frac 12- = (-\frac{1}{2}-) + 1$.
For now, let us ignore the resonant product 
$-2(X+Y-\<20>)\pe \<1>$ in \eqref{SNLW6}
and discuss the regularity of $Y$.
Recalling that $\<20>$ has regularity $\frac 12-$, 
we see that 
the paraproduct 
$- 2(X+Y-\<20>)\pg \<1>$ (with regularity $0-$)
as well as $- (X+Y-\<20>)^2 $ in \eqref{SNLW6} 
hints that $Y$ would have regularity $1- = (0-) + 1$.
This  is of course
provided that we can give a meaning
to the resonant product 
$-2(X+Y-\<20>)\pe \<1>$.
By postulating that $Y$ has regularity at least $\frac 12+\eps$ for some $\eps > 0$, 
we see that the resonant product 
$Y\pe \<1>$ makes sense as a distribution of regularity $s_2+( - \frac 12 -) > 0$
without any problem.
Furthermore, we can  make sense
of the following resonant product:
\begin{align}
\<21p> := \<20>\pe \<1>
\label{soX}
\end{align}

\noi
as a distribution of regularity $0- = (\frac{1}{2}-) + (-\frac 12 -)$
(without renormalization).

\begin{proposition}\label{PROP:sto2}
Let   $T >0$.
Then, 
$\<21p>_N := \<20>_N\pe \<1>_N$
 converges to $\<21p>$
 in 
$C([0,T];W^{ -\eps,\infty}(\T^3))$
almost surely
for any  $\eps > 0$.
 In particular, we have
  \[\<21p> \in C([0,T];W^{ -\eps,\infty}(\T^3))\]
  
  \noi
  almost surely  for any  $\eps > 0$.
\end{proposition}

If one simply writes out  $\<21p>$,  
then there seems to be a logarithmically divergent contribution
(see \eqref{Y3}).
We can, however, exploit dispersion at a multilinear level
as in Proposition~\ref{PROP:sto1}
and show that  $\<21p>$ is indeed a well defined distribution.

Hence, it remains to give a meaning
to the resonant product $X\pe \<1>$.
Writing 
the equation~\eqref{SNLW5}
in the Duhamel formulation, 
we have
\begin{align}
 X  = S(t) (X_0, X_1) -2\I \big((X+Y-\<20>)\pl \<1>\big),
\label{X1}
 \end{align}
\noi
where $(X, \dt X )|_{t = 0} = (X_0, X_1) \in \H^{s_1}(\T^3)$ and $S(t)$ is the propagator for the linear wave equation
defined by 
\begin{equation*}
S(t)(u_0, u_1) := \cos(t\jb{\nb})u_0+\frac{\sin (t\jb{\nb})}{\jb{\nb}}u_1.
\end{equation*}

We need to make sense of  the  resonant product between $\<1>$
and each of the terms on the right-hand side of~\eqref{X1}.
The next lemma establishes a regularity property of 
  the resonant product: 
\begin{align*}
Z =Z(X_0, X_1):= \big(S(t) (X_0,   X_1)\big) \pe \<1>.
\end{align*}

\begin{lemma}\label{LEM:IV}
Given   $s_1 > 0$, 
let  $(X_0, X_1) \in \H^{s_1}(\T^3)$.
Then, 
given  $T >0$  and $\eps > 0$,
\begin{align*}
Z_N := \big(S(t) (X_0,   X_1)\big) \pe \<1>_N
\end{align*}

\noi
 converges to $Z = \big(S(t) (X_0,   X_1)\big) \pe \<1>$
 in $C([0,T];H^{ s_1 - \frac 12 -\eps}(\T^3))$ almost surely.
 In particular, we have
\begin{align*}
Z = \big(S(t) (X_0,   X_1)\big) \pe \<1> \in C([0,T];H^{ s_1 - \frac 12 -\eps}(\T^3))
\end{align*}

\noi
  almost surely  for any  $\eps > 0$.

\end{lemma}

See Section \ref{SEC:po} for the proof.

\begin{remark}\label{REM:IV}\rm
While the proof of Lemma \ref{LEM:IV}
is a straightforward application of the Wiener chaos estimate (Lemma \ref{LEM:hyp}), 
we point out that the set of probability one on which 
the conclusion of Lemma \ref{LEM:IV} holds depends
on the choice of deterministic initial data $(X_0, X_1) \in \H^{s_1}(\T^3)$.
This is analogous to the situation
for the recent study of
nonlinear dispersive PDEs
with randomized initial data \cite{BT1, BT3, LM, BOP2, POC, OP, OP2, BOP3},
where a set of probability one for local-in-time or global-in-time well-posedness
depends on the choice of deterministic initial data (to which a randomization is applied).
See \cite{BT3} for a further discussion.

\end{remark}

The main difficulty arises in making sense of 
the resonant product of the second term on the right-hand side of \eqref{X1}
and $\<1>$.
In the parabolic setting, 
it is at this step where
one would introduce commutators in \eqref{X1} and exploit their smoothing properties.
For our dispersive problem, however, such an argument does not seem to work.
See Remark~\ref{REM:SQE} below.
This is where our discussion  diverges from the parabolic case.

The main idea is to study the following {\it paracontrolled operator} $\If_{\pl}$
 and exhibit some smoothing property.
Given a function $w \in C(\R_+; H^{s_1}(\T^3) )$ with $0 < s_1 < \frac 12$, define
\begin{align}
\begin{split}
 \If_{\pl}(w) (t)
:\! & =  \I (w\pl \<1>)(t) 
 = \sum_{j < k - 2} \I (\P_j w \cdot \P_k \<1>)\\
& =       \sum_{n \in \Z^3}
e_n  \sum_{\substack{n =  n_1 +  n_2\\  |n_1| \ll |n_2|}}
\int_0^t \frac{\sin ((t - t') \jb{n})}{\jb{n}} 
\ft w(n_1, t')\,  \ft{ \<1>}(n_2, t') dt'.
\end{split}
\label{X2}
\end{align}

\noi
Here, $|n_1| \ll |n_2|$ signifies the paraproduct $\pl$
in the definition of $\If_{\pl}$.\footnote{For simplicity of the presentation, 
we use the less precise definitions of paracontrolled operators
in the remaining part of this introduction.
See \eqref{XX2a}, \eqref{A0a},  and \eqref{A0b}
for the precise definitions of 
the paracontrolled operators $ \If_{\pl}^{(1)}$ 
and $\If_{\pl, \pe}$.}
In the following, we decompose the paracontrolled operator
$\If_{\pl}$ into two pieces
and study them separately.

Fix small $\theta  > 0$.
Denoting by $n_1$ and $n_2$ the spatial frequencies
of $w$ and $\<1>$ as in \eqref{X2}, 
we further define $ \If_{\pl}^{(1)}$ and $ \If_{\pl}^{(2)}$ 
as the restrictions of $\If_{\pl}$ onto $\{ |n_1| \ges |n_2|^{\theta}\}$
and  $\{ |n_1| \ll |n_2|^{\theta}\}$.
More concretely, we set 
\begin{align}
 \If_{\pl}^{(1)}(w)(t) 
: = \sum_{n \in \Z^3}
e_n  \sum_{\substack{n =  n_1 +  n_2\\    |n_2|^\theta \les |n_1| \ll |n_2|}}
\int_0^t \frac{\sin ((t - t') \jb{n})}{\jb{n}} 
\ft w(n_1, t')\,  \ft{ \<1>}(n_2, t') dt'
\label{X3}
\end{align}

\noi
and $ \If_{\pl}^{(2)}(w)   := \If_{\pl}(w) - \If_{\pl}^{(1)}(w)$.
As for the first paracontrolled operator $ \If_{\pl}^{(1)}$, 
thanks to the lower bound  $|n_1| \ges |n_2|^\theta$ 
 and the positive regularity of $w$, 
we exhibit some smoothing property
such that 
 the resonant product 
$ \If_{\pl}^{(1)}(X+Y-\<20>) \pe \<1>$ is well defined.
See Lemma~\ref{LEM:sto3}
and Corollary \ref{COR:sto3a}.

Next, we discuss the second paracontrolled operator $ \If_{\pl}^{(2)}$.
Our goal is to make sense of the resonant product
$  \If_{\pl}^{(2)}(w)\pe\<1>$
for $w$ with spatial regularity $\frac 12-$.
Unlike $ \If_{\pl}^{(1)}$, 
the operator  $ \If_{\pl}^{(2)}$ does not seem to possess a smoothing property
and thus we need to directly study the  operator $\If_{\pl, \pe}$ defined by
\begin{align}
\begin{split}
\If_{\pl, \pe}(w) (t)
 & :=
  \If_{\pl}^{(2)}(w)\pe\<1>(t) \\
&  = \sum_{n \in \Z^3}e_n 
    \int_0^{t} 
\sum_{n_1 \in \Z^3}
\ft w(n_1, t') \A_{n, n_1} (t, t') dt', 
\end{split}
\label{X6}
\end{align}

\noi
where $\A_{n, n_1} (t, t')$ is given by 
\begin{align}
\A_{n, n_1} (t, t')
= \ind_{[0 , t]}(t') \sum_{\substack{n - n_1 =  n_2 + n_3\\ |n_1| \ll |n_2|^\theta \\ |n_1 + n_2|\sim |n_3|}}
 \frac{\sin ((t - t') \jb{n_1 + n_2 })}{\jb{ n_1 + n_2 }} 
   \ft{\<1>}(n_2, t')\,   \ft{\<1>}(n_3, t) .
\label{X7}
\end{align}

\noi
Here, 
the condition  $ |n_1 + n_2|\sim |n_3|$
is used to denote the  spectral multiplier corresponding 
to the resonant product $\pe$ in \eqref{X6}.
See \eqref{A0b} for a more precise  definition.

Given $n \in \Z^3$ and  $0 \leq t_2\leq t_1$, 
define $\s_{n}(t_1, t_2 )$ by 
\begin{align}
\begin{split}
\s_{n}(t_1, t_2 )   : \! & =
\E  \big[  \ft{\<1>}(n, t_1)  \,  \ft{\<1>}(-n, t_2) \big]
  =
   \int_0^{t_2} \frac{\sin ((t_1 - t')\jb{n} )}{\jb{n}}
   \frac{\sin ((t_2 - t') \jb{n})}{\jb{n}} d t'\\
&  =    \frac{\cos((t_1 - t_2)\jb{n}) }{2 \jb{n}^2} t_2
+  \frac{\sin ((t_1-t_2)  \jb{n})}{4\jb{n}^3 }
-  \frac{\sin ((t_1 +  t_2) \jb{n})}{4\jb{n}^3 }.
\end{split}
\label{sigma2}
\end{align}

\noi
Recall from the definition \eqref{so3} (also see \eqref{sconv1})
that $\ft{\<1>}(n_2, t')$ and $\ft{\<1>}(n_3, t)$
are uncorrelated unless $n_2+ n_3 = 0$,
i.e.~$n = n_1$.
This leads to the following decomposition of $\A_{n, n_1}$:
\begin{align}
\begin{split}
\A_{n, n_1} (t, t')
& = \ind_{[0 , t]}(t') \sum_{\substack{n - n_1 =  n_2 + n_3\\ |n_1| \ll |n_2|^\theta \\ |n_1 + n_2|\sim |n_3|}}
 \frac{\sin ((t - t') \jb{ n_1 + n_2 })}{\jb{ n_1 + n_2 }} \\
& \hphantom{XXXXXXXX}
\times (   \ft{\<1>}(n_2, t') \, \ft{\<1>}(n_3, t) - \ind_{n_2 + n_3 = 0} \cdot \s_{n_2}(t, t')) \\
& \hphantom{X}
+ \ind_{[0 , t]}(t') \cdot \ind_{n=n_1}\cdot \sum_{\substack{n_2 \in \Z^3 \\ |n| \ll |n_2|^\theta}}
 \frac{\sin ((t - t') \jb{n + n_2 })}{\jb{ n + n_2} } 
 \s_{n_2}(t, t') \\
& = :
\A^{(1)}_{n, n_1} (t, t')+    \A^{(2)}_{n, n_1} (t, t').
\end{split}
\label{X9}
\end{align}

\noi
The second term $\A^{(2)}_{n, n_1}$ is a (deterministic) ``counter term''
for the contribution in \eqref{X7} from $n_2 + n_3 = 0$.
For this term,  the condition $|n_1 + n_2|\sim |n_3|$
reduces to $|n + n_2|\sim |n_2|$ which is automatically satisfied
under $|n| \ll |n_2|^\theta$ for small $\theta > 0$.
See \eqref{Y12} and \eqref{Y13} below.

In view of \eqref{sigma2}, 
 the sum in $n_2$ for the second term $\A^{(2)}_{n, n_1}$ is not absolutely convergent.
Nonetheless, 
by  exploiting dispersion, 
we  
show the following boundedness
property
of the paracontrolled operator $\If_{\pl, \pe}$ defined in \eqref{X6}.
 Given Banach spaces $B_1$ and $B_2$, 
 we use $\L(B_1; B_2)$ to denote the space
 of bounded linear operators from $B_1$ to $B_2$.

\begin{proposition}\label{PROP:sto4}
Let  $ s_2 < 1$ and $T > 0$.
Then, there exists small $\theta = \theta (s_2) > 0$
and $\eps > 0$
such that
the paracontrolled operator $\If_{\pl, \pe}$ 
defined in \eqref{X6}
belongs to  the class:
\begin{align}
\L_1 = \L(C([0, T]; L^2(\T^3) ) \cap C^1([0,T];H^{-1 -\eps}(\T^3))
\, ;\, 
 C([0, T]; H^{s_2 -1}(\T^3) )),
\label{L2}
\end{align}
almost surely.
Moreover, 
if we define  the paracontrolled operator $\If_{\pl, \pe}^N$, $N \in \N$, 
by replacing 
$\<1>$ in \eqref{X6} and \eqref{X7}
with the truncated stochastic convolution $\<1>_N$ in \eqref{so4a}, 
then
the truncated paracontrolled operators $\If_{\pl, \pe}^N$ converge almost surely to $\If_{\pl, \pe}$ 
in $\L_1$.
\end{proposition}

As in the proofs of Propositions \ref{PROP:sto1} and \ref{PROP:sto2}, 
dispersion plays an essential role in 
establishing the regularity property
of  the paracontrolled operator $\If_{\pl, \pe}$.
See Section \ref{SEC:po} for the proof.

Putting all together, we obtain the following system 
of PDEs for $X$ and $Y$:
\begin{align}
\begin{split}
 (\dt^2  + 1- \Dl) X  &  = - 2 (X+Y-\<20>)\pl \<1>, \\
 (\dt^2 +  1- \Dl) Y  
&   = - (X+Y-\<20>)^2 -  2(X+Y-\<20>)\pg \<1>\\
& \hphantom{X} 
  - 2 Y\pe \<1> + 2 \<21p>
 -2Z \\ 
& \hphantom{X} 
+  4 \If_{\pl}^{(1)}(X + Y - \<20>)\pe \<1> +4 \If_{\pl, \pe}(X + Y - \<20>), \\
(X,  \dt X, Y,     \dt Y)|_{t = 0} & = (X_0, X_1, Y_0, Y_1).
\end{split} 
 \label{SNLW7}
\end{align}

\noi
Let $s_1 < \frac 12$
and fix a pair of deterministic functions  $(X_0, X_1)$ in  $\H^{s_1}(\T^3)$.
The stochastic terms and operator appearing in the system \eqref{SNLW7} are 
\begin{align}
\<1>,\quad \<20>,\quad  \<21p>,
\quad  
Z = Z(X_0, X_1), 
\quad 
\text{and} 
\quad \If_{\pl, \pe}, 
\label{so5}
\end{align}

\noi
In  Lemma  \ref{LEM:stoconv} and Propositions \ref{PROP:sto1} and   \ref{PROP:sto2}, 
we  study the regularity properties of $\<1>$, $\<20>$, and  $\<21p>$, 
and show that 
each of these terms  belongs almost surely to 
$C(\R_+; W^{s, \infty}(\T^3))$
with the regularity $s$ shown in Table \ref{TAB:1}.
In  Lemma \ref{LEM:IV}, 
we prove that $Z \in C(\R_+; H^{s}(\T^3))$
almost surely for $s < s_1 - \frac{1}{2}$.
In Proposition~\ref{PROP:sto4}, 
we establish the almost sure boundedness property
of the paracontrolled operator
 $\If_{\pl, \pe}$
in an appropriate space.
We summarize these regularity properties
 in Table~\ref{TAB:1}.

\begin{table}[h]
  \begin{tabular}{|l||c|c|c|c|c|}
\hline
&  \rule{0mm}{5mm} 
\hspace*{2.5mm}  $\<1>$ \hspace*{3mm}
&\hspace*{3mm}  $\<20>$ \hspace*{4mm}
&\hspace*{3mm}  $\<21p>$ \hspace*{4mm}
&\hspace*{4.5mm}  $Z$ \hspace*{5.5mm} 
& $\If_{\pl, \pe}$ 
\\
\hline
\rule[-2.5mm]{0mm}{7mm}
$s$  
\hspace*{1.5mm} & $- \frac12 -\eps$ & $ \frac12 -\eps$ & $-\eps$  
& $s_1 - \frac 12 -\eps$
& \hspace*{2mm}$\L_1$ in \eqref{L2}\hspace*{2mm}
\\
\hline
  \end{tabular}

\vspace{4mm}
  \caption{The list of relevant stochastic terms
  with their regularities}
  \label{TAB:1}
\end{table}

\noi
In  Lemma \ref{LEM:sto3}
and Corollary \ref{COR:sto3a}, 
we also study 
the regularity property
of the paracontrolled operator $ \If_{\pl}^{(1)}$.

We now state our main result
on local well-posedness of the system \eqref{SNLW7}, 
viewing the terms and operators in \eqref{so5}
as {\it predefined deterministic} data
with certain regularity properties.

\begin{theorem}\label{THM:1}
Let $\frac 14 < s_1 < \frac 12 < s_2  \leq   s_1 +\tfrac 14$.
There exist small $\ta = \ta(s_2) > 0$
and 
$\eps = \eps(s_1, s_2, \ta) > 0$
such that 
if 
\begin{itemize}
\item   $\<1>, \, \<20>,$ and $\<21p>$,  are distributions
belonging
to 
$C(\R_+; W^{s, \infty}(\T^3))$
for $s$ as in Table~\ref{TAB:1}.
Moreover, we assume that 
\[ \<20>\in   C^1(\R_+;W^{-1 -\eps,\infty}(\T^3)),\]

\item 
$Z$ is a distribution 
belonging to $C(\R_+; H^{s_1 - \frac{1}{2}-\eps}(\T^3))$, 

\smallskip

\item 
 the operator 
 $\If_{\pl, \pe}$
belongs to  the class 
 $\L_1$ in \eqref{L2}, 

\end{itemize}

\noi
then
the system \eqref{SNLW7} is locally well-posed
in $\H^{s_1}(\T^3)\times \H^{s_2}(\T^3)$.
More precisely, 
given any $(X_0, X_1, Y_0, Y_1) \in \H^{s_1}(\T^3)\times \H^{s_2}(\T^3)$, 
there exists $T > 0$ 
such that there exists  a unique solution $(X, Y) $ to the system \eqref{SNLW7} on $[0, T]$
in the class 
\begin{align*}
 Z^{s_1, s_2}_T 
 & =  X^{s_1}_T\times  Y^{s_2}_T
 & \subset 
C([0, T];H^{s_1}(\T^3)\times H^{s_2}(\T^3))
\cap C^1([0, T]; H^{s_1-1}(\T^3)\times H^{s_2-1}(\T^3)), 
\end{align*}

\noi
depending  continuously 
on the enhanced data set:
\begin{align}
\Xi = \big(X_0, X_1, Y_0, Y_1, \<1>, \<20>,  \<21p>, Z, 
 \If_{\pl, \pe}\big)
\label{data1}
\end{align}

\noi
in the class:
\begin{align}
\begin{split}
\mathcal{X}^{s_1, s_2, \eps}_T
& = \H^{s_1}(\T^3) \times 
\H^{s_2}(\T^3)
\times 
C([0,T]; W^{-\frac 12 - \eps, \infty}(\T^3))\\
& \hphantom{X}
\times 
\big(C([0,T]; W^{\frac 12 - \eps, \infty}(\T^3))
\cap C^1([0,T];W^{-1 -\eps,\infty}(\T^3))\big)\\
& \hphantom{X}
\times
C([0,T]; W^{ - \eps, \infty}(\T^3))
 \times 
C([0, T]; H^{s_1 - \frac{1}{2}-\eps}(\T^3))
\times 
 \L_1.
\end{split}
\label{L3}
\end{align}

\noi
Here, 
$X^{s_1}_T$ and $Y^{s_2}_T$ are 
the energy spaces at the regularities $s_1$ and $s_2$
intersected with appropriate Strichartz spaces.
See \eqref{M0}  below.

\end{theorem}

Theorem \ref{THM:1} states local well-posedness of the system
\eqref{SNLW7} when we view 
the enhanced data set $\<1>$,  $\<20>$,  $\<21p>$,
$Z$, 
and $\If_{\pl, \pe}$ 
as given deterministic distributions or operator.
As such, the proof of Theorem \ref{THM:1} is entirely deterministic.
By writing~\eqref{SNLW7} in the Duhamel formulation
\begin{align}
\begin{split}
 X (t)  &  = \Phi_1(X, Y)(t) \\
 & := S(t)(X_0, X_1) - 2 \int_0^t \frac{\sin((t-t')\jb{\nb})}{\jb{\nb}} 
 \big[(X+Y - \<20>)\pl \<1>\big](t')dt',  \\
Y  (t) &   = \Phi_2(X, Y)(t) \\
 & :=
S(t)(Y_0, Y_1) -
 \int_0^t \frac{\sin((t-t')\jb{\nb})}{\jb{\nb}} 
\Big[(X+Y-\<20>)^2 +  2(X+Y-\<20>)\pg \<1>\\
& 
\hspace{25.5mm}
+ 2 Y\pe \<1> - 2 \<21p>
 + 2Z \\
& \hspace{25.5mm}
 -4 \If_{\pl}^{(1)}(X + Y - \<20>)\pe \<1> -4\If_{\pl, \pe}(X + Y - \<20>)\Big] (t') dt',
\end{split}
 \label{SNLW8}
\end{align}

\noi
we show  that the map $\Phi = (\Phi_1, \Phi_2)$
is a contraction on a closed ball in 
$Z^{s_1, s_2}_T$
for sufficiently small $T>0$
which  depends only on 
the $\mathcal{X}^{s_1, s_2, \eps}_T$-norm 
of the enhanced data set $\Xi$ in  \eqref{data1}.
The main tools are (i) the Strichartz estimates
for the wave equations (Lemma~\ref{LEM:Str}) 
and (ii)~the paraproduct estimates (Lemma~\ref{LEM:para}).
See Section \ref{SEC:proof1} for details.

\smallskip

Finally, let us discuss the consequence of Theorem \ref{THM:1}
on the original SNLW \eqref{SNLW1}.

\begin{proof}[Proof of Theorem \ref{THM:2-intro}]

Let $\frac 14 <  s < \frac 12$.
Given  $(u_0, u_1) \in \H^{s}(\T^3)$, 
let $(X_0, X_1, Y_0, Y_1) = (u_0, u_1, 0, 0)$.
For each $N \in \N$, 
we construct the enhanced data set associated with the truncated noise $\xi_N = \pi_N \xi$:
\begin{align*}
\Xi_N
= \big(u_0, u_1, 0, 0, \<1>_N, \<20>_N,  \<21p>_N, 
Z_N, 
 \If_{\pl, \pe}^N\big).
\end{align*}

\noi
Here, 
 $\<1>_N$, $\<20>_N$,  $\<21p>_N$,
 and  
$ \If_{\pl, \pe}^N$
are as in \eqref{so4a} and Propositions \ref{PROP:sto2} and \ref{PROP:sto4}, 
while we set 
$Z_N = Z_N(u_0, u_1) = S(t)(u_0, u_1)\pe \<1>_N$.
Let $(X_N, Y_N)$ be the unique local-in-time
solution to the system \eqref{SNLW7}
with  the enhanced data set $\Xi_N$
and  define $u_N$ by
\begin{align}
 u_N =  \<1>_N - \<20>_N + X_N + Y_N.
\label{decomp4}
 \end{align}

\noi
Then, by reversing the discussion above with \eqref{so4b}, 
we see that $u_N$ satisfies the renormalized SNLW \eqref{SNLW10} \emph{provided}  $\s_N$ is chosen as in \eqref{sigma1}.

From 
Lemma \ref{LEM:stoconv},
Propositions \ref{PROP:sto1},  \ref{PROP:sto2}, 
Lemma \ref{LEM:IV}, 
Corollary \ref{COR:sto3a}, 
and
Proposition \ref{PROP:sto4}, 
we see that $\Xi_N$ converges 
almost surely to 
\begin{align}
\Xi =  \big(u_0, u_1, 0, 0, \<1>, \<20>,  \<21p>,  
S(t)(u_0, u_1)\pe \<1>, 
\If_{\pl, \pe}\big)
\label{so7}
\end{align}

\noi
in the $\mathcal{X}^{s, \frac 12 +\eps, \eps}_1$-topology
for some small $\eps > 0$.
Then, 
the (pathwise) continuous dependence
of the solution map for the system \eqref{SNLW7}
on the enhanced data set in $\mathcal{X}^{s, \frac 12 +\eps, \eps}_1$
implies 
that 

\medskip

\begin{itemize}
\item
the (random) local existence time  $T = T(\o)$ 
depicted in Theorem \ref{THM:1} can be chosen uniformly 
for $\big\{(X_N, Y_N) \big\}_{N \in \N}$ and $(X, Y)$.
Here,  $(X, Y)$ is the unique solution to~\eqref{SNLW7}
with the enhanced data $\Xi$ in \eqref{so7}.

\smallskip

\item 
the solution $u_N$ to the renormalized SNLW \eqref{SNLW10}
defined  in \eqref{decomp4}
converges almost surely to  $u$ 
in $C([0, T]; H^{-\frac 12 -\eps} (\T^3))$, 
where $u$ is 
 given by 
\begin{align}
u = \<1> - \<20> + X + Y.
\label{so8}
\end{align}

\end{itemize}

\noi 
This proves Theorem~\ref{THM:2-intro} under the condition that $\s_N$ is chosen as described in  
\eqref{sigma1}.
\end{proof}

\begin{remark}\rm
As we pointed out in Remark \ref{REM:IV}, 
the set $\Si$ of probability one
on which Theorem \ref{THM:2-intro} holds
depends
on the choice of (deterministic) initial data $(u_0, u_1) \in \H^s(\T^3)$
due to Lemma~\ref{LEM:IV}.
If  we assume a slightly higher regularity, 
namely, if we work with $(u_0, u_1) \in \H^{s}(\T^3)$ for some $s > \frac 12$, 
we can choose the set $\Si$ of probability one, independent of $(u_0, u_1) \in \H^s(\T^3)$, 
by simply setting
$(X_0, X_1, Y_0, Y_1) = (0, 0, u_0, u_1)$,
which avoids the use of Lemma \ref{LEM:IV}.
\end{remark}

\begin{remark} \label{REM:uniq}\rm
Given  $\rho \in C^\infty_c(\R^3)$ 
with   $\int_{\R^3} \rho(x) dx = 1$
and $\supp \rho \subset [-\frac 12, \frac 12)^3\simeq  \T^3$, 
we define a smooth mollifier $\rho_\dl$, $0< \dl \leq 1$, 
by setting
\begin{align}
\rho_\dl(x) = \dl^{-3} \rho(\dl^{-1} x).
\label{molli2}
\end{align}

\noi
We also say that such $\rho$ is a mollification kernel. Then, the same argument leading to Theorem~\ref{THM:2-intro} can be used to prove the following convergence and uniqueness statement.
See~\cite{Hairer, CC}.
Given $\frac 14 < s < \frac 12$, 
let $(u_0, u_1) \in \H^{s}(\T^3)$.
Let $ \xi_\dl =  \rho_\dl * \xi$
be the smoothed noise
by a smooth mollifier $\rho_\dl$.
Then, for any $0< \dl \leq 1$, 
there exists $C_\dl=C_\dl(t, \rho)$
such that 
the solution $u_\dl$ to the following smoothed SNLW:
\begin{align*}
\begin{cases}
\dt^2 u_\dl + (1-  \Dl)  u_\dl  =  - u_\dl^2  + C_\dl+ \xi_\dl\\
(u_\dl, \dt u_\dl)|_{t = 0} = (u_0, u_1)
\end{cases}
\end{align*}

\noi
converges in probability  to some distribution $u$  in $C([0, T]; H^{-\frac 12 -\eps}(\T^3))$
for any $\eps > 0$, 
where $T = T(\o)$ is an almost surely positive stopping time, 
independent of $0< \dl \leq 1$.
Here, we have
$C_\dl(t, \rho) =  C_0 \frac{t}{\dl} + C(t, \rho)$, 
where $C_0$ is a universal constant and $C(t, \rho)$ is a finite constant.
Moreover, the limit $u$ is unique 
in the sense that it is independent
of the choice of the mollification kernel $\rho$.

In the proof of Proposition~\ref{PROP:sto4} presented 
in Section \ref{SEC:po} below, 
we make use of certain symmetry,
which may seem to suggest that 
the Fourier transform of a mollification kernel $\rho$ needs to be symmetric.
It is, however, possible to extend 
Theorem~\ref{THM:2-intro}
for general mollification 
even if the Fourier transform of a mollification kernel $\rho$ is not symmetric.
See Remark~\ref{REM:sym}
for a further discussion.

Furthermore, we point out that we can also consider
space-time mollifiers and obtain an analogous result.
In this case,  we impose an additional assumption\footnote
{As pointed out in \cite{Hairer0, Hairer}, 
in the case of the KPZ equation, 
 regularization via a non-symmetric space-time mollifier can cause the appearance 
  of an additional transport term. 
    See Proposition 15.12 and Remark 15.13 in 
  \cite{friz_course_2014}.
For the quadratic SNLW  \eqref{SNLW1} on $\T^3$, 
it may also be possible to use 
 regularization via a non-symmetric space-time mollifier 
by introducing new types of counter terms
for $\<21p>$ and the paracontrolled operator
$\If_{\pl, \pe}$.
We, however,  do not pursue this issue in this paper.}
that a space-time mollification kernel $\rho(x, t)$ is even in $x$, 
namely, $\rho(-x, t) = \rho(x, t)$ for any $t \in \R$.
See Remark \ref{REM:MAX}.
In the context of  
Theorem \ref{THM:weak} on the weak universality, 
this modification allows us to handle 
 noises that are smooth in both space and time.
\end{remark}

Let us complete this section by some additional observations.

\begin{remark}\rm
As we saw in \eqref{sigma1}, the variance $\s_N(t)$ of the truncated stochastic convolution
$\<1>_N $ is time dependent, 
resulting in a time-dependent renormalization constant
in Theorem \ref{THM:2-intro}.
This is due to the lack of any dissipation mechanism in the dispersive setting.
In the parabolic setting, for example, in the case of SQE \eqref{SQE1}, 
there exists a unique invariant measure 
for the truncated linear stochastic dynamics, 
which allows us to take time-independent renormalization constants.
In the wave equation case, 
we may consider the equation with damping,
namely, replace the left-hand side of
\eqref{SNLW1} by $\dt^2 u  + \dt u + (1 -  \Dl)  u$
such that there exists a unique invariant Gaussian measure $\mu_N$ for the (truncated) linear dynamics.
In this case, 
by taking initial data  distributed according to this invariant Gaussian measure $\mu_N$,
the variance of the truncated stochastic convolution
becomes time independent
and thus we can use a time-independent renormalization constant.
See \cite{GKOT, ORTz, OOR, ORW}.

We point out that
in the parabolic setting, it is possible to start with arbitrary deterministic initial data~$u_0$
(under some regularity assumption)
and still use time-independent renormalization constants.
This is thanks to the 
 strong parabolic smoothing, 
  allowing us to handle rough initial data of the form 
  $u_0  - z_0$, 
  where $z_0$
  is a random function distributed by the massive Gaussian free field:
 \[ z_0 = \sum_{n \in \Z^3} \frac{g_n}{\jb{n}}e_n.\]
 
 \noi
 Here, $\{g_n\}_{n \in \Ld_0}$ is a sequence of independent
 standard complex-valued Gaussian random variables 
 and $g_{-n}: = \cj{g_n}$, $n \in \Ld_0$.
 On the other hand, in the damped  wave case, 
due to the lack of strong smoothing, 
our solution theory does not allow us to 
 handle the random data 
 of the form $(u_0, u_1) - (z_0, z_1)$, 
 where $z_0$ is as above and $z_1$ is distributed by the spatial white noise.
Unfortunately, such initial data is too rough to handle in the deterministic manner
for the damped wave equation.
 This in particular implies that 
for the damped wave equation, 
it is not possible to start with 
 arbitrary deterministic initial data
(under some regularity assumption)
and  use a time-independent renormalization constant.
See also Remark 1.2\,(iii) in~\cite{OOR}.

\end{remark}

\begin{remark}
\rm

(i) 
In making sense of the resonant product $X\pe \<1>$, 
we substituted the Duhamel formula
for $X$ as in \eqref{X1}.
This is analogous to the treatment of SQE \eqref{SQE1}; see \cite{MW1}.
Note that 
such an iteration of the (part of) Duhamel formula
already appears in  in the study of the stochastic KdV equation 
with an additive (almost) white noise.
See~\cite{Gub12, Oh}.

\smallskip

\noi
(ii) 
Unlike the parabolic setting, we need to assume higher regularity for initial data
than the stochastic convolution.
This is due to the lack of smoothing in our dispersive problem.
If initial data is random (independent of the additive space-time white noise), we may take it to be of low regularity.

\smallskip

\noi
(iii) In Proposition \ref{PROP:sto4}, 
we assumed one time differentiability
of an input function for the paracontrolled operator
$\If_{\pl, \pe}$. 
This smoothness in time allows us to exploit the time oscillation
by integration by parts.  See \eqref{A7} below.
On the one hand, we may prove an analogous
boundedness result by assuming less time regularity of an input function.
On the other hand, it seems that 
we do need to assume some time regularity of an input function.
This necessity for smoothness in time is analogous to the parabolic setting,
but for a different reason in the parabolic setting; see \cite{Hairer, CC, MW1}.

%
%

\end{remark}

\begin{remark}[On commutators]\label{REM:SQE}
\rm

As we mentioned above, 
commutators play an important role in applying
the paracontrolled calculus in the parabolic setting.
If we were to follow the argument for SQE presented in \cite{MW1}, 
then we would 
write \eqref{X1} as 
\begin{align}
\begin{split}
 X   & = S(t) (X_0, X_1) - 2\I \big((X  +Y+\<20>)\pl \<1>\big)\\
&   = S(t) (X_0, X_1) - 2 (X+Y+\<20>)\pl \I(\<1>)+ \com_1.
\end{split}
\label{com0a}
 \end{align}

\noi
Here, the commutator
$\com_1$ denotes
the commutator of the paraproduct $\pl$ and the Duhamel integral operator $\I = (\dt^2 + 1 - \Dl)^{-1}$.

In the case of SQE \eqref{SQE1} on $\T^3$, it was crucial that 
the commutator of the paraproduct $\pl$ and the Duhamel integral operator $ (\dt - \Dl)^{-1}$
for the heat equation
enjoyed some smoothing property,
which resulted from the smoothing property of the commutator  $[e^{t\Dl}, \pl]$ between 
 the linear heat semigroup $e^{t \Dl}$ and 
the paraproduct $\pl$
(see Lemma 2.5 in \cite{CC}
and Proposition~A.16 in \cite{MW1}).
Unfortunately, in our dispersive setting, 
 the commutator $\com_1$ does not seem to provide any smoothing.
We point out that 
if the identity \eqref{com0a} were to hold
with a smoother commutator, 
then the rest would follow as in the parabolic setting \cite{MW1}
(and in particular, there would no need to introduce
paracontrolled operators).
Namely,  by defining
\begin{align*}
[\pl, \pe](f, g, h) =  (f\pl g) \pe h - f (g \pe h), 
\end{align*}

\noi
we can write $ X \pe \<1>$ as 
%
\begin{align*}
 X \pe \<1> = 
 S(t) (X_0, X_1) - 2 (X+Y+\<20>) \big(\I(\<1>)\pe\<1>\big)
+ \com_1 \pe  \<1>+ \com_2 , 
 \end{align*}

\noi
where $\com_2$ is given by 
$\com_2  = [\pl, \pe]\big(X+Y+\<20>,  \I(\<1>), \<1>\big)$.
Note that $\com_2$ is a well defined distribution 
thanks to the smoothing property of  $[\pl, \pe]$.
See Lemma~2.4 in \cite{GIP}
and Proposition~A.9 in \cite{MW1}.

Let us now consider the first commutator $\com_1$.
Given an operator $T$, 
let 
\begin{align*}
\big[T, \pl\big](f, g) =T(f\pl g) - f\pl (T g).
\end{align*}

\noi
Then, 
by setting   $\S = \jb{\nb} \I = \jb{\nb} (\dt^2 + 1 - \Dl)^{-1}$, 
we have
\[ [\I, \pl](f, g) = \S \circ \big[\jb{\nb}^{-1}, \pl\big](f, g) + [\S, \pl] (f,  \jb{\nb}^{-1}g).\]

\noi
It is easy to see that the first commutator 
$[ \jb{\nb}^{-1}, \pl]$ enjoys certain smoothing.\footnote{If $f$ and $g$ have regularities $0 < s_1 < 1$ and $s_2<0$ with  $s_1 + s_2 < 0$, 
then each of $\jb{\nb}^{-1}(f \pl g)$ and $f \pl (\jb{\nb}^{-1}g)$
has 
regularity $s_2 + 1$.
On the other hand, the commutator $\big[\jb{\nb}^{-1}, \pl\big](f, g)$
has regularity $s_1 + s_2 + 1$.
Roughly speaking, this fact follows from the following observation;
given $n, n_1, n_2 \in \Z^3$ with  $n = n_1 + n_2$, we have
\begin{align*}
\bigg|\frac{1}{\jb{n}} - \frac{1}{\jb{n_2}}\bigg|
= \frac{\big|\jb{n_2} - \jb{n}\big|}{\jb{n}\jb{n_2}}
\les \frac{\jb{n_1}}{\jb{n}\jb{n_2}}
\end{align*}

\noi
In particular, when $|n_1| \ll |n_2| \sim |n|$
and the first function $f$ has positive regularity, 
this observation provides  smoothing.
}
On the other hand, 
if we were to exhibit smoothing 
for  the second commutator $[\S, \pl]$ as in the parabolic setting, 
we would need to study the smoothing property
of the commutator $[ \,\sin(t \jb{\nb}), \pl]$.
Unfortunately, there seems to be no smoothing
for this commutator in general,\footnote{Under $|n_1| \ll |n_2|$, 
there is no smoothing for
$\sin( t\jb{n_1+n_2}) - (\sin t \jb{n_2})$.
}
which prevents us from working with commutators for our dispersive problem.
By introducing the paracontrolled operators, 
we indeed exhibit smoothing
under 
the commutator $[\S, \pl]$ (and hence under $[\I, \pl]$)
in a probabilistic manner with a specific second input function, i.e.~$g = \<1>$.
 See Proposition \ref{PROP:sto4}.
This is in sharp contrast with the parabolic setting,
where a smoothing can be shown for 
 $[e^{t\Dl}, \pl]$ in a deterministic manner
 (without specifying the second input function either).

Lastly, we point out that our approach
via  paracontrolled operators
also works in the parabolic setting.
In particular, 
in place of using commutators, 
we can direct study relevant  paracontrolled operators
to prove local well-posedness of SQE \eqref{SQE1} on $\T^3$.

\end{remark}

\section{Notations and basic lemmas}
\label{SEC:2}

\subsection{Sobolev spaces 
 and Besov spaces}

Let $s \in \R$ and $1 \leq p \leq \infty$.
We define the $L^2$-based Sobolev space $H^s(\T^3)$
by the norm:
\begin{align*}
\| f \|_{H^s} = \| \jb{n}^s \ft f (n) \|_{\l^2_n}
\end{align*}

\noi
and set $\H^s(\T^3)$ to be 
\[\H^s(\T^3) = H^s(\T^3)\times H^{s-1}(\T^3).\]

\noi
We also define the $L^p$-based Sobolev space $W^{s, p}(\T^3)$
by the norm:
\begin{align*}
\| f \|_{W^{s, p}} = \big\| \F^{-1} (\jb{n}^s \ft f(n))\big\|_{L^p}
\end{align*}

\noi
with the standard modification when $p = \infty$.
When $p = 2$, we have $H^s(\T^3) = W^{s, 2}(\T^3)$.

Let $\phi:\R \to [0, 1]$ be a smooth  bump function supported on $[-\frac{8}{5}, \frac{8}{5}]$ 
and $\phi\equiv 1$ on $\big[-\frac 54, \frac 54\big]$.
For $\xi \in \R^3$, we set $\phi_0(\xi) = \phi(|\xi|)$
and 
\[\phi_{j}(\xi) = \phi\big(\tfrac{|\xi|}{2^j}\big)-\phi\big(\tfrac{|\xi|}{2^{j-1}}\big)\]

\noi
for $j \in \N$.
Then, for $j \in \N_0 := \N \cup\{0\}$, 
we define  the Littlewood-Paley projector  $\P_j$ 
as the Fourier multiplier operator with a symbol $\varphi_j$
given by 
\begin{align}
 \varphi_j(\xi) = \frac{\phi_j(\xi)}{\sum_{k \in \N_0} \phi_k(\xi)}.
\label{phi1}
 \end{align}

\noi
Note that, for each $\xi \in \R^3$,  the sum in the denominator is over finitely many $k$'s.
Thanks to the normalization \eqref{phi1}, 
we have 
\[ f = \sum_{j = 0}^\infty \P_j f, \]

\noi
which is used in \eqref{para1}.

We briefly recall the basic properties of the Besov spaces $B^s_{p, q}(\T^3)$
defined by the norm:
\begin{equation*}
\| u \|_{B^s_{p,q}} = \Big\| 2^{s j} \| \P_{j} u \|_{L^p_x} \Big\|_{\l^q_j(\N_0)}.
\end{equation*}

\noi
Note that  $H^s(\T^3) = B^s_{2,2}(\T^3)$.

\begin{lemma}\label{LEM:para}
\textup{(i) (paraproduct and resonant product estimates)}
Let $s_1, s_2 \in \R$ and $1 \leq p, p_1, p_2, q \leq \infty$ such that 
$\frac{1}{p} = \frac 1{p_1} + \frac 1{p_2}$.
Then, we have 
\begin{align}
\| f\pl g \|_{B^{s_2}_{p, q}} \les 
\|f \|_{L^{p_1}} 
\|  g \|_{B^{s_2}_{p_2, q}}.  
\label{para2a}
\end{align}

\noi
When $s_1 < 0$, we have
\begin{align}
\| f\pl g \|_{B^{s_1 + s_2}_{p, q}} \les 
\|f \|_{B^{s_1 }_{p_1, q}} 
\|  g \|_{B^{s_2}_{p_2, q}}.  
\label{para2}
\end{align}

\noi
When $s_1 + s_2 > 0$, we have
\begin{align}
\| f\pe g \|_{B^{s_1 + s_2}_{p, q}} \les 
\|f \|_{B^{s_1 }_{p_1, q}} 
\|  g \|_{B^{s_2}_{p_2, q}}  .
\label{para3}
\end{align}

\noi
\textup{(ii)}
Let $s_1 <  s_2 $ and $1\leq p, q \leq \infty$.
Then, we have 
\begin{align} 
\| u \|_{B^{s_1}_{p,q}} 
&\les \| u \|_{W^{s_2, p}}.
\label{embed}
\end{align}

\end{lemma}

The product estimates \eqref{para2a},  \eqref{para2},  and \eqref{para3}
follow easily from the definition \eqref{para1} of the paraproduct 
and the resonant product.
See \cite{BCD, MW2} for details of the proofs in the non-periodic case
(which can be easily extended to the current periodic setting).
The embedding \eqref{embed}
 follows from the $\l^{q}$-summability 
of $\big\{2^{(s_1 - s_2)j}\big\}_{j \in \N_0}$ for $s_1 < s_2$
and the uniform boundedness of the Littlewood-Paley projector $\P_j$.

We also recall the following fractional Leibniz rule.

\begin{lemma}\label{LEM:bilin}
 Let $0\le s \le 1$.
Suppose that 
 $1<p_j,q_j,r < \infty$, $\frac1{p_j} + \frac1{q_j}= \frac1r$, $j = 1, 2$. 
 Then, we have  
\begin{equation*}  
\| \jb{\nb}^s (fg) \|_{L^r(\T^d)} 
\les  \| f \|_{L^{p_1}(\T^d)} 
\| \jb{\nb}^s g \|_{L^{q_1}(\T^d)} + \| \jb{\nb}^s f \|_{L^{p_2}(\T^d)} 
\|  g \|_{L^{q_2}(\T^d)}.
\end{equation*}  

\end{lemma} 

This lemma follows from 
 the Coifman--Meyer theorem on $\R^d$  (see \cite{CM} and the inequality~(1.1) in \cite{MS})
 and   the transference principle \cite[Theorem 3]{FS}.

\subsection{On discrete convolutions}

Next, we recall the following basic lemma on a discrete convolution.

\begin{lemma}\label{LEM:SUM}
\textup{(i)}
Let $d \geq 1$ and $\al, \be \in \R$ satisfy
\[ \al+ \be > d  \qquad \text{and}\qquad \al, \be < d.\]
\noi
Then, we have
\[
 \sum_{n = n_1 + n_2} \frac{1}{\jb{n_1}^\al \jb{n_2}^\be}
\les \jb{n}^{d - \al - \be}\]

\noi
for any $n \in \Z^d$.

\smallskip

\noi
\textup{(ii)}
Let $d \geq 1$ and $\al, \be \in \R$ satisfy $\al+ \be > d$.
\noi
Then, we have
\[
 \sum_{\substack{n = n_1 + n_2\\|n_1|\sim|n_2|}} \frac{1}{\jb{n_1}^\al \jb{n_2}^\be}
\les \jb{n}^{d - \al - \be}\]

\noi
for any $n \in \Z^d$.

\end{lemma}

Namely, in the resonant case (ii), we do not have the restriction $\al, \be < d$.
Lemma \ref{LEM:SUM} follows
from elementary  computations.
See, for example,  Lemmas 4.1 and 4.2 in \cite{MWX} for the proof.

\subsection{Strichartz estimates}
Given  $0 \leq s \leq 1$, 
we say that a pair $(q, r)$ is $s$-admissible 
(a pair $(\wt q, \wt r)$ is dual $s$-admissible,\footnote{Here, we define
the notion of dual $s$-admissibility for the convenience of the presentation.
Note that $(\wt q, \wt r)$ is dual $s$-admissible
if and only if $(\wt q', \wt r')$ is $(1-s)$-admissible.}
 respectively)
if $1 \leq \wt q < 2 < q \leq \infty$, 
 $1< \wt r \le 2 \leq r <\infty$, 
\begin{align*}
 \frac{1}{q} + \frac 3r  = \frac{3}{2}-  s =  \frac1{\wt q}+ \frac3{\wt r} -2, 
\qquad
\frac 1q + \frac{1}{r} \leq \frac 1 2, 
\qquad \text{and} 
\qquad  
\frac1{\wt q}+\frac1{\wt r} \geq \frac32   .
\end{align*}

\noi
We refer to the first two equalities as the scaling conditions
and the last two inequalities as the admissibility conditions.

We say that  $u$ is a solution to the following nonhomogeneous linear wave equation:
\begin{align}
\begin{cases}
(\dt^2 + 1-\Dl) u = f \\
( u, \dt u) |_{t = 0}=(u_0,  u_1) 
\end{cases}
\label{NLW1}
\end{align}

\noi
on a time interval containing $t= 0$, 
if $u$ satisfies the following Duhamel formulation:
\[ u =  \cos (t\jb{\nb}) u_0 +   \frac{\sin (t\jb{\nb})}{ \jb{\nb}} u_1 
+ \int_0^t  \frac{\sin ((t-t')\jb{\nb})}{ \jb{\nb}}   f(t') dt'. \]  

\noi
In the following, we often use
the following short-hand notation:
\[ \I(f)(t) =  \int_0^t  \frac{\sin ((t-t')\jb{\nb})}{ \jb{\nb}}   f(t') dt'. \]

\noi
We now recall the  Strichartz estimates
for solutions to the nonhomogeneous linear wave equation~\eqref{NLW1}.

\begin{lemma}\label{LEM:Str}
Given $0 \leq s \leq 1$,
let $(q, r)$ and $(\wt q,\wt r)$
be $s$-admissible and dual $s$-admissible pairs, respectively. 
Then, a solution $u$ to the nonhomogeneous linear wave equation \eqref{NLW1}
satisfies
\begin{align*}
\| (u, \dt u) \|_ {L^\infty_T \H^s_x } + 
 \| u  \|_{L^q_TL^r_x}
\lesssim 
\|(u_0, u_1) \|_{\H^s}  +  \| f \|_{L^{\wt q}_TL^{\wt r}_x}
\end{align*}

\noi
for all $0 < T \leq 1$. 
The following estimate also holds:
\begin{align*}
\| (u, \dt u) \|_ {L^\infty_T \H^s_x } 
+  \| u  \|_{L^q_TL^r_x}
\lesssim 
\|(u_0, u_1) \|_{\H^s}  +  \| f \|_{L^{1}_T H^{s-1}_x}
\end{align*}

\noi
for all $0 < T \leq 1$.
Here, we used a shorthand notation
$L^q_TL^r_x$ = $L^q([0, T]; L^r(\T^3))$, etc.
\end{lemma}

The Strichartz estimates on $\R^d$ have been studied extensively by many
mathematicians.  See \cite{GV, LS, KeelTao}
in the context of the wave equation.
For the Klein-Gordon equation under consideration, 
see \cite{KSV}.
Thanks to the finite speed of propagation, 
these estimates on $\T^3$ follow from the corresponding
estimates on $\R^3$.

In proving Theorem \ref{THM:1}, we use the fact that 
$\big(8, \frac{8}{3}\big)$
and $(4, 4)$ are $\frac 14$-admissible
and 
  $\frac 12$-admissible, respectively.
We also use
a dual $\frac12$-admissible pair  $\big(\frac 43, \frac 43\big)$.
In proving Theorem \ref{THM:weak}, 
we use 
$\big(\frac{4}{1+2\s}, \frac{4}{1-2\s}\big)$
and $\big(\frac{4}{3+8\s}, \frac{4}{3-4\s}\big)$
which are
$(\frac 12 + \s) $-admissible 
and dual $(\frac 12 + \s) $-admissible, respectively,
for small $\s > 0$.

\subsection{Tools from stochastic analysis}

We conclude this section by recalling useful lemmas
from stochastic analysis.
See \cite{Bog, Shige} for basic definitions.
Let $(H, B, \mu)$ be an abstract Wiener space.
Namely, $\mu$ is a Gaussian measure on a separable Banach space $B$
with $H \subset B$ as its Cameron-Martin space.
Given  a complete orthonormal system $\{e_j \}_{ j \in \N} \subset B^*$ of $H^* = H$, 
we  define a polynomial chaos of order
$k$ to be an element of the form $\prod_{j = 1}^\infty H_{k_j}(\jb{x, e_j})$, 
where $x \in B$, $k_j \ne 0$ for only finitely many $j$'s, $k= \sum_{j = 1}^\infty k_j$, 
$H_{k_j}$ is the Hermite polynomial of degree $k_j$, 
and $\jb{\cdot, \cdot} = \vphantom{|}_B \jb{\cdot, \cdot}_{B^*}$ denotes the $B$--$B^*$ duality pairing.
We then 
denote the closure  of 
polynomial chaoses of order $k$ 
under $L^2(B, \mu)$ by $\mathcal{H}_k$.
The elements in $\H_k$ 
are called homogeneous Wiener chaoses of order $k$.
We also set
\[ \H_{\leq k} = \bigoplus_{j = 0}^k \H_j\]

\noi
 for $k \in \N$.

Let $L = \Dl -x \cdot \nabla$ be 
 the Ornstein-Uhlenbeck operator.\footnote{For simplicity, 
 we write the definition of the Ornstein-Uhlenbeck operator $L$
 when $B = \R^d$.}
Then, 
it is known that 
any element in $\mathcal H_k$ 
is an eigenfunction of $L$ with eigenvalue $-k$.
Then, as a consequence
of the  hypercontractivity of the Ornstein-Uhlenbeck
semigroup $U(t) = e^{tL}$ due to Nelson \cite{Nelson2}, 
we have the following Wiener chaos estimate
\cite[Theorem~I.22]{Simon}.
See also \cite[Proposition~2.4]{TTz}.

\begin{lemma}\label{LEM:hyp}
Let $k \in \N$.
Then, we have
\begin{equation*}
\|X \|_{L^p(\O)} \leq (p-1)^\frac{k}{2} \|X\|_{L^2(\O)}
 \end{equation*}
 
 \noi
 for any $p \geq 2$
 and any $X \in \H_{\leq k}$.

\end{lemma}

The following lemma will be used in studying regularities of stochastic objects.
We say that a stochastic process $X:\R_+ \to \mathcal{D}'(\T^d)$
is spatially homogeneous  if  $\{X(\cdot, t)\}_{t\in \R_+}$
and $\{X(x_0 +\cdot\,, t)\}_{t\in \R_+}$ have the same law for any $x_0 \in \T^d$.
Given $h \in \R$, we define the difference operator $\dl_h$ by setting
\begin{align}
\dl_h X(t) = X(t+h) - X(t).
\label{diff1}
\end{align}

\begin{lemma}\label{LEM:reg}
Let $\{ X_N \}_{N \in \N}$ and $X$ be spatially homogeneous stochastic processes
$:\R_+ \to \mathcal{D}'(\T^d)$.
Suppose that there exists $k \in \N$ such that 
  $X_N(t)$ and $X(t)$ belong to $\H_{\leq k}$ for each $t \in \R_+$.

\smallskip
\noi\textup{(i)}
Let $t \in \R_+$.
If there exists $s_0 \in \R$ such that 
\begin{align*}
\E\big[ |\ft X(n, t)|^2\big]\les \jb{n}^{ - d - 2s_0}
\end{align*}

\noi
for any $n \in \Z^d$, then  
we have
$X(t) \in W^{s, \infty}(\T^d)$, $s < s_0$, 
almost surely.
Furthermore, if there exists $\g > 0$ such that 
\begin{align*}
\E\big[ |\ft X_N(n, t) - \ft X(n, t)|^2\big]\les N^{-\g} \jb{n}^{ - d - 2s_0}
\end{align*}

\noi
for any $n \in \Z^d$ and $N \geq 1$, 
then 
$X_N(t)$ converges to $X(t)$ in $W^{s, \infty}(\T^d)$, $s < s_0$, 
almost surely.

\noi

\smallskip
\noi\textup{(ii)}
Let $T > 0$ and suppose that \textup{(i)} holds on $[0, T]$.
If there exists $\s \in (0, 1)$ such that 
\begin{align*}
 \E\big[ |\dl_h \ft X(n, t)|^2\big]
 \les \jb{n}^{ - d - 2s_0+ \s}
|h|^\s, 
\end{align*}

\noi
for any  $n \in \Z^d$, $t \in [0, T]$, and $h \in [-1, 1]$,\footnote{We impose $h \geq - t$ such that $t + h \geq 0$.}
then we have 
$X \in C([0, T]; W^{s, \infty}(\T^d))$, 
$s < s_0 - \frac \s2$,  almost surely.
Furthermore, 
if there exists $\g > 0$ such that 
\begin{align*}
 \E\big[ |\dl_h \ft X_N(n, t) - \dl_h \ft X(n, t)|^2\big]
 \les N^{-\g}\jb{n}^{ - d - 2s_0+ \s}
|h|^\s, 
\end{align*}

\noi
for any  $n \in \Z^d$, $t \in [0, T]$,  $h \in [-1, 1]$, and $N \geq 1$, 
then 
$X_N$ converges to $X$ in $C([0, T]; W^{s, \infty}(\T^d))$, $s < s_0 - \frac{\s}{2}$,
almost surely.

\end{lemma}

Lemma \ref{LEM:reg} follows
from a straightforward application of the Wiener chaos estimate
(Lemma~\ref{LEM:hyp}).
For the proof, see Proposition 3.6 in \cite{MWX}
and  Appendix in \cite{OOTz}.
As compared to  Proposition~3.6 in \cite{MWX}, 
we made small adjustments.
In studying the time regularity, we 
made the following modifications:
$\jb{n}^{ - d - 2s_0+ 2\s}\mapsto\jb{n}^{ - d - 2s_0+ \s}$
and $s < s_0 - \s \mapsto s < s_0 - \frac \s2$  
so that it is suitable
for studying  the wave equation.
Moreover, while the result in \cite{MWX} is stated in terms of the
Besov-H\"older space $\mathcal{C}^s(\T^d) = B^s_{\infty, \infty}(\T^d)$, 
Lemma \ref{LEM:reg} handles the $L^\infty$-based Sobolev space $W^{s, \infty}(\T^3)$.
Note that 
the required modification of the proof is straightforward
since $W^{s, \infty}(\T^d)$ and $B^s_{\infty, \infty}(\T^d)$
differ only logarithmically:
\[ \| f \|_{W^{s, \infty}} \leq 
\sum_{j = 0}^\infty
 \|\P_j f \|_{W^{s, \infty}}
 \les \| f\|_{B^{s+\eps}_{\infty, \infty}}\]

\noi
for any $\eps > 0$.
For the proof of the almost sure convergence claims, 
see \cite{OOTz}.

Lastly, we recall the following Wick's theorem.
See Proposition I.2 in \cite{Simon}.

\begin{lemma}\label{LEM:Wick}	
Let $g_1, \dots, g_{2n}$ be \textup{(}not necessarily distinct\textup{)}
 real-valued jointly Gaussian random variables.
Then, we have
\[ \E\big[ g_1 \cdots g_{2n}\big]
= \sum  \prod_{k = 1}^n \E\big[g_{i_k} g_{j_k} \big], 
\]

\noi
where the sum is over all partitions of $\{1, \dots, 2 n\}$
into disjoint pairs $(i_k, j_k)$.
\end{lemma}

\section{On the stochastic terms, Part I}
\label{SEC:sto1}

In this and the next sections, 
we establish the regularity properties
of the stochastic objects 
$\<1>$, $\<20>$, and $\<21p>$ defined in \eqref{so3},
\eqref{so4}, and \eqref{soX}, respectively.
The following lemma establishes the regularity properties
of the stochastic convolution $\<1>$ and the Wick power $\<2>$.
See also the proof of Proposition 2.1 in \cite{GKO}.

\begin{lemma}\label{LEM:stoconv}
Let  $T >0$.

\smallskip

\noi
\textup{(i)}
For any $\eps > 0$, 
$ \<1>_N $  in \eqref{so4a} converges to $\<1>$
in $C([0,T];W^{- \frac 12 -\eps,\infty}(\T^3))$ almost surely.
 In particular, we have
  \[\<1> \in C([0,T];W^{- \frac 12 -\eps,\infty}(\T^3))
  \]
  
  \noi
  almost surely.

\smallskip

\noi
\textup{(ii)}
For any $\eps > 0$, 
$ \<2>_N $  in \eqref{so4b} converges to $\<2>$
 in 
$C([0,T];W^{- 1-\eps,\infty}(\T^3))$ almost surely.
 In particular, we have
 \[\<2> \in C([0,T];W^{-1-\eps,\infty}(\T^3))\]
 
 \noi
 almost surely.

\end{lemma}

\begin{proof}
(i) 
Let $t \geq 0$.
From \eqref{sconv1}, we have
\begin{align}
\ft{\<1>}(n, t)   = 
   \int_0^t \frac{\sin ((t - t') \jb{ n })}{\jb{ n }} d \beta_n (t')
\label{sconv2}
\end{align}

\noi
and thus
\begin{align}
\begin{split}
\E\big[|\ft{\<1>}(n, t)|^2\big] 
& = \s_n(t, t) 
= \frac{t}{2\jb{n}^2} - \frac{\sin(2t \jb{n})}{4\jb{n}^3}\\
& \leq C(t) \jb{n}^{-2}
\end{split}
\label{sconv3}
\end{align}

\noi
for any $n \in \Z^3$, where $\s_n(t, t)$ is defined in \eqref{sigma2}.
Hence from Lemma \ref{LEM:reg}, 
we conclude that 
 $\<1>(t) \in W^{-\frac12 -\eps, \infty}(\T^3)$
 almost surely for any $\eps >0$.

Let  $0 \leq t_2 \leq t_1$. 
From \eqref{sconv1},  we have
\begin{align}
\begin{split}
\ft{\<1>}(n, t_1) - \ft{\<1>}(n, t_2)
& = \int_{t_2}^{t_1}
\frac{\sin((t_1 - t')\jb{n})}{\jb{n}}d\be_n(t')\\
& \hphantom{X}
+ \int_0^{t_2} 
\frac{\sin((t_1 - t')\jb{n}) - \sin((t_2 - t')\jb{n})}{\jb{n}} d\be_n(t').
\end{split}
\label{V0}
\end{align}

\noi
Then, from the mean value theorem,  we have
\begin{align}
\begin{split}
\E\big[|\ft{\<1>}(n, t_1) - \ft{\<1>}(n, t_2)|^2\big] 
& \les \jb{n}^{-2}|t_1 - t_2|
+ t_2 \jb{n}^{-2 + \s} |t_1 - t_2|^\s \\
& \leq C(t_2) \jb{n}^{-2 + \s} |t_1 - t_2|^\s
\end{split}
\label{V0a}
\end{align}

\noi
for any $n \in \Z^3$,   $0 \leq t_2 \leq t_1$ with $t_1 - t_2 \leq 1$, and  $\s \in [0, 1]$.
Hence, from Lemma \ref{LEM:reg}, 
we conclude that
 $\<1> \in C(\R_+; W^{-\frac12 -\eps, \infty}(\T^3))$
 almost surely
  for any $\eps > 0$.

Proceeding as above, we have
\begin{align*}
\E\big[|\ft{\<1>}_M(n, t) - \ft{\<1>}_N(n, t)|^2\big] 
& \leq C(t) \ind_{|n|> N} \cdot \jb{n}^{-2}
 \leq C(t) N^{-\g} \jb{n}^{-2+\g}.
\end{align*}

\noi
for any $n \in \Z^3$, $M \geq N \geq 1$, 
and $\g \geq  0$.
Similarly, with $\dl_h$ as in \eqref{diff1}, we have
\begin{align*}
\E\big[|\dl_h \ft{\<1>}_M(n, t) - \dl_h \ft{\<1>}_N(n, t)|^2\big] 
& \les C(t) \ind_{|n|> N} \cdot \jb{n}^{-2 + \s} |h|^\s\\
& \les C(t) N^{-\g} \jb{n}^{-2 + \s+\g} |h|^\s
\end{align*}

\noi
for any $n \in \Z^3$, $M \geq N \geq 1$, $h \in [-1, 1]$, 
$\g\geq 0$,  and  $\s \in [0, 1]$.
Therefore, it follows from 
Lemma \ref{LEM:reg} that 
given $T>0$ and $\eps > 0$,  the truncated stochastic convolution
$\<1>_N$ converges to $\<1>$
in $C([0, T]; W^{-\frac12 -\eps, \infty}(\T^3))$ almost surely.

\smallskip

\noi
(ii) 
Proceeding as in Part (i), the main task is to estimate  
$\E\big[|\ft{\<2>}(n, t)|^2\big] $.
The following discussion holds for $\<2>_N$
with constants independent of $N \in \N\cup\{\infty\}$.
From \eqref{so4b} and~\eqref{sigma1}, we have
\begin{align}
\begin{split}
\E\big[|\ft{\<2>}(n, t)|^2\big] 
& = \sum_{n = n_1 + n_2} 
\sum_{n = n_1' + n_2'}
\E\bigg[ \Big(\ft{\<1>}(n_1, t)\ft{\<1>}(n_2, t)
- \ind_{n = 0} \cdot \E\big[|\ft{\<1>}(n_1, t)|^2\big]\Big)\\
& \hphantom{XXXXXXXXX}
\times
\cj{\Big(\ft{\<1>}(n_1', t)\ft{\<1>}(n_2', t)
- \ind_{n = 0} \cdot \E\big[|\ft{\<1>}(n_1', t)|^2\big]\Big)}\bigg].
\end{split}
\label{V1}
\end{align}

\noi
In order to have non-zero contribution in \eqref{V1}, 
we must have~$n_1 = n_1'$ and $n_2 = n_2'$ up to permutation.
Thus, with \eqref{sconv1} and Lemma \ref{LEM:SUM}, we have
\begin{align}
\E\big[|\ft{\<2>}(n, t)|^2\big] 
& \les t^2 \sum_{n = n_1 + n_2} 
\frac{1}{\jb{n_1}^2 \jb{n_2}^2}
\les t^2 \jb{n}^{-1}.
\label{V2}
\end{align}

\noi
Hence from Lemma \ref{LEM:reg}, 
we conclude that
 $\<2>(t) \in W^{-1 -\eps, \infty}(\T^3)$
 almost surely
 for any $\eps > 0$.
A similar argument shows 
that  $\<2> \in C([0, T]; W^{-1 -\eps, \infty}(\T^3))$
 almost surely
and that 
$\<2>_N$
 convergences to $\<2>$
in 
$C([0, T]; W^{-1 -\eps, \infty}(\T^3))$ almost surely.
\end{proof}

\begin{remark}\rm
As we saw in the proof of Lemma \ref{LEM:stoconv}\,(i), 
once we establish regularity properties of a given stochastic object $\tau$, 
then a slight modification of the argument
shows convergence of the truncated stochastic objects
$\tau_N$ to $\tau$.
Hence, in the following, we only establish claimed regularity properties
of given stochastic terms.

\end{remark}

Next, we study the regularity of $\<20>$.
As pointed in the introduction, 
a naive parabolic thinking
would give a regularity of $0- = (-\frac 12-) + (-\frac 12-) + 1$, 
where one degree of smoothing comes from the Duhamel integral operator $\I$.
By exploiting multilinear dispersive effect, 
we show that there is in fact an extra $\frac12$-smoothing.

\begin{proof}[Proof of Proposition~\ref{PROP:sto1}]

By definition  $\<20> = \I(\<2>)$, we have
\begin{align}
\ft{\<20>}(n, t)
=  \int_0^t \frac{\sin((t - t') \jb{n})}{\jb{n}}  \ft{\<2>}(n, t')dt'
\label{S2}
\end{align}

\noi
and thus  we have 
\begin{align*}
\ft{\dt \<20>}(n, t)
=  \int_0^t \cos((t - t') \jb{n})  \ft{\<2>}(n, t')dt'.
\end{align*}

\noi
Then, from (the proof of) Lemma \ref{LEM:stoconv} (ii), 
we conclude that 
\[\dt \<20> \in  C([0,T];W^{-1 -\eps,\infty}(\T^3))\]

\noi
almost surely for any $\eps >0$.

In the following, 
we focus on proving that
$\<20> \in C([0,T];W^{\frac 12 -\eps,\infty}(\T^3))$
almost surely.
In view of Lemma \ref{LEM:reg}, it suffices to show that 
there exists small $\s \in (0, 1)$ such that 
\begin{align}
\E \big[ |\ft{\<20>}(n, t)|^2\big] & \le C(T) \jb{n}^{-4+},
\label{S3a}\\
\E \big[ |\ft{\<20>}(n, t_1) - \ft{\<20>}(n, t_2)|^2\big] 
& \le C(T) \jb{n}^{-4 + \s + }|t_1 - t_2|^\s  
\label{S3b}
\end{align}

\noi
for any $n \in \Z^3$ and  $0 \leq t, t_1, t_2 \leq T$
with $0 < |t_1 - t_2| < 1$.

\smallskip

We first prove \eqref{S3a} in the following.
From \eqref{S2}, 
we have 
\begin{align}
\E \big[ |\ft{\<20>}(n, t)|^2\big]
& = 
 \int_0^t \frac{\sin((t - t_1) \jb{n})}{\jb{n}} 
 \int_0^t \frac{\sin((t - t_2) \jb{n})}{\jb{n}} 
\E\Big[ \ft{\<2>}(n, t_1) \cj{\ft{\<2>}(n, t_2)} \Big]  dt_2dt_1.
\label{S4}
\end{align}

\noi
When $n = 0$, it follows from 
\eqref{so4b} with \eqref{sigma1} and \eqref{sigma2} that, we have
\begin{align*}
\E \big[ |\ft{\<20>}(0, t)|^2\big]
& = 
 \int_0^t \sin (t - t_1) 
 \int_0^t \sin (t-t_2) 
\E\big[ \ft{\<2>}(0, t_1) \cj{\ft{\<2>}(0, t_2)} \big]  dt_2dt_1\notag\\
& = 
 \int_0^t \sin (t - t_1) 
 \int_0^t \sin (t-t_2) \notag \\
& \hphantom{XXX}
\times
\sum_{k_1, k_2 \in \Z^3}
\E\Big[\big( |\ft{\<1>}(k_1, t_1)|^2 -  \s_{k_1}(t_1, t_1)\big)\big( |\ft{\<1>}(k_2, t_2)|^2 -  \s_{k_2}(t_2, t_2)\big)\Big] 
 dt_2dt_1\notag\\
& \leq C(T) \sum_{k \in \Z^3}\frac{1}{\jb{k}^4}
 \leq C(T),
\end{align*}

\noi
where  $\s_{k_j}(t_j,t_j)$ is as in~\eqref{sigma2}. 
In the last step,  we used
\begin{align}
\begin{split}
\E& \Big[\big(  |\ft{\<1>}(k_1, t_1)|^2 -  \s_{k_1}(t_1, t_1)\big)\big( |\ft{\<1>}(k_2, t_2)|^2 -  \s_{k_2}(t_2, t_2)\big)\Big] 
 =  \ind_{k_1=\pm k_2} \cdot  \s_{k_1}(t_1,t_2)^2.
\end{split}
\label{eq:wick-squares-corr}
\end{align}

\noi
The identity \eqref{eq:wick-squares-corr}
follows from Wick's theorem (Lemma \ref{LEM:Wick}).
This proves \eqref{S3a} when $n = 0$.

In the following, we assume $n \ne 0$.
By expanding 
$\ft{\<2>}(n, t_1)$ and $\ft{\<2>}(n, t_2)$
as in  \eqref{V1}
with $n = n_1 + n_2$ 
for $\ft{\<2>}(n, t_1)$
and $n = n_1' + n_2'$ for $\ft{\<2>}(n, t_2)$, 
we see that we must have~$n_1 = n_1'$ and $n_2 = n_2'$ up to permutation
in order to have non-zero contribution in \eqref{S4}.
Without loss of generality, assume that $0 \le t_2 \le t_1 \leq t$.
Then, we have
\begin{align}
\E \big[ |\ft{\<20>}(n, t)|^2\big]
& = 4  \sum_{\substack{n = n_1 + n_2\\n_1 \ne \pm n_2}} \int_0^{t} \frac{\sin((t - t_1) \jb{n})}{\jb{n}} 
 \int_0^{t_1} \frac{\sin((t - t_2) \jb{n})}{\jb{n}} 
\s_{n_1}(t_1, t_2) \s_{n_2}(t_1, t_2)  dt_2dt_1\notag\\
& \hphantom{X}
+ 2 \cdot \ind_{n \in 2 \Z^3\setminus\{0\}} 
 \int_0^t \frac{\sin((t - t_1) \jb{n})}{\jb{n}} 
 \int_0^{t_1} \frac{\sin((t - t_2) \jb{n})}{\jb{n}} 
\notag\\
& \hphantom{XXXXXXX}
\times \E\Big[ \ft{\<1>}\big(\tfrac{n}{2}, t_1\big)^2 \, \cj{\ft{\<1>}\big(\tfrac{n}{2}, t_2\big)^2 } \Big]  dt_2dt_1
\notag\\
& =: 
\1 (n, t)+ \II(n, t), 
\label{S5}
\end{align}

\noi
where $\s_{n_j}(t_1, t_2)$ is as in \eqref{sigma2}
and $\II(n, t)$ denotes the contribution from $n_1 = n_2 = n_1' = n_2' = \frac n2$.

We first estimate the second term $\II(n, t)$ in \eqref{S5}.
By Wick's theorem (Lemma \ref{LEM:Wick}) 
with~\eqref{sigma2}, 
we have
\begin{align*}
 \bigg|\E\Big[ \ft{\<1>}\big(\tfrac{n}{2}, t_1\big)^2 & \, \cj{\ft{\<1>}\big(\tfrac{n}{2}, t_2\big)^2 } \Big] \bigg|
\le C(T) \jb{n}^{-4}
\end{align*}

\noi
under $0 \le t_2 \le t_1 \leq t \leq T$.
Hence, from \eqref{S5},  we conclude that 
\[ |\II(n, t)|\le C(T) \jb{n}^{-6},\]

\noi
verifying \eqref{S3a}.

In the following, we estimate $\1(n, t)$ in \eqref{S5}:
\begin{align}
\begin{split}
\1 (n, t) & = 
 - \sum_{k_1, k_2 \in \{1, 2\}} \sum_{\eps_1, \eps_2 \in \{-1, 1\}}
\frac{\eps_1\eps_2e^{i (\eps_1+\eps_2) t\jb{n}}}{\jb{n}^2}
\sum_{\substack{n = n_1 + n_2\\n_1 \ne \pm n_2}}
\int_0^{t} 
e^{-i \eps_1t_1\jb{n}} \\ 
& \hphantom{XX}
\times  \int_0^{t_1} e^{-i \eps_2t_2\jb{n}}
\prod_{j = 1}^2\s_{n_j}^{(k_j)}(t_1, t_2)
\,   dt_2dt_1 =: \sum_{k_1, k_2 \in \{1, 2\}} \1^{(k_1,k_2)} (n, t),
\end{split}
\label{S6a}
\end{align}

\noi
where 
$\s_{n}^{(1)}(t_1, t_2)$
and $\s_{n}^{(2)}(t_1, t_2)$ are defined by 
\begin{align}
\begin{split}
\s_{n}^{(1)}(t_1, t_2) &  :=  
 \frac{\cos((t_1 - t_2)\jb{n}) }{2 \jb{n}^2} t_2, 
\\ 
\s_{n}^{(2)}(t_1, t_2)  & := 
  \frac{\sin ((t_1-t_2)  \jb{n})}{4\jb{n}^3 }
-  \frac{\sin ((t_1 +  t_2) \jb{n})}{4\jb{n}^3 }
\end{split}
\label{S6}
\end{align}

\noi
such that 
$ \s_{n}(t_1, t_2) =  \s_{n}^{(1)}(t_1, t_2) + \s_{n}^{(2)}(t_1, t_2)$.
If $|n_1| \sim 1$ or $|n_2| \sim1$, then 
 from \eqref{S6}
 with $\jb{n_1} \jb{n_2} \ges \jb{n}$, we easily obtain 
\begin{align}
|\1 (n, t)| \le C(T) \jb{n}^{-4+},
\label{S7}
\end{align}

\noi
satisfying \eqref{S3a}.
Hence, 
we assume  $|n_1|, |n_2|\gg 1$
in the following.
 By Lemma \ref{LEM:SUM}
 with \eqref{S6}, 
we can easily bound 
the contribution to $\1(n, t)$ in~\eqref{S6a}
from  $(k_1, k_2) \ne (1, 1)$
and obtain for them the decay required in~\eqref{S7}.

In  the following, we  estimate the worst contribution to $\1(n, t)$ coming from 
 $(k_1, k_2) = (1, 1)$:
\begin{align*}
\begin{split}
\1^{(1,1)} (n, t)
& : = -\frac{1}{16} \sum_{\eps_1, \eps_2, \eps_3, \eps_4 \in \{-1, 1\}}
\sum_{\substack{n = n_1 + n_2\\n_1 \ne \pm n_2}} 
\frac{\eps_1\eps_2e^{i (\eps_1+\eps_2) t\jb{n}}}{\jb{n}^2\jb{n_1}^2 \jb{n_2}^2}\\
& \hphantom{XXXX}\times
\int_0^{t} 
e^{-i t_1 \kk_1( \bar n)}
 \int_0^{t_1} 
t_2^2   e^{-i t_2\kk_2(\bar n)}
 \, dt_2dt_1, 
\end{split}
\end{align*}

\noi
where $\kk_1( \bar n)$ and $\kk_2( \bar n)$ are defined by 
\begin{align*}
\begin{split}
\kk_1(\bar n) & = 
 \eps_1 \jb{n} -\eps_3 \jb{n_1} - \eps_4\jb{n_2}, \\ 
\kk_2(\bar n) & = 
 \eps_2 \jb{n} +\eps_3 \jb{n_1} + \eps_4\jb{n_2}.
\end{split}
\end{align*}

\noi
When $|n| \les 1$, \eqref{S7} trivially holds.
Hence, we assume $|n|\gg 1$. 
We have to carefully  estimate the different 
contributions coming from the various combinations of  $\bar \eps = (\eps_1, \eps_2, \eps_3, \eps_4)$
by  exploiting either (i) the dispersion (= oscillation) 
or (ii) smallness of the measure of the relevant frequency set.

Fix our choice of $\bar \eps = (\eps_1, \eps_2, \eps_3, \eps_4)$ 
and denote by $\1^{(1,1)}_{\bar \eps}(n,t)$ the associated contribution to $\1^{(1,1)}(n,t)$.
By switching the order of integration and first integrating in $t_1$, we have 
\begin{align*}
\bigg|\int_0^{t} 
& e^{-i t_1 \kk_1( \bar n)}
 \int_0^{t_1} 
t_2^2   e^{-i t_2\kk_2(\bar n)}
 \, dt_2dt_1\bigg| \\
& = \bigg| \int_0^{t} 
t_2^2   e^{-i t_2\kk_2(\bar n)}
\frac{e^{-i t \kk_1( \bar n)} - e^{-i t_2 \kk_1( \bar n)}}{-i \kk_1(\bar n)}
 \, dt_2\bigg| \le C(T) (1+ |\kk_1(\bar n)|)^{-1}.
\end{align*}

\noi
Thus,  we have 
\begin{align}
|\1^{(1,1)}_{\bar \eps} (n, t)|
& \le C(T)   
\sum_{n = n_1 + n_2} 
\frac{1}{\jb{n}^2\jb{n_1}^2 \jb{n_2}^2(1+ |\kk_1(\bar n)|)}.
\label{SS1}
\end{align}
Without loss of generality, by symmetry we can assume $|n_1|\geq |n_2|$ in the following when estimating
the sum on the right-hand side.

\smallskip

\noi
$\bullet$ {\bf Case 1:}
$(\eps_1, \eps_3, \eps_4) = (\pm 1, \mp1, \mp1)$ or
$(\pm1, \mp1, \pm1)$.

In this case, we have 
$|\kk_1(\bar n)| \geq \jb{n}$.
Then, from Lemma \ref{LEM:SUM}, we obtain
\begin{align*}
|\1^{(1,1)}_{\bar \eps} (n, t)|
& \le C(T)   \jb{n}^{-4}.
\end{align*}

\noi
This proves \eqref{S3a}.

\medskip

\noi
$\bullet$ {\bf Case 2:}
$(\eps_1, \eps_3, \eps_4) = (\pm 1, \pm1, \mp1)$.

In this case, we have
$|\kk_1(\bar n)|=  \jb{n} + \jb{n_2} -\jb{n_1}$.
Under $n = n_1 + n_2$
and  $|n_1| \geq |n_2|$, we have 
\begin{align}
\jb{n_1} \sim \jb{n}+ \jb{n_2}.
\label{SS1a}
\end{align}
Under $ n = n_1 + n_2$, 
three vectors $n$, $n_1$, and $n_2$ form a triangle,
where we view $n_1$ as a vector based at $n_2$.
Then, 
by the law of cosines, we have
\begin{align}
|n|^2 + |n_2|^2 - |n_1|^2 = 2 |n| |n_2| \cos \big( \angle(n, n_2)\big),  
\label{SS1b}
\end{align}

\noi
where
$\angle(n, n_2)$ denotes the angle between $n$ and $n_2$.
Then, from \eqref{SS1a} and \eqref{SS1b}, 
we have
\begin{align}
\begin{split}
|\kk_1(\bar n)| 
& =  \frac{(\jb{n} + \jb{n_2})^2 -\jb{n_1}^2}
{\jb{n} + \jb{n_2} +\jb{n_1}}
  =  \frac{2\jb{n}\jb{n_2} + |n|^2 + |n_2|^2 - |n_1|^2 + 1}
{\jb{n} + \jb{n_2} +\jb{n_1}}\\
& 
\ges
 \frac{|n| |n_2| (1 - \cos \theta)}
{\jb{n_1}} 
\end{split}
\label{SS2}
\end{align}

\noi
where $\ta = \angle(n_2, -n) \in [0, \pi]$ is the angle between $n_2$ and $-n$.

\smallskip
\noi
{\bf \underline{Subcase 2.i:}}
We first consider the case $1 - \cos \ta \ges 1$.
See Figure~\ref{FIG:2}.
In this case, from \eqref{SS1} and \eqref{SS2}
with Lemma \ref{LEM:SUM}, 
we have 
\begin{align}
\begin{split}
|\1^{(1,1)}_{\bar \eps} (n, t)|
& \le C(T)   
\sum_{n = n_1 + n_2} 
\frac{1}{\jb{n}^{3}\jb{n_1} \jb{n_2}^{3}}\\
& \le C(T)   
\jb{n}^{-4+}, 
\end{split}
\label{SS3}
\end{align}

\noi
yielding \eqref{S3a}.

\begin{figure}[h]
\begin{tikzpicture}

\draw[ decoration={markings,mark=at position 1 with {\arrow[scale=2]{>}}},
    postaction={decorate},
    shorten >=0.4pt] (0,0) node[dot]{}  --(1, 2) 
node[pos=.5,  left] {$n_2$}
;

\draw[ decoration={markings,mark=at position 1 with {\arrow[scale=2]{>}}},
    postaction={decorate},
    shorten >=0.4pt] (1,2) -- (4, 2) 
node[pos=.5, above] {$n_1$}
;

\draw[ decoration={markings,mark=at position 1 with {\arrow[scale=2]{>}}},
    postaction={decorate},
    shorten >=0.4pt] (0,0) node[below]{$O$}-- (4, 2) 
node[pos=0.5, below right] {$n = n_1 + n_2~~$};

\draw[densely dashed](0,0) --(-1.5, -0.75) ;

\draw (66:0.5) 
 arc (60:210:0.5)
node[pos=0.5, left ] {$\ta$};


%

\end{tikzpicture}

\caption{A typical configuration in Subcase 2.i}
\label{FIG:2}
\end{figure}
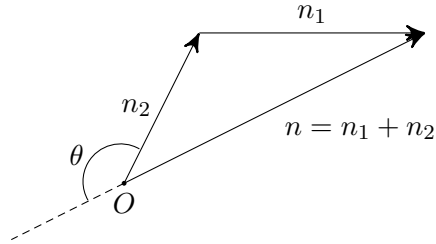

\smallskip
\noi
{\bf \underline{Subcase 2.ii:}}
Next, we  consider the case $1 - \cos \ta \ll1$.
In this case, we have $0\leq \ta \ll 1$, namely, 
$n$ and $n_2$ point in almost opposite directions.
In particular,  we have $1 - \cos \ta \sim \ta^2 \ll 1$.
By dyadically decomposing $n_2$ into $|n_2| \sim N_2$ for dyadic  $N_2 \geq 1$, 
we see that 
for fixed $n \in \Z^3$, 
 the range of possible $n_2$ with $|n_2| \sim N_2$
is constrained to a cone $\mathcal{C}$
whose height is $\sim N_2 \cos \ta \sim N_2$
and  the base disc of radius $\sim N_2 \sin \ta \sim N_2 \ta$
with the direction of  the central axis of the cone  given by $-n$.
Hence, we have $\text{vol}(\mathcal{C}) \sim N_2^3 \ta^2$.
See Figure~\ref{FIG:3}.
Then,  from \eqref{SS1} and \eqref{SS2}
with $|n_1| \ges \max(|n|, |n_2|)$, 
we have 
\begin{align}
\begin{split}
|\1^{(1,1)}_{\bar \eps} (n, t)|
& \le C(T)   
\sum_{\substack{N_2 \geq 1\\\text{dyadic}} }
\frac{1}{\jb{n}^{3}\max(\jb{n}, N_2) N_2^{3} \ta^{2}} N_2^3 \ta^2 \\
& \le C(T)   
\jb{n}^{-4 +}, 
\end{split}
\label{SS4}
\end{align}

\noi
yielding \eqref{S3a}.

\begin{figure}[h]
\begin{tikzpicture}

\draw[ decoration={markings,mark=at position 1 with {\arrow[scale=2]{>}}},
    postaction={decorate},
    shorten >=0.4pt] (0,0) node[below]{$O$}-- (2, 0) 
node[pos=0.5, below right] {$n$};

\draw(0,0) --(-2, -0.2);

\path[fill=gray!40] (0,0) -- (-2,0.2) -- (-2,-0.2) -- cycle;

\draw[fill=gray!40] (-2, 0) ellipse (0.1 and 0.2);

\draw[ decoration={markings,mark=at position 1 with {\arrow[scale=2]{>}}},
    postaction={decorate},
    shorten >=0.4pt] (0,0) node[dot]{}  --(-2, 0.2) 
node[pos=.5,  above ] {$n_2$};

\draw[->] (-1.5, -0.5) -- (-1.7, 0)
node[pos=0,  below ] {$\mathcal{C}$};
 ;

\draw[densely dashed](-2,0) --(-3, 0) ;

\end{tikzpicture}

\caption{A typical configuration in Subcase  2.ii.
Here, we omit the vector $n_1$.}
\label{FIG:3}
\end{figure}
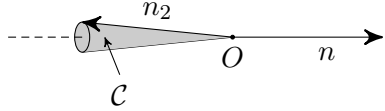

\medskip

\noi
$\bullet$ {\bf Case 3:}
$(\eps_1, \eps_3, \eps_4) = (\pm 1, \pm1, \pm1)$.

In this case, we have
$|\kk_1(\bar n)|=  \jb{n_1} + \jb{n_2} -\jb{n}$.
By the law of cosines, we
have
\begin{align}
|n_1|^2 + |n_2|^2 - |n|^2 = 2 |n_1| |n_2| \cos \big( \angle(-n_1, n_2)\big).
\label{SS4a}
\end{align}

\noi
Then, by proceeding as in Case 2
with \eqref{SS1a} and \eqref{SS4a}, 
we have
\begin{align}
|\kk_1(\bar n)| 
& =  \frac{(\jb{n_1} + \jb{n_2})^2 -\jb{n}^2}
{\jb{n_1} + \jb{n_2} +\jb{n}}
\ges
 \frac{|n_1| |n_2| (1 - \cos \theta)}
{\jb{n_1}} 
\label{SS5}
\end{align}

\noi
where $\ta = \angle(n_1, n_2) \in [0, \pi]$ is the angle between $n_1$ and $n_2$.
When  $1 - \cos \ta \ges 1$, 
we can proceed as in~\eqref{SS3}.
Next, consider the case $1 - \cos \ta \sim \ta^2 \ll 1$.
Since $n = n_1 + n_2$, we see that the angle
$\angle (n, n_2)$
between $n$ and $n_2$ is smaller than 
$\ta = \angle(n_1, n_2)$
in this case.\footnote{Form a triangle
with three vectors $n$, $n_1$, and $n_2$ with $n$ and $n_2$ sharing a common base point
such that $n = n_1 + n_2$.
Then, the angle $\angle (n_1, n_2)$ is an exterior angle to this triangle
and thus we have 
$\angle (n_1, n_2) = \angle (n, n_2) + \angle (-n, -n_1) > \angle (n, n_2)$.}
Hence, with $|n_1|\ges |n|$, 
we can repeat the computation in \eqref{SS4} and obtain the same bound.
This concludes the proof of \eqref{S3a} by choosing $\dl>0$ sufficiently small.

\smallskip

Next, we briefly discuss the difference estimate \eqref{S3b}.
Let $0 \leq t_2 \leq t_1 \leq T$.
We need to estimate
\begin{align}
\E \big[ |\ft{\<20>}(n, t_1) - \ft{\<20>}(n, t_2)|^2\big] 
& = \E \Big[ (\ft{\<20>}(n, t_1) - \ft{\<20>}(n, t_2))\cj{\ft{\<20>}(n, t_1)}\Big]\notag\\
& \hphantom{X}
- \E \Big[ (\ft{\<20>}(n, t_1) - \ft{\<20>}(n, t_2))\cj{\ft{\<20>}(n, t_2)}\Big].
\label{S12}
\end{align}

\noi
From \eqref{S2}, we have
\begin{align}
\ft{\<20>}(n,  t_1) - \ft{\<20>}(n,  t_2)
& =  \int_{t_2}^{t_1} \frac{\sin((t_1 - t') \jb{n})}{\jb{n}}  \ft{\<2>}(n, t')dt'\notag\\
& \hphantom{X}
+ \int_0^{t_2} \frac{\sin((t_1 - t') \jb{n}) - \sin((t_2 - t') \jb{n})}{\jb{n}}  \ft{\<2>}(n, t')dt'.
\label{S13}
\end{align}

\noi
We crudely 
estimate \eqref{S12}
%
by using 
\eqref{S13}, \eqref{V2}, and   
the mean value theorem to control the difference.
As a result, 
we have
\begin{align}
\E \big[ |\ft{\<20>}(n, t_1) - \ft{\<20>}(n, t_2)|^2\big] 
& \les C(T) \jb{n}^{-2} |t_1 - t_2|.
\label{S14}
\end{align}

\noi
By interpolating \eqref{S3a} and \eqref{S14}, 
we obtain \eqref{S3b} for some small $\s \in (0, 1)$.

This completes the proof of Proposition \ref{PROP:sto1}.
\end{proof}

\begin{remark} \label{REM:res}\rm
In Cases 2 and 3, 
we separately estimated
the contributions from (i) $1 - \cos \ta \ges 1$
and (ii) $1 - \cos \ta \ll 1$.
Note that these correspond to
the time non-resonant and (nearly) time resonant case
in the dispersive PDE terminology.
In the time resonant case (ii), 
there was no gain from time integration
and thus we needed to exploit the smallness of the set
(i.e.~the cone $\mathcal{C}$)
for the time resonant case
by observing that the time resonance is caused by 
the parallel interaction of waves, 
i.e. $n$, $n_1$, and $n_2$  (close to) being parallel.
A need for such a geometric consideration is
one difference between  the study of dispersive equations
from that of parabolic equations.

\end{remark}

\section{On the stochastic terms, Part II}
\label{SEC:sto2}

In this section, we study 
 the regularity property of  the resonant product 
$\<21p>$ defined in~\eqref{soX}
(Proposition~\ref{PROP:sto2}).
From \eqref{para1} and the definition of the Littlewood-Paley projector $\F(\P_j f)(n) = \varphi_j(n) \ft f(n)$, we have
 \begin{align*}
 \begin{split}
\ft{\<21p>} (n, t)
& = \sum_{\substack{n = n_1 + n_2+n_3\\ n_1 + n_2 \ne 0}} 
\sum_{|j - k| \leq 2} \varphi_j(n_1 + n_2) \varphi_k(n_3) \\
& \hphantom{XXX}
\times
\int_0^t  \frac{\sin ((t-t') \jb{ n_1 + n_2 })}{\jb{ n_1 + n_2 }} 
\ft{\<1>}(n_1, t')\ft{\<1>}(n_2, t') dt' \cdot  \ft{\<1>}(n_3, t) \\
& \hphantom{X} 
+ \sum_{n_1\in \Z^3}
\sum_{| k| \leq 2}  \varphi_k(n_3)
 \int_0^t \sin (t - t')\cdot \big(  |\ft{\<1>}(n_1, t')|^2  -  \s_{n_1}(t')\big) 
dt' \cdot  \ft{\<1>}(n, t)\\
& =: \ft \RR_1(n, t) + \ft \RR_2(n, t).
 \end{split}
 \end{align*}

\noi
For simplicity of notation, however, 
we write 
 \begin{align}
\begin{split}
\ft \RR_1(n, t)& = \sum_{\substack{n = n_1 + n_2+n_3\\|n_1 + n_2| \sim |n_3|\\ n_1 + n_2 \ne 0}} 
\int_0^t  \frac{\sin ((t-t') \jb{ n_1 + n_2 })}{\jb{ n_1 + n_2 }} 
\ft{\<1>}(n_1, t')\ft{\<1>}(n_2, t') dt' \cdot  \ft{\<1>}(n_3, t), \\
\ft \RR_2(n, t)
& =  \sum_{n_1\in \Z^3} \ind_{|n|\sim 1} \int_0^t \sin (t - t')\cdot \big(  |\ft{\<1>}(n_1, t')|^2  -  \s_{n_1}(t')\big) 
dt' \cdot  \ft{\<1>}(n, t), 
\end{split}
\label{Y1}
 \end{align}

\noi
where  the conditions $|n_1 + n_2| \sim |n_3|$ in the first term
and $|n|\sim 1$ in the second term signify the resonant product $\pe$.
The second term $\RR_2$ in  \eqref{Y1}
corresponds to the contribution
from $n_1 + n_2 = 0$ and is already renormalized from 
the  Wick renormalization: $\<1>^2 \rightsquigarrow \<2>$.
Using~\eqref{eq:wick-squares-corr} and Lemma \ref{LEM:reg}, 
it is easy to see that $\RR_2 \in C(\R_+; C^\infty(\T^3))$
almost surely, 
since $|n| \sim 1$.

In the following, our main goal is to show
 \begin{align}
\E\big[|\ft \RR_1(n, t)  |^2\big]
&  \leq C(t) \jb{n}^{-3+}.
\label{Y2}
  \end{align}

\noi
Then,  Lemma \ref{LEM:reg} allows us to conclude that 
 $\RR_1(t) \in W^{0-, \infty}(\T^3)$ almost surely.
We  decompose $\RR_1$ as 
 \begin{align}
\ft \RR_1(n, t)
& = \sum_{\substack{n = n_1 + n_2+n_3\\|n_1 + n_2| \sim |n_3|\\ (n_1 + n_2)(n_2 + n_3) (n_3 + n_1) \ne 0}} 
\int_0^t  \frac{\sin ((t-t') \jb{ n_1 + n_2 })}{\jb{ n_1 + n_2 }} 
\ft{\<1>}(n_1, t')\ft{\<1>}(n_2, t') dt' \cdot  \ft{\<1>}(n_3, t) \notag\\
& \hphantom{X} 
+ 2 
\int_0^t \ft{\<1>}(n, t') 
\bigg[\sum_{\substack{n_2 \in \Z^3\\|n_2| \sim |n+n_2|\ne 0}}  \frac{\sin ((t-t') \jb{ n + n_2 })}{\jb{ n + n_2 }} \notag\\
& \hphantom{XXXXXXXXXlllllXX}
\times 
\Big(\ft{\<1>}(n_2, t')\ft{\<1>}(-n_2, t) - \s_{n_2}(t,  t')\Big)\bigg] dt'    \notag\\
& \hphantom{X} 
+ 2 
\int_0^t \ft{\<1>}(n, t') 
\bigg[\sum_{\substack{n_2 \in \Z^3\\| n_2| \sim |n+n_2| \ne 0 }} \frac{\sin ((t-t') \jb{ n + n_2 })}{\jb{ n + n_2 }} 
 \s_{n_2}(t, t')\bigg] dt'   \notag\\
& \hphantom{X} 
- \ind_{n\ne0}
\int_0^t  \frac{\sin ((t-t') \jb{ 2n })}{\jb{ 2n }} 
\, (\ft{\<1>}(n, t'))^2dt' \cdot \ft{\<1>}(-n, t)\notag\\
&
=: \ft \RR_{11}(n, t)+  \ft \RR_{12}(n, t)+  \ft \RR_{13}(n, t)+  \ft \RR_{14}(n, t).
\label{Y3}
  \end{align}

\noi
Here, the second term $\RR_{12}$ corresponds to the ``renormalized'' contribution
from $n_1 + n_3= 0$ or $n_2 + n_3= 0$,
while the fourth term corresponds
to the contribution
from $n_1 = n_2 = n = - n_3$.

From \eqref{sconv2}, we have
 \begin{align*}
\E\big[|\ft \RR_{14}(n, t)  |^2\big]
&  \le C(t) \jb{n}^{-8},
  \end{align*}

\noi
satisfying \eqref{Y2}.
Under $|n+ n_2| \sim|n_2|$, we have $|n_2| \ges |n|$.
Then, using a variant of~\eqref{eq:wick-squares-corr},
we obtain 
 \begin{align*}
\E\big[|\ft \RR_{12}(n, t)  |^2\big]
&  \le C(t)
\sum_{\substack{n_2 \in \Z^3\\|n + n_2| \sim |n_2|}}
\frac{1}{\jb{n}^2 \jb{n_2}^6}
 \les \jb{n}^{-5},
\end{align*}

\noi
satisfying \eqref{Y2}.

Given $n \in \Z^3$, define $\NR(n)$ by 
\begin{align*}
\NR(n) 
= \big\{(n_1, n_2, n_3)\in \Z^3: \,
& n = n_1 + n_2+n_3, \ |n_1 + n_2| \sim |n_3|, \\
&  (n_1 + n_2)(n_2 + n_3) (n_3 + n_1) \ne 0 \big\}.
\end{align*}

\noi
Then, with a shorthand notation $n_{ij} = n_i + n_j$, we have
 \begin{align*}
\E\big[|\ft \RR_{11} & (n, t)  |^2\big]\\
& = \E\Bigg[\sum_{(n_1, n_2, n_3) \in \NR(n)}
\int_0^t  \frac{\sin ((t-t_1) \jb{ n_{12} })}{\jb{ n_{12} }} 
\ft{\<1>}(n_1, t_1)\ft{\<1>}(n_2, t_1) dt' \cdot  \ft{\<1>}(n_3, t)\\
& \hphantom{X}
\times  \sum_{(n_1', n_2', n_3') \in \NR(n)}
\int_0^t  \frac{\sin ((t-t_2) \jb{ n_{12}'  })}{\jb{ n_{12}' }} 
\cj{\ft{\<1>}(n_1', t_2)\ft{\<1>}(n_2', t_2)} dt' \cdot  \cj{\ft{\<1>}(n_3', t)}\Bigg].
  \end{align*}

\noi
In order to compute the expectation above, 
we need to take all possible pairings between $(n_1, n_2, n_3)$
and $(n_1', n_2', n_3')$.
By Jensen's inequality, however, 
we see that it suffices to consider the case
$n_j = n_j'$, $j = 1, 2, 3$.
See the discussion on $\<31p>$ in Section 4 of \cite{MWX}.
See also Section~10 in~\cite{Hairer}. 
Hence, by Wick's theorem and  \eqref{S4}, we have
 \begin{align*}
\E\big[|\ft \RR_{11}(n, t)  |^2\big]
&  \les 
\sum_{\substack{n = m + n_3\\ |m|\sim|n_3|}}
\E \big[ |\ft{\<20>}(m, t)|^2\big]
\E \big[ |\ft{\<1>}(n_3, t)|^2\big]\\
\intertext{From \eqref{sconv3}, \eqref{S3a},  and Lemma \ref{LEM:SUM} (ii),}
& \le C(t)
\sum_{\substack{n = m + n_3\\ |m|\sim|n_3|}} \frac{1}{\jb{m}^{4-} \jb{n_3}^2}
\le C(t) \jb{n}^{-3+}, 
  \end{align*}
verifying \eqref{Y2}.
Note that in evaluating the last sum, we crucially used the fact that the product 
is a resonant product.

It remains to study the third term $\ft\RR_{13}$ 
on the right-hand side of  \eqref{Y3}.
Let  $0 \leq t_2 \leq t_1 \leq T$.
Then, from \eqref{Y3} with \eqref{sigma2},  we have
\begin{align*}
\begin{split}
\E\big[|\ft \RR_{13}(n, t)|^2\big]
& =8
\sum_{k_0, k_1, k_2 \in \{1, 2\}}
\int_0^t 
\int_0^{t_1} 
\s_{n}^{(k_0)}(t_1, t_2)
 \\
& \hphantom{XX}
\times 
\bigg[\sum_{\substack{n_2 \in \Z^3\\| n_2| \sim |n+n_2| \ne0}} \frac{\sin ((t-t_1) \jb{ n + n_2 })}{\jb{ n + n_2 }} 
\s_{n_2}^{(k_1)}(t, t_1)\bigg]  \\
& \hphantom{XX}
\times 
\bigg[\sum_{\substack{n'_2 \in \Z^3\\| n'_2| \sim |n+n'_2| \ne0}} \frac{\sin ((t-t_2) \jb{ n + n'_2 })}{\jb{ n + n'_2 }} 
\s_{n_2'}^{(k_2)}(t, t_2)\bigg]
   dt_2 dt_1 \\ & =: \sum_{k_0, k_1, k_2 \in \{1, 2\}} \1^{(k_0,k_1,k_2)}(n,t) , 
\end{split}
  \end{align*}

\noi
where  $\s_{n}(t, t' )   = \s_{n}^{(1)}(t, t' )  + \s_{n}^{(2)}(t, t' )  $
as in \eqref{S6}.
In the following, we only consider the contribution
from 
  $(k_0, k_1, k_2) = (1, 1, 1)$
  since the other cases follow in a similar (but easier) manner.

From  \eqref{S6}, 
we have 
\begin{align}
\1^{(1,1,1)}& (n,t)
 = 
\int_0^t 
\int_0^{t_1} 
 \frac{ \cos((t_1 - t_2)\jb{n}) }{ \jb{n}^2} t_2
\notag  \\
& \hphantom{XXXX}
\times 
\bigg[\sum_{\substack{n_2 \in \Z^3\\| n_2| \sim |n+n_2|\ne0}} \frac{\sin ((t-t_1) \jb{ n + n_2 })}{\jb{ n + n_2 }} 
 \frac{ \cos((t - t_1)\jb{n_2}) }{ \jb{n_2}^2} t_1 \bigg] \notag  \\
& \hphantom{XXXX}
\times 
\bigg[\sum_{\substack{n'_2 \in \Z^3\\| n'_2| \sim |n+n'_2|\ne0}} \frac{\sin ((t-t_2) \jb{ n + n'_2 })}{\jb{ n + n'_2 }} 
 \frac{ \cos((t - t_2)\jb{n'_2}) }{ \jb{n'_2}^2} t_2 \bigg]
   dt_2 dt_1\notag  \\
& =  -\frac{1}{32}\sum_{\substack{\eps_j \in \{-1, 1\}\\j = 1, \dots, 5}}
\sum_{\substack{n_2 \in \Z^3\\| n_2| \sim |n+n_2|\ne 0}} 
\sum_{\substack{n'_2 \in \Z^3\\| n'_2| \sim |n+n'_2| \ne0}}
\frac{
\eps_1\eps_2e^{i  t (\eps_1 \jb{n+n_2} 
+   \eps_2  \jb{n+n_2'}
+ \eps_3  \jb{n_2} +  \eps_4  \jb{n_2'})}
}{\jb{n}^2 \jb{n+n_2} \jb{n_2}^2 \jb{n+n_2'} \jb{n_2'}^2}
\notag  \\
& \hphantom{XXXX}
\times  \int_0^t 
t_1  e^{ - i t_1  \kk_3(\bar n)}
 \int_0^{t_1} 
t_2^2  e^{- i t_2\kk_4(\bar n')} 
dt_2 dt_1,
\label{Y5}  
\end{align}

\noi
where
$\kk_3(\bar n)$ and $\kk_3(\bar n)$ are defined by 
\begin{align}
\begin{split}
\kk_3(\bar n) 
& = \eps_1 \jb{n+n_2}  
+ \eps_3  \jb{n_2}
 - \eps_5  \jb{n}, \\
\kk_4(\bar n') 
& = 
 \eps_2  \jb{n+n_2'}
+  \eps_4  \jb{n_2'}
+  \eps_5  \jb{n}.
\end{split}
\label{Y6}
\end{align}

\noi
Under the constraint $|n+ n_2 |\sim |n_2|$
and $|n+ n_2'| \sim |n_2'|$, 
we have $|n_2 |, |n_2'| \ges |n|$.
In the following, we also assume 
$|n_2| \ges |n_2'|$.

We   decompose $\1^{(1,1,1)} (n,t)$ according to the value of
 $\bar \eps = (\eps_1, \dots, \eps_5)\in\{\pm 1\}^5$ and write
\[
\1^{(1,1,1)} (n,t) =:\sum_{\bar\eps \in \{\pm 1\}^5 } \1^{(1,1,1)}_{\bar \eps} (n,t).
\]

\noi
In the following, we study $\1^{(1,1,1)}_{\bar \eps} $
for each fixed $\bar \eps\in\{\pm 1\}^5$.
Note that the sum over $n_2$ and $n_2'$ in~\eqref{Y5} are not absolutely convergent
at a first glance.
In many cases, we make use of dispersion (i.e.~time oscillation)
and show that they are indeed absolutely convergent.
In Case 3 below, however, there is a subcase,
where we show that the sum is only conditionally convergent.
In this case, it is understood that the sum is first studied
under the constraint $|n_2|, |n_2'| \leq N$ for some $N \geq 1$
and that the sum remains bounded in taking a limit $N \to \infty$.
We do not mention this procedure in an explicit manner
in the following.

By first  integrating~\eqref{Y5} in $t_1$
when  $|\kk_3(\bar n)|\geq 1$
and simply bounding the integral in~\eqref{Y5} by $C(T)$
when  $|\kk_3(\bar n)|<  1$, 
we have
\begin{align}
\begin{split}
|\1^{(1,1,1)}_{\bar \eps} (n,t)|
& \le \frac{C(T)}{\jb{n}^2}
\sum_{\substack{n_2 \in \Z^3\\|n + n_2| \sim |n_2|}} 
\frac{1}
{ \jb{n+n_2} \jb{n_2}^2 (1+ |\kk_3(\bar n)|)}\\
& \hphantom{X}
\times
\sum_{\substack{n'_2 \in \Z^3\\|n + n'_2| \sim |n'_2|}}
\frac{\ind_{\{|n_2| \ges |n_2'|\}}}
{ \jb{n+n_2'} \jb{n_2'}^2}.
\end{split}
\label{Y7}
\end{align}

\noi
$\bullet$ {\bf Case 1:}
 $(\eps_1, \eps_3, \eps_5) = (\pm 1, \pm 1, \mp1)$.
\quad  
In this case, it follows from \eqref{Y6} that  $|\kk_3(\bar n) |\ges \jb{n_2}$.
By writing $(1+ |\kk_3(\bar n)|)^{-1} \les \jb{n}^{-1+2\dl} \jb{n_2}^{-\dl}
\jb{n_2'}^{-\dl}$ in \eqref{Y7} for sufficiently small  $\dl > 0$ 
and applying Lemma~\ref{LEM:SUM}, 
we obtain
\begin{align}
|\1^{(1,1,1)}_{\bar \eps} (n,t)|
\le C(T) \jb{n}^{-3}.
\label{Y10}
\end{align}

\noi
$\bullet$ {\bf Case 2:} $(\eps_1, \eps_3, \eps_5) = (\pm 1, \pm 1, \pm1)$.
\quad 
If   $|n + n_2|\sim |n_2| \gg |n|$, 
then we have $|\kk_3(\bar n) |\ges \jb{n_2}$
and hence \eqref{Y10} holds as above.
Otherwise, we have 
$|n + n_2|\sim |n_2| \sim |n|$.
In this case, 
by  $\jb{n}^{-2\dl} \les \jb{n_2}^{-\dl}\jb{n_2'}^{-\dl}$
for $\dl > 0$.
Then, we have
\begin{align}
|\1^{(1,1,1)}_{\bar \eps} (n,t)|
& \le \frac{C(T)}{\jb{n}^{2-2\dl}}
\sum_{\substack{n_2 \in \Z^3\\|n + n_2| \sim |n_2|}} 
\frac{1}
{ \jb{n+n_2} \jb{n_2}^{2+\dl} (1+ |\kk_3(\bar n)|)} \notag \\
& \hphantom{X}
\times
\sum_{\substack{n'_2 \in \Z^3\\|n + n'_2| \sim |n'_2|}}
\frac{\ind_{\{|n_2| \ges |n_2'|\}}}
{ \jb{n+n_2'} \jb{n_2'}^{2+\dl}}  \notag \\
& \le \frac{C(T)}{\jb{n}^{2-\dl}}
\sum_{\substack{n_2 \in \Z^3\\|n + n_2| \sim |n_2|}} 
\frac{1}
{ \jb{n+n_2} \jb{n_2}^{2+\dl} (1+ |\kk_3(\bar n)|)}.
\label{Y10a}
\end{align}

\noi
We can now proceed as in Case 3 of the proof of Proposition \ref{PROP:sto1}
by replacing $(n, n_1, n_2)$ with $(n, n+n_2, -n_2)$.
In particular, from \eqref{SS5}, we have 
\begin{align}
|\kk_3(\bar n)| 
\ges
 |n_2| (1 - \cos \theta)
\label{Y10b}
\end{align}

\noi
where $\ta = \angle(n + n_2, - n_2) \in [0, \pi]$ is the angle between $n + n_2$ and $-n_2$.
When  $1 - \cos \ta \ges 1$, 
by summing over $n_2$ in \eqref{Y10a}
with \eqref{Y10b} and Lemma \ref{LEM:SUM}, 
we obtain \eqref{Y10}.

Next, consider the case  $1 - \cos \ta \sim \ta^2 \ll 1$.
Since $n = (n + n_2) + (- n_2)$, we see that the angle
$\theta_0 = \angle (n, - n_2)$
between $n$ and $- n_2$ is smaller than 
$\ta = \angle(n+n_2, -n_2)$
in this case.
Moreover, 
we see that 
for fixed $n \in \Z^3$, 
 the range of possible $- n_2$ with $|n_2| \sim N_2$, dyadic $N_2\geq 1$, 
is constrained to a cone
whose height is $\sim N_2 |\cos \ta_0| \sim N_2$
and the base disc of radius $\sim N_2 \sin \ta_0 \les N_2 \ta$. 
Then,  from \eqref{Y10a} and \eqref{Y10b}
with $|n + n_2|\sim |n_2|\sim|n|$, 
we have 
\begin{align*}
\begin{split}
|\1^{(1,1,1)}_{\bar \eps} (n,t)|
& \le \frac{C(T)}{\jb{n}^{2-\dl}}
\sum_{\substack{N_2 \sim \jb{n} \\\text{dyadic}} }
\frac{1}{N_2^{4+\dl} \, \ta^{2}} N_2^3 \ta^2 \\
& \le C(T)   
\jb{n}^{-3 +\dl}, 
\end{split}
\end{align*}

\noi
yielding \eqref{Y10}.

\medskip

\noi
$\bullet$ {\bf Case 3:}
 $\eps_1 = - \eps_3$.
\quad  
First, suppose that $|n| \geq |n_2|^{\g}$ for some small $\g>0$
(to be chosen later).
Then, 
with  $\jb{n}^{-2\dl} \les \jb{n_2}^{-\g \dl}\jb{n_2'}^{-\g \dl}$
for $\ld > 0$, 
we can proceed as in Case 2 
(but using the computation in Case 2 of the proof of Proposition \ref{PROP:sto1}
by replacing $(n, n_1, n_2)$ with $(n, n+n_2, -n_2)$ or $(n, -n_2, n+n_2)$)
and obtain
\begin{align*}
|\1^{(1,1,1)}_{\bar \eps} (n,t)|
& \le \frac{C(T)}{\jb{n}^{2-2\dl}}
\sum_{\substack{n_2 \in \Z^3\\|n + n_2| \sim |n_2|}} 
\frac{1}
{ \jb{n+n_2} \jb{n_2}^{2+\g\dl} (1+ |\kk_3(\bar n)|)} \notag \\
& \hphantom{X}
\times
\sum_{\substack{n'_2 \in \Z^3\\|n + n'_2| \sim |n'_2|}}
\frac{\ind_{\{|n_2| \ges |n_2'|\}}}
{ \jb{n+n_2'} \jb{n_2'}^{2+\g\dl}}  \notag \\
& \le C(T) \jb{n}^{-3 + (2-\g)\dl}, 
\end{align*}

\noi
verifying \eqref{Y2} by choosing $\dl > 0$ sufficiently small.

Next, we consider the case 
$|n| \ll |n_2|^{\g}$.
In this case, we are not able to prove absolute summability in~\eqref{Y7} 
since  $\kk_3(\bar n)$ does not have any  good lower bound, 
and thus we need to proceed more carefully.
By writing out the contribution from the sum over $n_2$ in \eqref{Y5}
(namely, ignoring the sum over $n_2'$), we have
\begin{align}
 \int_0^t  t_1  e^{ i t_1 \eps_5\jb{n}}
\sum_{\substack{n_2 \in \Z^3\\|n + n_2| \sim |n_2|\\|n| \ll |n_2|^{\g}}} 
\frac{\sin ((t - t_1) (\jb{n+ n_2} - \jb{n_2}) )}{\jb{n}^2 \jb{n+n_2} \jb{n_2}^2 }
 dt_1.
\label{Y11a}  
\end{align}

\noi
By going back to the definition \eqref{para1}
of the resonant product $\pe$, 
we can write down the  sum over $n_2$ in \eqref{Y11a} 
as 
\begin{align}
\sum_{j \in \N_0}
\sum_{\substack{n_2 \in \Z^3\\|n| \ll |n_2|^\g}}\varphi_j(n_2)
\sum_{|k - j|\leq 2}\varphi_k(n + n_2), 
\label{Y12}
\end{align}

\noi
where $\varphi_j$ is as in \eqref{phi1}.
Thanks to the restriction $|n| \ll |n_2|^\g$ with small $\g > 0$, 
the sum in \eqref{Y12} is in fact given by 
\begin{align}
\sum_{j \in \N_0}
\sum_{\substack{n_2 \in \Z^3\\|n| \ll |n_2|^\g}}\varphi_j(n_2).
\label{Y13}
\end{align}

\noi
While the sum over $n_2$ in \eqref{Y11a} is not absolutely convergent, 
we do not expect to gain anything from the time integration in $t_1$
in this case
due to the lack of a good lower bound on $\kk_3(\bar n)$.
The reduction to \eqref{Y13}, however, 
allows us to exploit 
the symmetry $n_2 \leftrightarrow -n_2$
and the oscillatory nature of the sine kernel in \eqref{Y11a}.

By the Taylor remainder theorem, we have
\begin{align}
\Theta^\pm(n, n_2)
:=  \jb{n \pm n_2 }- \jb{n_2}
 \mp \frac{\jb{n, n_2}}{\jb{n_2}}
 = O\bigg(\frac{\jb{n}^2}{\jb{n_2}}\bigg), 
\label{Y14}
\end{align}

\noi
where $\jb{n, n_2} = \jb{n, n_2}_{\R^3}$
denotes the standard inner product on $\R^3$.
Let $\Ld$ be the index set $\text{``}\sim \Z^3 / 2\text{''}$
in \eqref{index}.
Then, with \eqref{Y14} and the mean value theorem, we have
\begin{align}
\begin{split}
\eqref{Y11a} &
 =   
 \int_0^t  t_1  e^{ i t_1 \eps_5\jb{n}}
\sum_{j \in \N_0}
\sum_{\substack{n_2 \in \Z^3\\|n| \ll |n_2|^\g}}\varphi_j(n_2)
\frac{\sin ((t - t_1) (\jb{n+ n_2} - \jb{n_2}) )}{\jb{n}^2 \jb{n+n_2} \jb{n_2}^2 }
 dt_1\\
& =    \int_0^t  t_1  e^{ i t_1 \eps_5\jb{n}}
\sum_{j \in \N_0}
\sum_{\substack{n_2 \in \Ld\\|n| \ll |n_2|^\g}}\frac{\varphi_j(n_2)}
{\jb{n}^2 \jb{n+n_2} \jb{n_2}^2 }\\
& \hphantom{X}
\times
\Big\{ \sin ((t - t_1) (\jb{n+ n_2} - \jb{n_2}) )
+ \sin ((t - t_1) (\jb{n- n_2} - \jb{n_2}) )\Big\}
 dt_1\\
& =    \int_0^t  t_1  e^{ i t_1 \eps_5\jb{n}}
\sum_{j \in \N_0}
\sum_{\substack{n_2 \in \Ld\\|n| \ll |n_2|^\g}}\frac{\varphi_j(n_2)}{\jb{n}^2 \jb{n+n_2} \jb{n_2}^2 }\\
& \hphantom{X}
\times
   \bigg\{
\sin \bigg((t - t_1)\Big(  \frac{\jb{n, n_2}}{\jb{n_2}} + \Theta^+(n, n_2)\Big)\bigg) \\
& \hphantom{XXXXX}
- \sin \bigg((t - t_1)\Big(  \frac{\jb{n, n_2}}{\jb{n_2}} - \Theta^-(n, n_2)\Big)\bigg)
 \bigg\} dt_1
\\
 & \leq C(T) 
\sum_{j \in \N_0}
\sum_{\substack{n_2 \in \Ld\\|n| \ll |n_2|^\g}}\frac{\varphi_j(n_2)}{\jb{n}^2 \jb{n+n_2} \jb{n_2}^2 }
\frac{\jb{n}^{4\dl}}{\jb{n_2}^{2\dl}}\\
 & \leq \frac{C(T)}{\jb{n_2'}^{\dl}}
\sum_{\substack{n_2 \in \Ld\\|n| \ll |n_2|^\g}}\frac{1}{\jb{n}^{2-4\dl} \jb{n_2}^{3+\dl} }
\end{split}
\label{Y14a}
\end{align}

\noi
for any $\dl\in  [0, \frac 12]$.
Fix small $\dl > 0$. 
By applying  Lemma \ref{LEM:SUM} to sum over $n_2$
and $n_2'$
and then using the condition $|n| \ll |n_2|^\g$, we obtain
\begin{align*}
|\1^{(1,1,1)}_{\bar \eps} (n,t)|
& \le \frac{C(T)}{\jb{n}^{2-3\dl}} \cdot \frac{1}{\jb{n}^{\frac{\dl}{\g}}}
\le C(T){\jb{n}^{-3}}
\end{align*}

\noi
by choosing $\g = \g(\dl)> 0$ sufficiently small.
This proves \eqref{Y2}.

Next, we briefly discuss the difference estimate.
In view of Lemma \ref{LEM:reg}, we need to show that there exists $\s \in (0, 1)$ such that 
\begin{align}
\E \big[ |\ft \RR_1 (n, t_1) - \ft \RR_1 (n, t_2)|^2\big] 
& \le C(T) \jb{n}^{-3 + \s + }|t_1 - t_2|^\s  
\label{Y15}
\end{align}

\noi
for any $n \in \Z^3$, $0 \leq t, t_1, t_2 \leq T$
with $0 < |t_1 - t_2| < 1$.
As in  \eqref{S12}, 
we need to estimate
\begin{align}
\E \big[ |\ft \RR_1(n, t_1) - \ft \RR_1(n, t_2)|^2\big] 
& = \sum_{j = 1}^2 (-1)^{j+1}  \E \big[ (\ft \RR_1(n, t_1) - \ft \RR_1(n, t_2))\cj{\ft{\<21p>}(n, t_j)}\big].
\label{Y16}
\end{align}

\noi
As for  $\RR_{11}$, $\RR_{12}$, and $\RR_{14}$ in \eqref{Y3}, 
we can crudely estimate them and obtain
\begin{align}
\E \big[ |\ft \RR_{1j} (n, t_1) - \ft  \RR_{1j} (n, t_2)|^2\big] 
& \les C(T) \jb{n}^{-2+} |t_1 - t_2|,
\label{Y17}
\end{align}
for $j = 1, 2, 4$, since the relevant summations are absolutely convergent.
Then, \eqref{Y15} follows from interpolating \eqref{Y17}
and
\begin{align*}
\E \big[ |\ft \RR_{1j} (n, t_1) - \ft  \RR_{1j} (n, t_2)|^2\big] 
& \les C(T) \jb{n}^{-3+} .
\end{align*}

It remains to discuss $\RR_{13}$.
In view of \eqref{Y5} and \eqref{Y16}, we need to consider an expression like
\begin{align*}
\int_0^{t_2} 
\int_0^{\tau_1} 
&  \frac{ \cos((\tau_1 - \tau_2)\jb{n}) }{ \jb{n}^2}  \tau_2
\bigg[\sum_{\substack{n_2 \in \Z^3\\|n + n_2| \sim |n_2|}} 
\frac{\sin ((t-\tau_1) \jb{ n + n_2 })}{\jb{ n + n_2 }} 
 \frac{ \cos((t - \tau_1)\jb{n_2}) }{ \jb{n_2}^2} \tau_1 \bigg] \bigg|_{t = t_2}^{t_1} \notag  \\
&
\times 
\bigg[\sum_{\substack{n'_2 \in \Z^3\\|n + n'_2| \sim |n'_2|}} 
\frac{\sin ((t_j-\tau_2) \jb{ n + n'_2 })}{\jb{ n + n'_2 }} 
 \frac{ \cos((t_j  - \tau_2)\jb{n'_2}) }{ \jb{n'_2}^2} \tau_2 \bigg]
   d\tau_2 d\tau_1
\end{align*}

\noi
for $0 \leq \tau_2 \leq \tau_1 \leq T$.
Then, by repeating the computations in Cases 1 - 3 above
and applying the mean value theorem, 
we directly obtain \eqref{Y15}.
Note that some of the relevant summations are not absolutely convergent in this case
and hence we can not proceed with a crude estimate and 
interpolation.  Compare this with the parabolic setting.
See Section 5 in \cite{MWX}.

This completes the proof of Proposition \ref{PROP:sto2}.

\section{Paracontrolled operators}
\label{SEC:po}

We first present the proof of 
Lemma \ref{LEM:IV} on the regularity of $Z = \big(S(t) (X_0,   X_1)\big) \pe \<1> $.

\begin{proof}[Proof of Lemma \ref{LEM:IV}]
Let $H(t) = S(t) (X_0, X_1)$.
Under
$n = n_1 + n_2$ and $|n_1|\sim |n_2|$, 
we have $\jb{n} \les \jb{n_1} \sim \jb{n_2}$.
Then, it follows from   Minkowski's integral inequality, 
the Wiener chaos estimate (Lemma~\ref{LEM:hyp}), 
independence of $\ft{\<1>}(n_2, t)$, and~\eqref{sconv3}
that for any $p \geq 2$,  we have
\begin{align}
\Big\|\| Z (t) \|_{H^s} \Big\|_{L^p(\O)}
& \le \Big\| \| \jb{n}^{s} \F_x (H \pe \<1>) (n, t) \|_{L^p(\O)} \Big\|_{\l^2_n}\notag\\
& \leq p^\frac{1}{2}
 \Big\| \| \jb{n}^{s} \F_x (H \pe \<1>) (n, t) \|_{L^2(\O)} \Big\|_{\l^2_n}\notag\\
&  \sim p^\frac{1}{2}  
\bigg(\sum_{n \in \Z^3} 
\jb{n}^{2s}
\sum_{\substack{n = n_1 + n_2\\|n_1|\sim |n_2|}}
 |\ft H (n_1, t)|^2 
\E\big[|\ft{\<1>}(n_2, t)|^2 \big] \bigg)^\frac{1}{2}\notag\\
&  \les   p^\frac{1}{2}
\bigg( \sum_{n \in \Z^3} 
\jb{n}^{2(s - s_1 - 1)}
\sum_{n_1 \in \Z^3}
 \jb{n_1}^{2s_1} |\ft H (n_1, t)|^2 \bigg)^\frac{1}{2}
\notag\\
& \les  p^\frac{1}{2}
\| (X_0, X_1) \|_{\H^{s_1}}
\label{XX1}
\end{align}

\noi
provided that $s < s_1 - \frac 12$.
Fix $\eps > 0$ small.
Then, 
by writing
\[(H \pe \<1>) (t_1)- (H \pe \<1>) (t_2)
= H(t_1) \pe \big(\<1> (t_1) - \<1>(t_2)\big)
+ \big(H(t_1) - H(t_2) \big) \pe \<1> (t_2)\]

\noi
for $0 \leq t_2 \leq t_1$, 
we can repeat the computation in \eqref{XX1}.
In particular, by 
 \eqref{V0a} and the mean value theorem, we obtain
\begin{align*}
\Big\|\| (H \pe \<1>) (t_1)  - & (H \pe \<1>) (t_2) \|_{H^s} \Big\|_{L^p(\O)}\\
& \les  C(t_2) p^\frac{1}{2} |t_1 - t_2|^\frac{\eps}{2}
\| (X_0, X_1) \|_{\H^{s_1}}
+ p^\frac{1}{2} \| H (t_1) - H (t_2) \|_{H^{s_1-\frac{\eps}{2}}}\\
& \les  C(t_2) p^\frac{1}{2} |t_1 - t_2|^\frac{\eps}{2}
\| (X_0, X_1) \|_{\H^{s_1}}, 
\end{align*}

\noi
provided that $s < s_1 - \frac 12 - \frac{\eps}{2}$.
Therefore, by taking large $p = p(\eps) \gg1 $, 
we conclude from Kolmogorov's continuity criterion
(\cite[Theorem 8.2]{Bass})
that 
$Z$ belongs to 
$C([0,T];H^{ s_1 - \frac 12 -\eps}(\T^3))$
almost surely.
\end{proof}
	
%
%
%
%
%
%
%
%

The remaining part of this section is devoted 
to studying the mapping properties
of 
the paracontrolled operators
$ \If_{\pl}^{(1)}$ in \eqref{X3}
and $\If_{\pl, \pe}$ in \eqref{X6}.

We first study the regularity property of the paracontrolled operator
$\If_{\pl}^{(1)}$ defined in~\eqref{X3}.
By writing out the frequency relation  $|n_2|^\theta \les |n_1| \ll |n_2|$
more carefully, we have
\begin{align}
\begin{split}
 \If_{\pl}^{(1)} (w) (t)
 &   =  \sum_{n \in \Z^3}
e_n  \sum_{n =  n_1 +  n_2}
\sum_{ \ta k + c_0 \leq j < k-2}
\varphi_j(n_1) \varphi_k(n_2)  \\
& \hphantom{XXXXX}
\times 
\int_0^t \frac{\sin ((t - t') \jb{n})}{\jb{n}} 
\ft w(n_1, t')\,  \ft{ \<1>}(n_2, t') dt', 
\end{split}
\label{XX2a}
\end{align}

\noi
where $c_0 \in \R$ is some fixed constant.
In the following, we establish
the mapping property of 
$\If_{\pl}^{(1)}$ in a deterministic manner
by using a pathwise regularity of the stochastic convolution~$\<1>$.

Given $\Xi  \in C(\R_+; W^{-\frac 12 -\eps}(\T^3))$ for some small $\eps > 0$, 
define a paracontrolled operator  $\If_{\pl}^{(1), \Xi}$ by 
\begin{align}
\begin{split}
 \If_{\pl}^{(1), \Xi} (w) (t)
 &  : =  \sum_{n \in \Z^3}
e_n  \sum_{n =  n_1 +  n_2}
\sum_{ \ta k + c_0 \leq j < k-2}
\varphi_j(n_1) \varphi_k(n_2)  \\
& \hphantom{XXXXX}
\times 
\int_0^t \frac{\sin ((t - t') \jb{n})}{\jb{n}} 
\ft w(n_1, t')\,  \ft{ \Xi}(n_2, t') dt'.
\end{split}
\label{XX2}
\end{align}

\noi
Note that we have $\If_{\pl}^{(1)} = \If_{\pl}^{(1), \<1>}$, 
i.e.~with $\Xi = \<1>$. 

\begin{lemma}\label{LEM:sto3}
Let  $0 < s_1 < \frac 12$ and $T > 0$.
Then, given small $\theta > 0$, 
there exists small $\eps = \eps(s_1, \ta) > 0$
such that 
given $\Xi  \in C(\R_+; W^{-\frac 12 -\eps, \infty}(\T^3))$, 
the paracontrolled operator $ \If_{\pl}^{(1), \Xi}$ defined in \eqref{XX2} belongs to the class: 
\begin{align}
 \L_2 = \L( C([0, T]; H^{s_1}(\T^3) )
 \, ; \, 
C([0, T]; H^{\frac 12 + 2\eps}(\T^3) ).
\label{L1}
\end{align}

\end{lemma}

As a direct consequence of Lemma \ref{LEM:sto3}
with Lemma \ref{LEM:stoconv}, 
we obtain the following corollary for the paracontrolled operator
$\If_{\pl}^{(1)}$ defined in \eqref{X3} (and \eqref{XX2a}).

\begin{corollary}\label{COR:sto3a}
Let  $0 < s_1 < \frac 12$ and $T > 0$.
Then, given small $\theta > 0$, 
there exists small $\eps= \eps(s_1, \ta)  > 0$
such that 
the paracontrolled operator $ \If_{\pl}^{(1)}$ defined in \eqref{X3} belongs 
to $\L_2$ in \eqref{L1}
almost surely.
Moreover, by letting
$ \If_{\pl}^{(1), N}$, $N \in \N$, denote the paracontrolled operator
in \eqref{X3} with $\<1>$ replaced by the truncated stochastic convolution $\<1>_N$ in \eqref{so4a}, 
the truncated paracontrolled operator $ \If_{\pl}^{(1), N}$ converges almost surely to $ \If_{\pl}^{(1)}$ 
in $\L_2$.

\end{corollary}

\begin{proof}[Proof of Lemma \ref{LEM:sto3}]
Let $s_1 >0$.
Under $|n_2|^\theta \les |n_1| \ll |n_2|$ with $n = n_1 + n_2$,
we have
\begin{align}
\jb{n}^{\frac 12+2\eps} \frac{1}{\jb{n}} 
\les \jb{n_1}^{\frac{4\eps}{\theta}}\jb{n_2}^{-\frac{1}{2}-2\eps}
\les \jb{n_1}^{s_1-\eps}\jb{n_2}^{-\frac{1}{2}-2\eps}
\label{A0}
\end{align}

\noi
by choosing $\eps = \eps (s_1, \theta) > 0$ sufficiently small.

Let $\ft w_{j}(n_1, t')
= \varphi_j(n_1)\ft w(n_1, t')$
and 
$ \ft{\Xi}_{k}(n_2, t')
=  \varphi_k(n_2)  \ft{\Xi}(n_2, t')$.
Then, 
from \eqref{XX2}
and~\eqref{A0}, 
we have
\begin{align*}
\| \If_{\pl}^{(1), \Xi}(w)(t) \|_{H^{\frac12+2\eps}}
& \les \int_0^t 
\sum_{j, k =0}^\infty
2^{(s_1-\eps)j} 2^{(-\frac{1}{2}-2\eps)k}
\bigg\|  \sum_{n =  n_1 +  n_2}
\ft w_{j}(n_1, t')
\ft{ \Xi}_{k}(n_2, t')\bigg\|_{\l^2_n} dt'\\
& \les \int_0^t 
\sum_{j, k =0}^\infty
2^{(s_1-\eps)j} 2^{(-\frac{1}{2}-2\eps)k}
\| w_{j}(t')\|_{L^{2}_x}
\| \Xi_{k}( t')\|_{L^\infty_x} dt'.
\end{align*}

\noi
Then, by summing over dyadic blocks and applying the trivial embedding \eqref{embed}, 
we obtain
\begin{align*}
\| \If_{\pl}^{(1), \Xi}(w)(t) \|_{H^{\frac12+2\eps}}
& \les T \|w\|_{L^\infty_T H_x^{s_1}}
\|\Xi\|_{L^\infty_T (B_{\infty, 1}^{-\frac 12 -2 \eps})_x}\\
& \les T \|w\|_{L^\infty_T H_x^{s_1}}
\|\Xi\|_{L^\infty_T W_x^{-\frac 12 - \eps, \infty}}
\end{align*}

\noi
for any $t \in [0, T]$.
The continuity in time of $ \If_{\pl}^{(1), \Xi}(w)$
follows
from modifying the computation above as in the previous subsections.
We omit the details.
\end{proof}

Finally, we present the proof of Proposition \ref{PROP:sto4}
on the  paracontrolled operator
$\If_{\pl, \pe}$ in~\eqref{X6}.
By writing out the frequency relations
more carefully as in \eqref{XX2a}, we have
\begin{align}
\begin{split}
\If_{\pl, \pe}(w) (t)
&  = \sum_{n \in \Z^3}e_n 
    \int_0^{t} 
\sum_{j = 0}^\infty
\sum_{n_1 \in \Z^3}
\varphi_j(n_1)
\ft w(n_1, t') \A_{n, n_1} (t, t') dt', 
\end{split}
\label{A0a}
\end{align}

\noi
where $\A_{n, n_1} (t, t')$ is given by 
\begin{align}
\begin{split}
\A_{n, n_1} (t, t')
& = \ind_{[0 , t]}(t')
\sum_{\substack{k = 0\\ j \leq \ta k + c_0}}^\infty
\sum_{\substack{\l, m = 0\\|\l-m|\leq 2}}^\infty
\sum_{n - n_1 =  n_2 + n_3} 
 \varphi_k(n_2)  
 \varphi_\l(n_1 + n_2) 
 \varphi_m(n_3) \\
& \hphantom{XXXX}
\times  \frac{\sin ((t - t') \jb{n_1 + n_2 })}{\jb{ n_1 + n_2 }} 
   \ft{\<1>}(n_2, t')\,   \ft{\<1>}(n_3, t)  .
\end{split}
\label{A0b}
\end{align}

\noi
For ease of notation, however, 
we simply use \eqref{X6} and \eqref{X7}
in the following, 
with the understanding that 
the frequency relations
$|n_1| \ll |n_2|^\theta$ and $ |n_1 + n_2|\sim |n_3|$
are indeed characterized by the use of smooth frequency cutoff functions
as in \eqref{A0a} and \eqref{A0b}.
Moreover, we drop the cutoff function $\ind_{[0 , t]}(t')$ in the following
with the understanding that $0 \leq t' \le t$.

\begin{proof}[Proof of Proposition \ref{PROP:sto4}]

We separately consider the contributions to $\If_{\pl, \pe}$
from $\A^{(1)}_{n, n_1}$ and $\A^{(2)}_{n, n_1}$ defined in \eqref{X9} and denote them respectively
$\If_{\pl, \pe}^{(1)}(w)$  and $\If_{\pl, \pe}^{(2)}(w)$.

Given dyadic $N_2 \geq 1$, 
let $\A_{n, n_1, N_2}^{(1)}(t, t')$
be 
the contribution 
to $\A_{n, n_1}^{(1)} (t, t')$ from $\big\{|n_2| \sim N_2\big\}$.\footnote{More precisely speaking, 
$\A_{n, n_1, N_2}^{(1)}(t, t')$
denotes the contribution in \eqref{A0b}
from $2^k \sim N_2$.}
Fix $0 \leq t \leq T$.
Then, 
from \eqref{X6}  and \eqref{X9}, we have
\begin{align*}
\begin{split}
\|\If_{\pl, \pe}^{(1)}(w)(t) \|_{H^{s_2 - 1}}
& \le \bigg\|\int_0^t 
\jb{n}^{s_2-1}\sum_{n_1 \in \Z^3}
\ft w(n_1, t') \A_{n, n_1}^{(1)}(t, t') dt'
\bigg\|_{\l^2_n}\\
& \les T^\frac{1}{2}  \|w\|_{L^\infty_t L^2_x}
\big\|\jb{n}^{s_2-1}
\A_{n, n_1}^{(1)} (t, t')
\big\|_{ L^2_{t'}([0, T]; \l^2_{n, n_1})}\\
& \les T^\frac{1}{2}  \|w\|_{L^\infty_t L^2_x}\sum_{\substack{N_2\geq 1\\ \text{dyadic}}}
\big\|\jb{n}^{s_2-1}
\A_{n, n_1, N_2}^{(1)} (t, t')
\big\|_{ L^2_{t'}([0, T]; \l^2_{n, n_1})}\\
& \les T^\frac{1}{2}  \|w\|_{L^\infty_t L^2_x} \|\A^{(1)} (t, \cdot)\|_{\A(T)},
\end{split}
\end{align*}

\noi
where we introduced the norm:
\begin{align}
\|\A^{(1)} (t, \cdot)\|_{\A(T)}  := 
\bigg(
\sum_{\substack{N_2\geq 1\\ \text{dyadic}}} N_2^\dl
\big\|\jb{n}^{s_2-1}
\A_{n, n_1, N_2}^{(1)} (t, t')
\big\|^2_{ L^2_{t'}([0, T]; \l^2_{n, n_1})}\bigg)^{\frac12}.
\label{Aop}
\end{align}

\begin{remark}\rm
For fixed $t, t' \in [0, T]$, set  $\mathcal{T}_{t, t'}(f) = 
\sum_{n_1 \in \Z^3}
\ft f(n_1) \A_{n, n_1}^{(1)}(t, t')e_n$.
Then, 
the expression
$\big\|\jb{n}^{s_2-1}
\A_{n, n_1}^{(1)} (t, t')
\big\|_{  \l^2_{n, n_1}}$ is nothing but the Hilbert-Schmidt norm
of the operator $\mathcal{T}_{t, t'}$
from $L^2(\T^3)$ into $H^{s_2-1}(\T^3)$.
Recalling that the Hilbert-Schmidt norm of a given operator
controls its  operator norm, 
it is natural to work with 
 the $\A(T)$-norm 
of $\A^{(1)} (t, \cdot)$ defined above
(which is conveniently modified to carry out analysis on each dyadic
block $\{ |n_2|\sim N_2\}$).

\end{remark}

By a similar argument, 
 we obtain 
\begin{align}
\begin{split}
\|\If_{\pl, \pe}^{(1)}(w)(t_1) - \If_{\pl, \pe}^{(1)}(w)(t_2)\|_{H^{s_2 - 1}}
& \les 
|t_1 - t_2|^\frac{1}{2}  \|w\|_{L^\infty_t L^2_x} \|\A^{(1)} (t_1, \cdot)\|_{\A(T)}\\
& \hphantom{X}
+ T^\frac{1}{2}  \|w\|_{L^\infty_t L^2_x} \|\A^{(1)} (t_1, \cdot)-\A^{(1)} (s_2, \cdot)\|_{\A(T)}
\end{split}
\label{Aop2}
\end{align}

\noi
for $t_1,t_2\in[0,T]$.

We  now  show that the random process $t \mapsto \A^{(1)} (t, \cdot)$ has almost surely 
 continuous trajectories (in $t$) with respect to  the Banach space generated by the norm $\|\cdot\|_{\A(T)}$.  
 In order to do so,  we apply Kolmogorov's continuity criterion 
 and we need, as usual, to evaluate sufficiently high  moments of the random variable $\A^{(1)} (t_1, \cdot)-\A^{(1)} (t_2, \cdot)$ with $t_1,t_2\in[0,T]$.

Note that the conditions $|n_1| \ll |n_2|^\theta$ for some small $\theta > 0$
and $|n_1 + n_2|\sim |n_3|$ imply $|n_2| \sim |n_3| \gg |n_1|$.
Moreover, with the condition $n - n_1 =  n_2 + n_3$, 
we have $|n_2| \sim |n_3| \ges |n|$.
Then, with~\eqref{sigma2}, we have
\begin{align}
 \E\big[\| & \A_{n, n_1, N_2}^{(1)}  (t, t')\|_{L^2_{t'}([0, T])}^2\big]\notag \\
& \leq\bigg\|\sum_{\substack{n - n_1 =  n_2 + n_3\\ |n_1| \ll |n_2|^\theta \\ |n_1 + n_2|\sim |n_3|\\
|n_2| \sim N_2\\n_2+n_3 \ne 0}}
 \frac{|\sin ((t - t') \jb{ n_1 + n_2 })|^2}{\jb{ n_1 + n_2 }^2} 
\E\Big[   | \ft{\<1>}(n_2, t') \, \ft{\<1>}(n_3, t) |^2\Big] \bigg\|_{L^1_{t'}([0, T])}\notag \\ 
& \hphantom{X}
+\bigg\|\sum_{\substack{n_2 \in \Z^3 \\ |n| \ll |n_2|^\theta\\|n_2| \sim N_2}}
  \frac{|\sin ((t - t') \jb{ n + n_2 })|^2}{\jb{ n + n_2 }^2} 
\E\Big[| \ft{\<1>}(n_2, t') \, \ft{\<1>}(-n_2, t) - \s_{n_2}(t, t')|^2\Big]  \bigg\|_{L^1_{t'}([0, T])}\notag \\
& \les T\bigg\{
\sum_{\substack{n - n_1 =  n_2 + n_3\\ |n_1| \ll |n_2|^\theta \\ |n_1 + n_2|\sim |n_3|\\
|n_2|\sim N_2\\n_1\ne n}}
\frac{1}{\jb{n_2}^6}
+ \sum_{\substack{n_2 \in \Z^3 \\ |n| \ll |n_2|^\theta\\|n_2| \sim N_2}}
\frac{1}{\jb{n+ n_2}^2 \jb{n_2}^4}
\bigg\} \notag \\
& \les TN_2^{-3} \ind_{|n_1|\ll N_2^\theta} \ind_{ |n| \les  N_2}.
\label{A1}
\end{align}

\noi
Therefore, we obtain
\begin{align}
\begin{split}
\E\big[\|\A^{(1)} (t, \cdot)\|_{\A(T)}^2\big] 
& = \sum_{\substack{N_2\geq 1\\ \text{dyadic}}} N_2^\dl \sum_{n,n_1}
\jb{n}^{2s_2-2}   \, \E\big[\|  \A_{n, n_1, N_2}^{(1)}  (t, t')\|_{L^2_{t'}([0, T])}^2\big] \\
& \leq \sum_{\substack{N_2\geq 1\\ \text{dyadic}}} 
N_2^{\dl-3} \sum_{n,n_1 \in \Z^3}
\jb{n}^{2s_2-2} \, \ind_{|n_1|\ll N_2^\theta} \ind_{ |n| \les  N_2} \\
& \leq \sum_{\substack{N_2\geq 1\\ \text{dyadic}}} N_2^{\dl+3\theta+2s_2-2} < \infty
\end{split}
\label{A1bis}
\end{align}

\noi
by choosing $\dl = \dl(s_2) > 0$ and $\theta = \ta(s_2)>0$
sufficiently small since $s_2<1$.

From \eqref{X9}, we have
\begin{align}
\begin{split}
\A^{(1)}_{n, n_1}
&  (t_1, t')
 - \A^{(1)}_{n, n_1} (t_2, t')\\
& =\sum_{\substack{n - n_1 =  n_2 + n_3\\ |n_1| \ll |n_2|^\theta \\ |n_1 + n_2|\sim |n_3|}}
 \frac{\sin ((t_1 - t') \jb{ n_1 + n_2 }) - \sin ((t_2 - t') \jb{ n_1 + n_2 })}{\jb{ n_1 + n_2 }} 
B_{n_2, n_3}(t_1, t')
 \\
& \hphantom{X}
+ \sum_{\substack{n - n_1 =  n_2 + n_3\\ |n_1| \ll |n_2|^\theta \\ |n_1 + n_2|\sim |n_3|}}
 \frac{\sin ((t_2 - t') \jb{ n_1 + n_2 })}{\jb{ n_1 + n_2 }} 
\Big( B_{n_2, n_3}(t_1, t') - B_{n_2, n_3}(t_2, t')\Big), 
\end{split}
\label{A2}
\end{align}

\noi
where $B_{n_2, n_3}(t, t') =  \ft{\<1>}(n_2, t') \, \ft{\<1>}(n_3, t) - \ind_{n_2 + n_3 = 0}\cdot \s_{n_2}(t, t')$.
Arguing as in \eqref{A1} and \eqref{A1bis}, we obtain
\begin{align*}
\E\big[\|\A^{(1)} (t_1, \cdot)-\A^{(1)} (t_2, \cdot)\|_{\A(T)}^2\big] \les |t_1-t_2|^{\s},
\end{align*}

\noi
for some small $\s>0$. 
Indeed, the first term on the right-hand side of \eqref{A2}
can be controlled by the mean value theorem, 
creating an additional factor of $\jb{n_1+n_2}^\s |t_1 - t_2|^\s$.
On the other hand, by writing
\begin{align*}
B_{n_2, n_3}(t_1, t')
- B_{n_2, n_3}(t_2, t')
&  =  \ft{\<1>}(n_2, t')  (\ft{\<1>}(n_3, t_1)- \ft{\<1>}(n_3, t_2))\\ 
 & \hphantom{X} 
- \E\Big[  \ft{\<1>}(n_2, t') \, (\ft{\<1>}(-n_2, t_1)- \ft{\<1>}(-n_2, t_2))\Big]\\
 & \hphantom{X} 
 - \ind_{n_2 + n_3 = 0}\big\{ \s_{n_2}(t_1, t') - \s_{n_2}(t_2, t')\big\}, 
\end{align*}

\noi
we can apply \eqref{V0} and the mean value theorem
to create $|t_1 - t_2|^\s$ at the expense of losing a small power
in $n_2$ or $n_3$.

Finally, note that 
 $\A^{(1)}_{n, n_1}$ is a homogenous Wiener chaos of order 2.
Therefore by Minkowski's inequality and the Wiener chaos estimate  (Lemma \ref{LEM:hyp}), we obtain 
 \begin{align*}
\E\big[\|\A^{(1)} (t_1, \cdot)-\A^{(1)} (t_2, \cdot)\|_{\A(T)}^p\big] \les 
p^p |t_1-t_2|^{\s p },
\end{align*}

 \noi
for any $p \geq 2$.
 Finally, by Kolmogorov's continuity criterion, 
 we conclude that 
 \begin{align}
  \|\A^{(1)} (t_1, \cdot)-\A^{(1)} (t_2, \cdot)\|_{\A(T)} \le C(\omega) 
 |t_1-t_2|^{\s-},  
\label{A2tri}
 \end{align}
 
 \noi
 for all $t_1,t_2\in[0,T]$,  where the constant $C= C(\o)$ lies in $L^p(\O)$ for some large $p\gg 1$. 
  From~\eqref{Aop2} 
and \eqref{A2tri}, 
we then deduce the required almost sure continuity for $\If_{\pl, \pe}^{(1)}(w)$.

Next, we consider the contribution from 
$ \A^{(2)}_{n, n_1} (t, t')$ in \eqref{X9}.  This part is entirely deterministic.
Since 
$ \A^{(2)}_{n, n_1} (t, t') = 0$ unless $n = n_1$, 
we only consider $ \A^{(2)}_{n, n} (t, t')$.
In view of \eqref{S6}, we decompose $ \A^{(2)}_{n, n} (t, t')$ as
 \begin{align}
 \A^{(2)}_{n, n} (t, t')
&  =  t'\sum_{\substack{n_2 \in \Z^3 \\ |n| \ll |n_2|^\theta}}
 \frac{\sin ((t - t') \jb{n + n_2 })}{\jb{ n + n_2} } 
   \frac{\cos((t - t')\jb{n_2}) }{2 \jb{n_2}^2}  \notag\\
& \hphantom{X}
+ \sum_{\substack{n_2 \in \Z^3 \\ |n| \ll |n_2|^\theta}}
 \frac{\sin ((t - t') \jb{n + n_2 })}{\jb{ n + n_2} } 
 \cdot    O\bigg(\frac{1}{\jb{n_2}^3 }\bigg).
\notag   \\
& = t'    \sum_{\substack{n_2 \in \Z^3 \\ |n| \ll |n_2|^\theta}}
\frac{\sin ((t - t')( \jb{n + n_2 }+ \jb{n_2}))}{4\jb{ n + n_2} \jb{n_2}^2}  \notag\\
& \hphantom{X}
+ t'    \sum_{\substack{n_2 \in \Z^3 \\ |n| \ll |n_2|^\theta}}
 \frac{\sin ((t - t')( \jb{n + n_2 }- \jb{n_2}))}{4\jb{ n + n_2} \jb{n_2}^2}  \notag\\
& \hphantom{X}
+ \sum_{\substack{n_2 \in \Z^3 \\ |n| \ll |n_2|^\theta}}
 \frac{\sin ((t - t') \jb{n + n_2 })}{\jb{ n + n_2} } 
\cdot    O\bigg(\frac{1}{\jb{n_2}^3 }\bigg)\notag\\
& = : 
 \A^{(3)}_{n} (t, t')
 +  \A^{(4)}_{n} (t, t')
 +  \A^{(5)}_{n} (t, t').
 \label{A4a}
\end{align}

We will show that 
\begin{align}
\| \jb{n}^{-\eps} \A^{(j)}_{n} (t, t') \|_{\l^\infty_n(\Z^3)}
\leq C(T)
\label{A3}
\end{align}

\noi
for any $\eps >0$, $0 \leq t' \leq t \leq T$,  and  $j = 4, 5$.
Then, 
by denoting by $\If_{\pl, \pe}^{(j)}(w)$ the contribution to $\If_{\pl, \pe}(w)$
from $\ind_{n = n_1} \cdot \A_{n}^{(j)}$, 
it follows from \eqref{X6} and~\eqref{A3} that 
\begin{align}
\|\If_{\pl, \pe}^{(j)}(w)(t) \|_{H^{s_2 - 1}}
& \les \bigg\|\int_0^t 
\frac{1}{\jb{n}^{1-s_2}}
\ft w(n, t') \A_{n}^{(j)}(t, t') dt'
\bigg\|_{\l^2_n}\notag\\
& \les T  \|w\|_{L^\infty_t L^2_x}
\bigg\|\frac{1}{\jb{n}^{1-s_2}}
\A_{n}^{(j)} (t, t')
\bigg\|_{L^\infty_{t, t'}([0, T]; \l^\infty_{n})}\notag\\
& \les C(T)  \|w\|_{L^\infty_t L^2_x}
\label{A3a}
\end{align}

\noi
for $t \in [0, T]$ and $j = 4, 5$, provided that $s_2 < 1$.
The continuity in time of $\If_{\pl, \pe}^{(j)}(w)(t)$ follows from a similar argument.

By noting that $\jb{n+n_2} \sim \jb{n_2}\gg \jb{n}$
under $ |n| \ll |n_2|^\theta$, 
we see that \eqref{A3} is easily verified for $j = 5$.
On the other hand, 
  the sum for $\A^{(4)}_{n} (t, t')$ 
is not absolutely convergent.
As in Case 3 in Section \ref{SEC:sto2}, 
we exploit the symmetry $n_2 \leftrightarrow -n_2$
and the oscillatory nature of the sine kernel.
With \eqref{Y14} and the mean value theorem, we have
\begin{align}
\begin{split}
 \A^{(4)}_{n} (t, t')
& = t'    \sum_{\substack{n_2 \in \Ld \\ |n| \ll |n_2|^\theta}}
 \frac{\sin ((t - t')( \jb{n + n_2 }- \jb{n_2}))
 + \sin ((t - t')( \jb{n - n_2 }- \jb{n_2}) )}{4\jb{ n + n_2} \jb{n_2}^2}  \\
& \hphantom{X}
- \sum_{\substack{n_2 \in \Ld \\ |n| \ll |n_2|^\theta}}
 \frac{  \sin ((t - t')( \jb{n - n_2 }- \jb{n_2}) )}{4 \jb{n_2}^2}
 \bigg(\frac{1}{\jb{ n + n_2}} - \frac{1}{\jb{ n - n_2}}\bigg)
 \\
& =    \sum_{\substack{n_2 \in \Ld \\ |n| \ll |n_2|^\theta}}
 \frac{1}{4\jb{ n + n_2} \jb{n_2}^2} 
 \bigg\{
\sin \bigg((t - t')\Big(  \frac{\jb{n, n_2}}{\jb{n_2}} + \Theta^+(n, n_2)\Big)\bigg)\\
& \hphantom{XXXXX}
- \sin \bigg((t - t')\Big(  \frac{\jb{n, n_2}}{\jb{n_2}} - \Theta^-(n, n_2)\Big)\bigg)
 \bigg\} 
\\
& \hphantom{X}
 +   O\bigg(  \sum_{\substack{n_2 \in \Ld \\ |n| \ll |n_2|^\theta}}
 \frac{ \jb{n}}{\jb{n_2}^4}\bigg)\\
 & \les   \sum_{\substack{n_2 \in \Ld \\ |n| \ll |n_2|^\theta}}
 \frac{1}{\jb{ n + n_2} \jb{n_2}^2} 
\frac{\jb{n}^{2\dl}}{\jb{n_2}^\dl} + O(1)\\
& \les \jb{n}^\dl
\end{split}
\label{A4}
\end{align}
for any $\dl \in (0, 1]$.
This proves \eqref{A3} and hence \eqref{A3a}
for $j = 4$.

It remains to consider 
$ \A^{(3)}_{n} (t, t')$.
For this term, there is no internal cancellation structure
and we need to make use of its fast oscillation
by directly studying $\If_{\pl, \pe}^{(3)}(w)$.
From  \eqref{X6} and~\eqref{A4a},  we have
\begin{align}
\begin{split}
\|\If_{\pl, \pe}^{(3)}(w) (t)\|_{H^{s_2 - 1}}
&  \les 
\sum_{\eps_1 \in \{-1, 1\}}
\bigg\|
\frac{e^{i \eps_1 t ( \jb{n + n_2 }+ \jb{n_2})}}{\jb{n}^{1-s_2}} \\
& \hphantom{XX}
\times\sum_{\substack{n_2 \in \Z^3 \\ |n| \ll |n_2|^\theta}}
    \int_0^{t} 
\ft w(n, t')
t'    \frac{e^{ - i \eps_1   t'( \jb{n + n_2 }+ \jb{n_2})}}{ \jb{ n + n_2} \jb{n_2}^2} dt'\bigg\|_{\l^2_n}.
\end{split}
\label{A6}
\end{align}

\noi
Integrating by parts, we have
\begin{align}
\begin{split}
    \int_0^{t} 
& \ft w(n, t') t'    \frac{e^{ - i \eps_1   t'( \jb{n + n_2 }+ \jb{n_2})}}{  \jb{ n + n_2} \jb{n_2}^2} dt'\\
& = 
\frac{1}{ - i \eps_1   ( \jb{n + n_2 }+ \jb{n_2})\jb{ n + n_2} \jb{n_2}^2}\\
& \hphantom{X}
\times \bigg\{
\ft w(n, t) t e^{ - i \eps_1   t( \jb{n + n_2 }+ \jb{n_2})}\\
& \hphantom{XX}
 -     \int_0^{t} 
\big(\ft w(n, t')  + t' \dt\ft w(n, t') \big)   e^{ - i \eps_1   t'( \jb{n + n_2 }+ \jb{n_2})} dt'\bigg\}.
\end{split}
\label{A7}
\end{align}

\noi
Hence, from \eqref{A6} and \eqref{A7}, we obtain
\begin{align*}
\|\If_{\pl, \pe}^{(3)}(w) (t)\|_{H^{s_2 - 1}}
&  \les 
\sum_{n_2 \in \Z^3}
\frac{1}{\jb{n_2}^{4 - (s_2 + \eps) \theta}}
\Big(
\|w\|_{L^\infty_T H^{-1-\eps}_x} 
+ \|\dt w\|_{L^\infty_T H^{-1-\eps}_x} \Big)\\
&  \les 
\|w\|_{L^\infty_T H^{-1-\eps}_x} 
+ \|\dt w\|_{L^\infty_T H^{-1-\eps}_x}
\end{align*}

\noi
for some small $\eps > 0$.
The continuity in time of $\If_{\pl, \pe}^{(3)}(w)(t)$ follows from a similar argument.

This completes the proof of Proposition \ref{PROP:sto4}.
\end{proof}

\begin{remark}\label{REM:sym}
\rm
In handling the term 
$\A^{(4)}_{n} (t, t')$ in \eqref{A4}, 
we exploited the symmetrization  $n_2 \leftrightarrow -n_2$.
This may seem to raise a possible issue in 
treating regularization via  mollification
when a mollification kernel does not satisfy certain symmetry properties
(contrary to our claim in Remark \ref{REM:uniq}).
We point out, however, that this is not the case.

Recall that 
$\A^{(4)}_{n} (t, t')$
is defined in \eqref{A4a}
from 
$ \A^{(2)}_{n, n_1} (t, t')$,
which in turn is defined  in \eqref{X9}.  
Given a mollification kernel $\rho$, 
define $\rho_\dl$ as in \eqref{molli2}.
Then, defining the smoothed stochastic 
convolution $\<1>_\dl = \I(\xi_\dl)$
with $\xi_\dl = \rho_\dl * \xi$, 
we construct 
$ \A^{(2)}_{n, n_1} (t, t')$ associated
with this smoothed stochastic convolution 
$\<1>_\dl$.
In this case, instead of  \eqref{sigma2}, 
we have
\[\s_{n_2}^\dl(t_1, t_2 ) 
: = \E  \big[  \ft{\<1>}_\dl(n_2, t_1)  \,  \ft{\<1>}_\dl(-n_2, t_2) \big]
= \E  \big[  \ft \rho_{\dl}(n_2) \ft{\<1>}(n_2, t_1)  
\ft \rho_{\dl}(-n_2)
  \ft{\<1>}(-n_2, t_2) \big].
\]

\noi
Note that the effect of mollification 
$ \ft \eta_{\dl}(n_2)
\ft \eta_{\dl}(-n_2)$
is symmetric in  $n_2 \leftrightarrow -n_2$.
This observation allows us to carry out the symmetrization argument
in \eqref{A4} 
(and also the 
 symmetrization argument
in \eqref{Y14a} for the construction of 
$\<21p>$)
even for general (non-symmetric) mollification.

\end{remark}

\begin{remark}\label{REM:MAX}\rm
  We can also accommodate a space-time regularization of the noise in the form 
  of a smoothing kernel $\rho(x, t)$. 
In this case,  we need to impose an additional assumption
that a space-time mollification kernel $\rho(x, t)$ is even in $x$, 
namely, $\rho(-x, t) = \rho(x, t)$ for any $t \in \R$, implying that 
\begin{align}
\ft \rho(-n, t) = \ft \rho(n, t)
\label{P1}
\end{align}

\noi
for any $n \in \Z^3$ and $t \in \R$.

This is not
directly apparent from the computations above, so let us give some
indications of the argument. 
In order to treat a space-time mollification, 
 it is more
convenient to switch to a representation of the stochastic objects based
directly on the white noise. We write the stochastic convolution 
$\<1>_\dl = \I(\xi_\dl)
= \I(\rho_\dl * \xi)$
as
\begin{align}
 \ft{\<1>}_\dl (n, t) 
:= \int_\R \bigg[ \int_0^t 
\sum_{\eps \in \{ -1,  1 \}}
 \frac{\eps}{2 i} \frac{e^{i \eps \jb{n} (t - t')}}{\jb{n}} 
 \ft{\rho}_\dl (n, t' - \tau)  dt'
\bigg] d \beta_n (\tau),  
\label{P2}
\end{align}

\noi
  where $\ft{\rho}_\dl (n, t)$ is the spatial Fourier transform of the 
  space-time mollifier  $\rho_\dl$:
\[ \rho_\dl(x, t) = \dl^{-4} \rho(\dl^{-1} x, \dl^{-1}t).\]
  
\noi
The renormalizations and computations of the
various stochastic objects then proceed in a similar way as above. 
For example,  we have
\begin{align*}
  \ft{\<20>}_\dl  (n, t) 
  :\! & = \int_0^t 
\sum_{\eps_0 \in \{-1, 1\}}
\frac{\varepsilon_0}{2 i} 
\frac{e^{i \eps_0 \jb{n}(t - t')}}{\jb{n}}\ft{\<2>}_\dl  (n, t') dt'\\
& = 2 \int_{\R^2} \sum_{n_1, n_2 \in \Z^3} Q^\dl_{n, n_1, n_2} (\tau_1, \tau_2)d \beta_{n_1} (\tau_1)d
\beta_{n_2} (\tau_2) , 
\end{align*}

\noi
where 
\begin{align*}
Q_{n, n_1, n_2}^\dl (\tau_1, \tau_2) 
& = \int_{0 \leq t_1 \leq t_2 \leq t}
\bigg( \int^t_{\tau_2}\sum_{\eps_0, \eps_1, \eps_2 \in \{-1, 1\}}
 \frac{\eps_0
\eps_1 \eps_2}{(2 i)^3} \frac{e^{i \eps_0 \jb{n} (t - t') 
+ i \sum_{j = 1}^2 \eps_j \jb{ n_j } (t' - t_j )}}
{\jb{n} \jb{n_1}\jb{n_2}} dt' \bigg)\\
& 
\hphantom{X}
\times
\ft{\rho}_\dl (n_1, t_1 - \tau_1) \ft{\rho}_\dl (n_2, t_2 - \tau_2)d t_1 d
t_2. 
\end{align*} 

\noi
Note that the double Wiener-It\^o integral
accounts for  the Wick renormalization 
on $\ft{\<2>}_\dl $.
Hence, we have 
\[ \E\big[|  \ft{\<20>}_\dl  (n, t) |^2\big] 
\sim \sum_{n_1,n_2\in \Z^3 } 
\int_{\R^2}
 | Q^\dl_{n, n_1, n_2} (\tau_1, \tau_2) |^2 d \tau_1 d \tau_2. \]

\noi
By Young's inequality in  the two convolutions in time, 
we obtain
\begin{align*}
 \E\big[|  \ft{\<20>}_\dl  (n, t) |^2\big]  
 \les\sum_{n_1, n_2\in \Z^3} 
\int_{0 \leq t_1 \leq t_2 \leq t} 
\bigg| &  \int^t_{t_2}  \sum_{\eps_0, \eps_1, \eps_2\in \{-1, 1\}}
\eps_0 \eps_1 \eps_2 \\
& 
\times \frac{e^{i \eps_0 
\jb{n} (t - t') + i \sum_{j = 1}^2 \eps_j \jb{n_j }
(t' - t_j)}}{\jb{n}\jb{n_1}\jb{n_2}}dt' \bigg|^2 dt_1 d t_2 
\end{align*}

\noi
uniformly in $0 < \dl < 1$
for the space-time mollifier  $\rho_\dl$. 
After integrating in $t'$, 
the above expression essentially reduces to   
  \eqref{SS1} in the proof of Proposition \ref{PROP:sto1}
and hence
the rest follows as before. 
A more refined argument yields 
 convergence in $L^p(\O)$ as the
regularization is removed.

Recall that 
one of the key ingredients in studying 
the resonant product 
 $\<21p>$
and the paracontrolled operator  $\If_{\pl, \pe}$ 
is
 the symmetrization argument
 $n_2 \leftrightarrow -n_2$, 
  appearing 
in
\eqref{Y14a} 
and 
 \eqref{A4}.
It is at this point that we need to make use of the symmetry assumption
\eqref{P1}.
As in Remark \ref{REM:sym}, 
it suffices to 
consider
\begin{align}
\begin{split}
\s_{n_2}^\dl(t_1, t_2 ) 
: & = \E  \big[  \ft{\<1>}_\dl(n_2, t_1)  \,  \ft{\<1>}_\dl(-n_2, t_2) \big]\\
& =   \E  \big[  
 \big( \rho_\dl (n_2, \cdot\,)*
  \ft{\<1>}(n_2, \cdot\, )\big)(t_1)  
\big(  \rho_\dl (- n_2, \cdot\,)*
    \ft{\<1>}_\dl(-n_2, \cdot\,)\big)(t_2) \big].
\end{split}
\label{P3}
\end{align}

\noi
Using \eqref{P2}, we have
\begin{align*}
\s_{n_2}^\dl(t_1, t_2 ) 
&  = - 
 \int_0^{t_1} 
 \int_0^{t_2}
\sum_{\eps_1, \eps_2 \in \{ -1,  1 \}}
 \frac{\eps_1\eps_2}{4} 
 \frac{e^{i \eps_1 \jb{n_2} (t_1 - t'_1)
 + i \eps_2 \jb{n_2} (t_2 - t'_2)
 }}{\jb{n_2}^2} \\
& \hphantom{XXXXXXXXX}
\times \bigg( \int_\R 
\ft{\rho}_\dl (n_2, t'_1 - \tau)  
  \ft{\rho}_\dl (- n_2, t'_2 - \tau)  
d \tau  
  \bigg)
  dt'_1
 dt'_2\\
 & = 
\s_{- n_2}^\dl(t_1, t_2 ),   
\end{align*}

\noi
where we used the symmetry assumption \eqref{P1}
in the last step.
This shows that 
 the effect of space-time mollifications in \eqref{P3} 
is symmetric in  $n_2 \leftrightarrow -n_2$.
This observation allows us to carry out the symmetrization argument
in \eqref{Y14a} and \eqref{A4}, provided  that 
 a space-time mollification kernel $\rho(x, t)$ is even in $x$.

%
%
\end{remark}

\section{Proof of Theorem \ref{THM:1}}
\label{SEC:proof1}

We present now the proof of Theorem \ref{THM:1}.
In the following, we assume that $0 < s_1< s_2 < 1$.
Recall that 
$\big(8, \frac{8}{3}\big)$ 
and $(4, 4)$ are $\frac 14$-admissible
and  $\frac 12$-admissible, respectively.
Given $0 < T \leq 1$, 
we define $X^{s_1}_T$
(and $Y^{s_2}_T$) as the intersection of the energy spaces of regularity  $s_1$ 
(and $s_2$, respectively) and the Strichartz space:
\begin{align}
\begin{split}
 X^{s_1}_T 
 & = C([0,T];H^{s_1}(\T^3))\cap C^1([0,T]; H^{s_1-1}(\T^3))
 \cap L^8([0, T]; W^{s_1-\frac 14, \frac{8}{3}}(\T^3)),\\
 Y^{s_2}_T 
 & = C([0,T];H^{s_2}(\T^3))\cap C^1([0,T]; H^{s_2-1}(\T^3))
 \cap L^4([0, T]; W^{s_2 - \frac 12, 4}(\T^3))
\end{split}
 \label{M0}
\end{align}

\noi
and set 
\begin{align*}
Z^{s_1, s_2}_T =  X^{s_1}_T\times  Y^{s_2}_T.
\end{align*}

Let $\Phi = (\Phi_1, \Phi_2)$ be as in \eqref{SNLW8}
with the enhanced data set $\Xi$ in \eqref{data1}
belonging to $\mathcal{X}^{s_1, s_2, \eps}_T$
for some small $\eps = \eps(s_1, s_2)> 0$.
By the Strichartz estimates (Lemma~\ref{LEM:Str}), 
 the paraproduct estimate (Lemma \ref{LEM:para}), 
 and the regularity assumptions on $\<1>$ and $\<20>$, we have 
\begin{align}
\| \Phi_1(X, Y) \|_{X^{s_1}_T}
& \les \| (X_0, X_1) \|_{\H^{s_1}}
+  \|(X+Y-\<20>)\pl \<1>\|_{L^1_T H^{s_1 - 1}}\notag\\
& \les \| (X_0, X_1) \|_{\H^{s_1}}
+  T \|X+Y-\<20>\|_{L^\infty_{T}L^2_x} \| \<1>\|_{L^\infty_T W_x^{-\frac 12 -\eps, \infty}}\notag \\
& \les \| (X_0, X_1) \|_{\H^{s_1}}
+  T \big( 1 + \|(X, Y)\|_{Z^{s_1, s_2}_T} \big)
\label{M1}
\end{align}

\noi
provided that $s_1 - 1 < -\frac 12-\eps$, namely $s_1 < \frac 12$.
Similarly, 
by Lemmas~\ref{LEM:Str} and \ref{LEM:para}
with the regularity assumption on the enhanced data set
 $\Xi$ in \eqref{data1} and Corollary~\ref{COR:sto3a}, 
we have
\begin{align}
\bigg\| S(t)(Y_0 & , Y_1)  -  \int_0^t   \frac{\sin((t-t')\jb{\nb})}{\jb{\nb} }
\Big[ 2(X+Y-\<20>)\pg \<1> + 2  Y\pe \<1> - 2 \<21p> \notag\\
& \hphantom{XXXXXXXX}
 + 2Z
-4 \If_{\pl}^{(1)}(X + Y - \<20>)\pe \<1>\Big](t') dt' \bigg\|_{Y^{s_2}_T}\notag\\
& \les 
\| (Y_0, Y_1) \|_{\H^{s_2}}
+ 
 \|(X+Y-\<20>)\pg \<1>\|_{L^1_T H^{s_2 -1 }_x} +  
\| Y\pe \<1>\|_{L^1_T H^{s_2 -1 }_x} 
+ \| \<21p>\|_{L^1_T H^{s_2 -1 }_x} \notag\\ 
& \hphantom{XXXXXXXX}
+ \|Z\|_{L^1_T H^{s_2-1}_z}
+ \| \If_{\pl}^{(1)}(X + Y - \<20>)\pe \<1>\|_{L^1_T H^{s_2 -1 }_x}
\notag\\
& \les 
 \| (Y_0, Y_1) \|_{\H^{s_2}}
+  T\big( 1 + \|X+Y-\<20>\|_{L^\infty_T H^{s_1 }_x} +  
\| Y\|_{L^\infty_T H^{s_2  }_x} \big) \notag\\ 
& \les 
 \| (Y_0, Y_1) \|_{\H^{s_2}}
+ T \big( 1 + \|(X, Y)\|_{Z^{s_1, s_2}_T} \big), 
\label{M2}
\end{align}

\noi
provided that
$s_2 - 1 < \min(s_1 - \frac 12-2\eps, -\eps)$ and $s_2 + (-\frac 12-\eps) > 0$,
namely, 
\begin{align*}
\tfrac 12 < s_2 < \min\big(1, s_1 +\tfrac 12\big).
\end{align*}

\noi
Similarly, we have
\begin{align}
\bigg\|  \int_0^t   \frac{\sin((t-t')\jb{\nb})}{\jb{\nb}}    \If_{\pl, \pe}(X + & Y - \<20>)(t') dt'
 \bigg\|_{Y^{s_2}_T}
   \les  \|\If_{\pl, \pe}(X + Y - \<20>)\|_{L^1_T H^{s_2 -1 }_x}\notag\\
 &  \les T  \|X + Y - \<20>\|_{L^\infty_T L^2_x\cap C^1_TH^{-1-\eps}_x}\notag\\
& \les  T \big( 1 + \|(X, Y)\|_{Z^{s_1, s_2}_T} \big), 
\label{M3} 
\end{align}

\noi
provided that $s_2 < 1$.
Lastly, by Lemma \ref{LEM:Str}
with the fractional Leibniz rule (Lemma \ref{LEM:bilin}), we have
\begin{align}
\bigg\| \int_0^t &  \frac{\sin((t-t')\jb{\nb})}{\jb{\nb}} 
 (X+Y-\<20>)^2(t') dt'\bigg\|_{Y^{s_2}_T}
 \les \| \jb{\nb}^{s_2 - \frac 12}(X+Y-\<20>)^2\|_{L^\frac{4}{3}_{T, x}}\notag\\
& \les T^\frac{1}{4} 
\Big(\| \jb{\nb}^{s_2 - \frac 12}X\|_{L^8_T L^\frac{8}{3}_x}^2
+ \| \jb{\nb}^{s_2 - \frac 12}Y\|_{L^4_{T, x}}^2
+ \| \jb{\nb}^{s_2 - \frac 12}\<20>\|_{L^\infty_{T, x}}^2\Big)\notag\\
& \les T^\frac{1}{4} \big( 1 + \|(X, Y)\|_{Z^{s_1, s_2}_T}^2 \big), 
\label{M4}
\end{align}

\noi
provided that $s_2 \leq \min(1 -\eps, s_1 + \frac 14)$.

From \eqref{M1}, \eqref{M2}, \eqref{M3}, and \eqref{M4}, we obtain 
\begin{align}
\| \Phi(X, Y) \|_{Z^{s_1, s_2}_T}
\les 
\| (X_0, X_1) \|_{\H^{s_1}}
+ 
\| (Y_0, Y_1) \|_{\H^{s_2}}
+ 
T^\ta \big( 1 + \|(X, Y)\|_{Z^{s_1, s_2}_T}^2\big)
\label{M5}
\end{align}

\noi
for some $\ta > 0$.
By repeating a similar computation, 
we also obtain the following estimate on the difference: 
\begin{align}
\| \Phi(X, Y) & - \Phi(\wt X, \wt Y)\|_{Z^{s_1, s_2}_T}\notag\\
& \les 
T^\ta \big( 1 + \|(X, Y)\|_{Z^{s_1, s_2}_T} + \|(\wt X, \wt Y)\|_{Z^{s_1, s_2}_T} \big) 
 \|(X, Y) - (\wt X, \wt Y) \|_{Z^{s_1, s_2}_T}.
\label{M6}
\end{align}

\noi
Therefore, by choosing $T>0$ sufficiently small
(depending on the $\mathcal{X}^{s_1, s_2,\eps}_1$-norm
of the  enhanced data set $\Xi$), 
we conclude from \eqref{M5} and \eqref{M6} that $\Phi$ in \eqref{SNLW8}
is a contraction on the ball $B_R \subset Z^{s_1, s_2}_T$
of radius $R\sim 
\| (X_0, X_1) \|_{\H^{s_1}}
+ \| (Y_0, Y_1) \|_{\H^{s_2}}$.
A similar computation yields 
continuous dependence of the solution $(X, Y)$
on the enhanced data set $\Xi$
measured in the $\mathcal{X}^{s_1, s_2,\eps}_1$-norm.
This concludes the proof of 
Theorem \ref{THM:1}.

\section{On the weak universality of the renormalized SNLW}

We conclude this paper by presenting the proof of
Theorem~\ref{THM:weak}. 
For the sake of concreteness, 
we take the Gaussian noise $\eta_\kk$ 
to be the mollified space-time white noise $\rho * \xi$ on $(\kk^{-1} \T)^3 \times \R_+$
given by 
\begin{align}
\eta_\kk  
=  \kk^{\frac{3}{2}} \sum_{n\in \Z^3}  \ft \rho (\kk n)  \frac{d\beta_n}{dt} e_{\kk n},
\label{eta-kappa}
\end{align}

\noi
where $\rho$ is a (smooth) 
mollification kernel with  support in $\T^3 \cong \big[-\frac 12,\frac 12\big)^3$, 
 $\{\beta_n\}_{n \in \Ld_0}$
 is a family of mutually independent complex-valued
Brownian motions, and   $\beta_{-n} := \cj{\be_n}$, $n \in \Ld_0$,  as in~\eqref{Wiener1}. 
It is not difficult to see that $\eta_\kk$ is indeed a random field on 
$(\kk^{-1} \T)^3\times \R_+$ which is smooth in space and white in time with stationary correlations.

Our aim is to describe the long time and large space behavior of the solution $w_\kk$
to \eqref{weq}.
In order to do so, we perform a change of variables 
$u_{\kk} (x, t) = \kk^{- 2} w_{\kk} (\kk^{-1} x, \kk^{-1} t)$ as in \eqref{scale1}.
Then, the equation \eqref{weq}  takes the form:
\begin{align}
\dt^2 u_{\kk} + (1 - \Dl) u_{\kk}
 = \kk^{- 4}  f (\kk^{2} u_{\kk} ) + \kk^{- 4}
   a_{\kk}^{(0)} + \kk^{- 2} a_{\kk}^{(1)} u_{\kk}
   + (1-\kk^{-2})u_\kk
    +  \xi_{\kk} 
\label{WU1}
\end{align}

\noi
on $\T^3\times \R_+$.
Here, 
 $\xi_{\kk} (x, t) = \kk^{-2} \eta_{\kk} (\kk^{-1} x, \kk^{-1} t)$ is chosen so that 
$\xi_\kk$ converges in law to 
the space-time white noise $\xi$ on $\T^3\times \R_+$ as $\kk \to 0$.
Indeed, from \eqref{eta-kappa}, 
 we deduce that
\begin{align}
\label{xi-kappa}
\xi_\kk =  \sum_{n\in \Z^3}  \ft {\rho}(\kk n)  \frac{d\wt \be _n}{dt} e_{n}, 
\end{align}

\noi
where $\{\wt \be_n\}_{n \in \Ld_0}$ is a family of 
mutually independent complex Brownian motions 
with the same joint law as $\{\beta_n\}_{n\in \Ld_0}$
and $\wt \beta_{-n} := \cj{\wt \be_n}$, $n \in \Ld_0$. 
By taking 
\begin{align*}
\xi =  \sum_{n\in \Z^3}    \frac{d\wt \be _n}{dt} e_{n}
\end{align*}

\noi
as a realization of the space-time white noise $\xi$, 
we see that $\xi_\kk$ converges to $\xi$
in $C^{-\frac 12-\eps}(\R_+; W^{-\frac 32- \eps, \infty}(\T^3))$
(endowed with the compact-open topology)
almost surely for any $\eps > 0$.

By the Taylor remainder theorem, 
we can write the right-hand side of \eqref{WU1} (excluding $\xi_\kk$)
as 
\begin{align*}
\begin{split}
 \kk^{- 4} f (\kk^{2} u_{\kk}) & + \kk^{- 4} a_{\kk}^{(0)} + \kk^{- 2} a_{\kk}^{(1)} u_{\kk} 
 +  (1-\kk^{-2})u_\kk\\
 & = \big\{\kk^{
    - 4} f (0) + \kk^{ - 4} a_{\kk}^{(0)}\big\} 
 + \big\{\kk^{- 2} f' (0) + \kk^{- 2} a_{\kk}^{(1)}+  (1-\kk^{-2}) \big\} u_{\kk}\\
 & \hphantom{X}  +
   \frac{f'' (0)}{2}   u_{\kk}^2 + R_{\kk},   
\end{split}
\end{align*}

\noi
where $R_{\kk}$ is the remainder given by 
\begin{align}
 R_{\kk} = \kk^2 u_{\kk}^3 \int_0^1 \frac{ f'''
   (\tau \kk^2 u_{\kk})}{6}  (1 - \tau)^2  d\tau . 
\label{WU3}
\end{align}

Let  $\<1>_{\kk}$ be the solution of the linear equation:
\begin{align}
 (\partial_t^2  + 1 - \Delta) \<1>_\kk = \xi_{\kk}. 
\label{WU3a}
 \end{align}

\noi
Then, with  $b_\kk(t) = \E\big[(\<1>_\kk(t))^2\big]$, 
we define the Wick power $\<2>_\kk$ by 
\begin{align}
 \<2>_\kk = (\<1>_\kk)^2 - b_{\kk}.
\label{WU4}
\end{align}

\noi
We now choose the time-dependent parameters $a_\kk^{(0)}$ and $a_\kk^{(1)}$ 
by setting
\begin{align}
\label{WU5}
 a^{(0)}_{\kk} = - f (0) -  \kk^4 c_f b_{\kk}
 \qquad\text{and}\qquad
a^{(1)}_{\kk} = - f' (0) + (1 - \kk^2), 
\end{align}

\noi
where $c_f =  \frac{f'' (0)}{2}$.
Then, by writing
\begin{align*}
u_{\kk} =  \<1>_\kk -  w_{\kk}, 
\end{align*}

\noi
we see from \eqref{WU1}, \eqref{WU3}, and \eqref{WU5} that $v_\kk$ satisfies
\begin{align}
 \partial_t^2 w_{\kk}  + (1 - \Dl) w_{\kk} 
= c_f \<2>_\kk + 2c_f  \<1>_{\kk}  w_{\kk} +
 c_f  w_{\kk}^2 + R_{\kk}, 
\label{WU6}
\end{align}

\noi
where we used \eqref{WU4}
to replace $(\<1>_\kk)^2 - b_{\kk}$ by $ \<2>_\kk$.
In the following, by scaling, 
we assume that 
\[c_f = -1.\]

Letting $\<20>_\kk = (\dt^2 + 1 - \Dl)^{-1} (\<2>_\kk)$, 
we decompose $w_\kk$ as 
\begin{align*}
w_\kk = - \<20>_\kk + X_\kk + Y_\kk
\end{align*}

\noi
as in Section \ref{SEC:1}.
Then, by repeating the discussion in Section \ref{SEC:1}, 
we can rewrite the equation~\eqref{WU6} for $w_\kk$
into the following system for $X_\kk$ and $Y_\kk$:
\begin{align}
\begin{split}
 X_\kk (t)  &  =  -2 \int_0^t \frac{\sin((t-t')\jb{\nb})}{\jb{\nb}} 
 \big[(X_\kk+ Y_\kk - \<20>_\kk)\pl \<1>_\kk\big](t')dt',  \\
Y_\kk  (t) &   = 
- 
 \int_0^t \frac{\sin((t-t')\jb{\nb})}{\jb{\nb}} 
\Big[(X_\kk+Y_\kk-\<20>_\kk)^2 +  2(X_\kk + Y_\kk - \<20>_\kk)\pg \<1>_\kk\\
& \hspace{4em} 
+ 2 Y_\kk \pe \<1>_\kk - 2 \<21p>_\kk
  -R_\kk \\
& \hspace{4em}
 -4 \If_{\pl}^{(1)}(X_\kk+  Y_\kk - \<20>_\kk)\pe \<1> 
 -4\If_{\pl, \pe}^\kk(X_\kk Y_\kk - \<20>_\kk)\Big] (t') dt', 
\end{split}
 \label{SNLW8-wu}
\end{align}

\noi
where $\<21p>_\kk = \<20>_\kk\pe \<1>$
and $ \If_{\pl, \pe}^\kk$
is defined as in \eqref{X6} with $\<1>$ replaced by $\<1>_\kk$.

Let $\frac 14 < s_1 < \frac 12 < s_2 \leq s_1 + \frac 14$.
Note that the rescaled noise $\xi_\kk$ in \eqref{xi-kappa}  is basically the mollified
space-time white noise.
Hence, it is easy to see that 
the enhanced data set associated with the rescaled noise $\xi_\kk$:
\begin{align}
\Xi_\kk
= \big(0, 0, 0, 0, \<1>_\kk, \<20>_\kk,  \<21p>_\kk, 
0, 
 \If_{\pl, \pe}^\kk\big)
\label{so6-wu}
\end{align}

\noi
belongs to the class $\mathcal{X}^{s_1, s_2, \eps}_1$ defined in \eqref{L3}
since 
 $\<1>_\kk$, $\<20>_\kk$,  $\<21p>_\kk$,
 and  
$ \If_{\pl, \pe}^\kk$
satisfy the statement analogous
to Lemma \ref{LEM:stoconv}
and  Propositions \ref{PROP:sto1}, \ref{PROP:sto2},  and \ref{PROP:sto4}.

Note that the system \eqref{SNLW8-wu}
 is analogous to the original system~\eqref{SNLW8}
with the enhanced data set $\Xi$ replaced by $\Xi_\kk$ and  an additional source term given by the remainder
term  $R_\kk$. 
In the following, we  proceed as in the proof of Theorem \ref{THM:1}
and prove 
local well-posedness of the system \eqref{SNLW8-wu}
for  $\kk > 0$ on a time interval $[0, T]$,
where  $T = T(\o)$
is an almost surely positive stopping time, 
independent of $\kk > 0$.
Under the assumption $\|f'''\|_{L^\infty} < \infty$, 
we have  $R_\kk = O(\kk^2 u_\kk^3)$, 
where 
\begin{align*}
 u_\kk =  \<1>_\kk - \<20>_\kk + X_\kk + Y_\kk.
 \end{align*}

\noi
In order to handle the cubic structure of $R_\kk$, 
we need to modify the norm for the second component $Y_\kk$.
Let  $s_2 = \frac 12 + \s$
with some small $\s > 0$.
Noting that 
$\big(\frac{4}{1+2\s}, \frac{4}{1-2\s}\big)$
is $s_2$-admissible,
we define the $\wt Y^{s_2}_T$-space
by the norm:
\begin{align*}
\begin{split}
 \wt Y^{s_2}_T 
 & = C([0,T];H^{s_2}(\T^3))\cap C^1([0,T]; H^{s_2-1}(\T^3))
 \cap L^{\frac{4}{1+2\s}}([0, T]; L^{\frac{4}{1-2\s}}(\T^3))
\end{split}
\end{align*}

\noi
and set 
\begin{align*}
\wt Z^{s_1, s_2}_T =  X^{s_1}_T\times  \wt Y^{s_2}_T,
\end{align*}

\noi
where $X^{s_1}_T$ is as in \eqref{M0}.
In the following, we use the fact that 
$\big(\frac{4}{3+8\s}, \frac{4}{3-4\s}\big)$
is dual $s_2$-admissible.

Note that \eqref{M1}, \eqref{M2}, and \eqref{M3}
hold 
true even after replacing 
$Z^{s_1, s_2}_T$ and  $Y^{s_2}_T$
by $\wt Z^{s_1, s_2}_T$ and $\wt Y^{s_2}_T$, respectively.
Instead of \eqref{M4}, 
from the Strichartz estimates (Lemma \ref{LEM:Str})
and Sobolev's inequality on $X$, 
we have
\begin{align}
\bigg\| \int_0^t &  \frac{\sin((t-t')\jb{\nb})}{\jb{\nb}} 
 (X_\kk+Y_\kk-\<20>_\kk)^2(t') dt'\bigg\|_{\wt Y^{s_2}_T}
 \les \| (X_\kk+Y_\kk-\<20>_\kk)^2\|_{L^\frac{4}{3+8\s}_{T} L^\frac{4}{3-4\s}_x}\notag\\
& \les T^\ta 
\Big(\| X_\kk\|_{L^8_T L^\frac{8}{3-4\s}_x}^2
+ \| Y_\kk\|_{L^\frac{4}{1+2\s}_T L^\frac{4}{1-2\s}_x}^2
+ \| \<20>_\kk\|_{L^\infty_{T, x}}^2\Big)\notag\\
& \le C(\o) T^\ta \big( 1 + \|(X_\kk, Y_\kk)\|_{\wt Z^{s_1, s_2}_T}^2 \big)
\label{WU11}
\end{align}

\noi
for some $\ta > 0$, provided that $s_1 - \frac 1 4  \geq \frac{3\s}{2}$,
 which allows us to apply   Sobolev's inequality:
$
\|  X_\kk\|_{L^8_T L^\frac{8}{3-4\s}_x}  \les
\|  X_\kk\|_{L^8_T W^{s_1-\frac14, \frac83 }_x}
$.
Given $s_1 > \frac 14$, this condition can be satisfied by choosing $\s > 0$ sufficiently small.

Next, we estimate the contribution from  the remainder term $R_\kk$.
From \eqref{xi-kappa} and \eqref{WU3a}, 
we see that
 $\ft{\<1>}_\kk(n, t)$  is essentially supported on the spatial frequencies $\{|n|\les \kk^{-1}\}$.
Hence, we have 
$ \kk^{\frac{1}{2}+\eps}\<1>_\kk\in L^\infty([0, T]; L^\infty(\T^3))$
almost surely for any $\eps > 0$.
By a similar reasoning, 
the paracontrolled structure of the $X_\kk$-equation in \eqref{SNLW8-wu}
allows us to conclude that $X_\kk$
essentially has the spatial frequency support on $\{|n|\les \kk^{-1}\}$.
Therefore, 
from Lemma \ref{LEM:Str}, \eqref{WU3} with $\|f'''\|_{L^\infty} < \infty$, 
and Sobolev's inequality, 
we have
\begin{align}
\begin{split}
\bigg\| \int_0^t 
& \frac{\sin((t-t')\jb{\nb})}{\jb{\nb}} 
R_\kk(t') dt'
\bigg\|_{\wt Y^{s_2}_T}
  \les \kk^2\| ( \<1>_\kk - \<20>_\kk + X_\kk + Y_\kk)^3\|_{L^\frac{4}{3+8\s}_{T} L^\frac{4}{3-4\s}_x}\\
& 
\les 
\kk^{\frac{1}{2}-3\eps}T \Big( \| \kk^{\frac{1}{2}+\eps}\<1>_\kk\|^3_{L^\infty_{T, x}}
+ \| \<20>_\kk\|^3_{L^\infty_{T, x}} \Big)\\
& \hphantom{X}
+  \kk^{3\dl} T^\ta\Big( \| \kk^{\frac{2}{3}-\dl} X_\kk\|^3_{L^8_T L^\frac{12}{3-4\s}_x}
+  \| Y_\kk\|^3_{L^\frac{12}{3+8\s}_T L^\frac{12}{3-4\s}_x}\Big)\\
& \le C(\o) \, \kk^{\dl} T^\ta \big( 1 + \|(X_\kk, Y_\kk)\|_{\wt Z^{s_1, s_2}_T}^3 \big)
\end{split}
\label{WU12}
\end{align}

\noi
for some $\dl, \ta >0$. Here we used the frequency  support of $ X_\kk$ and
 Sobolev's inequality to bound
\[
\| \kk^{\frac{2}{3}-\dl} X_\kk\|_{L^8_T L^\frac{12}{3-4\s}_x}  \lesssim
\|  X_\kk\|_{L^8_T W^{- \frac23+\delta,\frac{12}{3-4\s}}_x} \lesssim
\|  X_\kk\|_{L^8_T W^{s_1-\frac14, \frac83 }_x}
\]
which holds when
 $3(\frac{3}{8} - \frac{3-4\s}{12}) 
 = \frac38 + \s \leq  (s_1 - \frac 14 ) + (\frac 23 - \dl) = s_1 + \frac{5}{12}-\dl$. This last condition is guaranteed by choosing $\s, \dl > 0$ sufficiently small. 
 Note that we  used the following bound:
 $ \| Y_\kk\|_{L^\frac{4}{1+8\s/3}_T L^\frac{4}{1-4\s/3}_x} 
 \les T^\frac{\s}{6} \| Y_\kk\|_{L^\frac{4}{1+2\s}_T L^\frac{4}{1-2\s}_x}$.
Lastly, we point out that 
it was important to use $s_2$-admissible
and dual $s_2$-admissible pairs
such that there is no derivative on $\R_\kk$
after applying the Strichartz estimate in~\eqref{WU12}.
Otherwise, a (fractional) derivative would fall on $f'''(\tau \kk^2 u_\kk)$ in~\eqref{WU3}
and we would need to use the fractional chain rule,
which would make the computation far more complicated.

Putting \eqref{M1}, \eqref{M2},  \eqref{M3}, 
\eqref{WU11}, and \eqref{WU12} together, 
we conclude that the system \eqref{SNLW8-wu}
is locally well-posed on $[0, T]$, 
where  $T = T(\o)$
is an  almost surely positive stopping time, 
independent of $\kk > 0$.

 As for the sequence $\{\Xi_N\}_{N\in \N}$ above, 
 one can show that, at least along subsequences, 
 the family $\{\Xi_\kk\}_{\kk\in(0,1)}$ in \eqref{so6-wu}converges (in the natural 
 $\mathcal{X}^{s_1, s_2,  \eps}_1$-topology)  almost surely 
 towards the random vector $\Xi$ given by~\eqref{so7} with $(u_0,u_1)=(0,0)$.
Let $(X, Y)$ be the solution to the original system \eqref{SNLW8}
with this random data  $\Xi$
and set $u$ by \eqref{so8}.
Then, by using the above estimates, 
we can estimate the difference $(X - X_\kk, Y - Y_\kk)$.
As a consequence, 
we conclude that that, along any countable sequence, 
 $u_\kk$ converges to 
to the same limit $u$ 
in $C([0,T];H^{-\frac12 -\eps}(\T^3))$
almost surely (and hence in probability), 
where $T = T(\o)$ is a random local existence time
 whose size depends only on the random data $\Xi$, 
in particular, independent of $\kk \to 0$.
Since the limit $u$ does not depend on a particular countable sequence of $\kk \to 0$, 
we can deduce that the whole family $\{u_\kk\}_{\kk \in (0, 1)}$ converges in probability towards $u$.
This completes the proof of Theorem~\ref{THM:weak}.

\begin{ackno}\rm
M.G. and H.K. were supported by the Deutsche Forschungsgemeinschaft (DFG, German Research Foundation) through the Hausdorff Center for Mathematics under Germany's Excellence Strategy - EXC-2047/1 - 390685813 and through CRC 1060 - project number 211504053.
%
T.O.~was supported by the European Research Council (grant no.~637995 ``ProbDynDispEq'' and grant no.~864138 ``SingStochDispDyn'').
M.G.~would like to thank the Isaac Newton Institute for Mathematical Sciences for support and hospitality during the programme Scaling limits, rough paths, quantum field theory when work on this paper was undertaken.  This work was supported by EPSRC Grant Number EP/R014604/1.
The authors are grateful to the anonymous referee for the helpful comments that have improved the presentation of this paper.

\end{ackno}

\end{document}